\documentclass[a4paper,reqno]{amsart}

\usepackage{amsfonts}
\usepackage{amsmath}
\usepackage[colorlinks, citecolor={red}, linkcolor=.]{hyperref}
\usepackage{amssymb}
\usepackage{tikz}
\usepackage{color}
\usepackage[all]{xy}
\usepackage{soul}
\usepackage{enumerate}
\usepackage{bm}
\usepackage{bbm}
\usepackage[normalem]{ulem}
\usepackage{mathrsfs}
\usepackage{tikz-cd}
\usepackage{mathtools}

\usepackage[margin = 1.25in, centering]{geometry}

\allowdisplaybreaks
\numberwithin{equation}{section}
\newtheorem{theorem}{Theorem}[section]
\newtheorem{proposition}[theorem]{Proposition}
\newtheorem{lemma}[theorem]{Lemma}
\newtheorem{corollary}[theorem]{Corollary}

\theoremstyle{definition}
\newtheorem{definition}[theorem]{Definition}
\newtheorem{notation}[theorem]{Notation}

\newtheorem{observation}[theorem]{Observation}
\theoremstyle{remark}
\newtheorem{remark}[theorem]{Remark}
\newtheorem{example}[theorem]{Example}

\newcommand{\CC}{\mathbb{C}}
\newcommand{\NN}{\mathbb{N}}

\newcommand{\ZZ}{\mathbb{Z}}

\newcommand{\G}{G}

\newcommand{\src}{\bm{s}}
\newcommand{\ran}{\bm{r}}

\newcommand{\supp}{\operatorname{supp}}

\newcommand{\Span}{\operatorname{span}}

\newcommand{\Iso}{\operatorname{Iso}}
\newcommand{\Bisc}{\operatorname{Bis}_c}

\newcommand{\Ann}{\operatorname{Ann}}
\newcommand{\Ind}{\operatorname{Ind}}
\newcommand{\Indu}{\operatorname{Ind}_u}
\newcommand{\End}{\operatorname{End}}
\newcommand{\notrational}[1]{1_{(#1,0,#1)}}
\newcommand{\mayberational}[1]{1_{(#1,0,#1)}}

\title[A classification of ideals in Steinberg and Leavitt path algebras]{A classification of ideals in Steinberg and Leavitt path algebras over arbitrary rings}

\author[S. W. Rigby]{Simon W. Rigby}

\author[T. van den Hove]{Thibaud van den Hove}

\address{Simon W. Rigby: Department of Mathematics: Algebra and Geometry, Ghent University, Belgium}

\email{simon.rigby@ugent.be}

\address{Thibaud van den Hove: Mathematical Institute, University of Bonn, Germany}

\email{s6thvand@uni-bonn.de}


\begin{document}

\maketitle

\begin{abstract}
	We give a one-to-one correspondence between ideals in the Steinberg algebra of a Hausdorff ample groupoid $G$, and certain families of ideals in the group algebras of isotropy groups in $G$. This generalises a known ideal correspondence theorem for Steinberg algebras of strongly effective groupoids. We use this to give a complete graph-theoretic description of the ideal lattice of Leavitt path algebras over arbitrary commutative rings, generalising the classification of ideals in Leavitt path algebras over fields. 
	\end{abstract}

\section{Introduction}

The theory of Steinberg and Leavitt path algebras has become highly relevant and these days it has many applications in algebra. Some of these applications are to fundamental questions relating noncommutative rings to more basic algebraic or set-theoretic objects -- What are the possible module types of a ring~\cite{leavitt}? Which monoids can be realised as the monoid of finitely-generated projective modules over a (von Neumann regular) ring \cite{ara-realization, ara-separated}? Which posets can be realised as the spectrum of a ring \cite{abrams-posets}? For which cardinals $\kappa$  is there a ring with $\kappa$-many distinct irreducible representations \cite{ararangaswamy2}, or a ring that has Krull dimension $\kappa$ and yet has no chain of $\kappa$-many prime ideals~\cite{Loper}?

To enable this sort of progress, it is important to develop as far as possible the ideal theory and representation theory of Leavitt path algebras. We take a step further in that direction, by determining the full ideal lattice of a Leavitt path algebra in the most general setting yet, namely when the graph is arbitrary and the coefficients come from an arbitrary commutative ring with identity. We approach this problem initially through the lens of Steinberg algebras, where we find a correspondence theorem for ideals that is good for gaining insights, but somewhat difficult to work with. Then we specialise this to obtain a more user-friendly correspondence theorem for ideals in Leavitt path algebras, and give  a sample of some applications.

Steinberg algebras are convolution algebras of functions from an  ample topological groupoid to a commutative unital ring; they are the algebraic version of the groupoid $\mathrm{C^*}$-algebras that Renault first introduced in \cite{renault}. The role of groupoids is that they are extremely versatile objects for capturing dynamical information, and from early on this led to a unified treatment of many different examples of $\mathrm{C^*}$-algebras of dynamical or combinatorial origins. Steinberg algebras were first seen much later in \cite{steinberg0} and \cite{clark2014groupoid}. Their appearance has led to a proliferation of new examples and results in ring theory, and shows signs of bridging the gap between combinatorial $\mathrm{C^*}$-algebras and their various algebraic analogues \cite{clark-hazrat}.

There are two main reasons why we have used  Steinberg algebras as a starting point for our study of ideals in Leavitt path algebras, rather than the more traditional generators-and-relations construction. Firstly, it gives us access to more topological techniques and this gets us rather quickly to a first view of the full lattice of ideals, from which we can begin a process of specialising and simplifying. Secondly, this approach could in principle be applied to some classes of algebras more general than Leavitt path algebras, of which there are several interesting candidates \cite{abrams-corners,clark-KP_algebras,hazrat-nam}.

In trying to find a useful description or classification of ideals in Steinberg algebras, we took inspiration mainly from Steinberg's Disintegration Theorem \cite{steinberg}, which gives an equivalence between the category of right modules over a Steinberg algebra and the category of sheaves over its underlying groupoid. The equivalence involves a limiting process where multiplication on the right by successively smaller idempotents converges to a module for the isotropy group at each point of the unit space. The collection of all such modules has the structure of a sheaf.

Ideals are in particular also right modules, so one possible approach to our task might have been to try and characterise precisely which sheaves are the images of two-sided ideals. However, we settled instead  on a  different, more symmetric method where one takes the limit from both sides and this converges to an ideal in the group algebra of each isotropy group. The collection of all such ideals is a sheaf of ideals over the unit space of the groupoid (although we do not really emphasise this), and it provides an invariant from which one can uniquely recover the original ideal in the Steinberg algebra.

Leavitt path algebras are $\ZZ$-graded algebras associated to a directed graph, and form a special class of Steinberg algebras. They were introduced before Steinberg algebras, in \cite{abrams2005leavitt} and \cite{ara2007nonstable}, and are the algebraic analogue of graph \(\mathrm{C^*}\)-algebras. Even when these algebras were only just discovered, quite some research was done on the structure of their ideal lattices. For example, in~\cite{ara2007nonstable} and \cite{tomforde2007}, the authors analysed 
the lattice of graded ideals in the Leavitt path algebra $L_K(E)$ of a row-finite graph $E$ and a field $K$. They found that this lattice is independent of $K$, since every ideal is generated by a unique hereditary saturated set of vertices. Moving to a more general setting immediately raises some complications, which we outline as follows.

Firstly, if $E$ is not row-finite, the lattice of graded ideals in $L_K(E)$ is still independent of $K$ but there are some subtleties caused by the vertices that emit infinitely many edges. This can be dealt with by introducing the concept of a breaking vertex for a hereditary saturated set. Secondly, when $E$ has cycles that do not feed into any other cycles --- in other words, $E$ fails Condition (K) --- then $L_K(E)$ has non-graded ideals. Already for the minimal example $L_K(\xymatrix{\bullet \ar@(ur,dr)}\qquad)\cong K[x,x^{-1}]$, there are infinitely many non-graded ideals and the lattice of all ideals depends on the field $K$. However, this turns out to be essentially the only kind of obstacle. One can classify the intermediate ideals between two graded ideals in $L_K(E)$ by assigning an ideal in $K[x,x^{-1}]$ (or, equivalently, a non-constant polynomial in $K[x]$ with constant term 1) to each of the cycles posing an obstacle.

When considering Leavitt path algebras over a field, the ideal lattice is completely understood: see \cite[Theorem 2.8.10]{LPAbook}. Every ideal is determined by a hereditary saturated set~$H$, a set $S$ of breaking vertices for $H$, a set $C$ of cycles, and a  non-constant polynomial with constant term~1 for each cycle in $C$. This classification makes use of the fact that when a scalar multiple of a vertex is contained in an ideal, so is the vertex itself. This brings us to the third major complication. If~$R$ is a commutative unital ring instead of a field, then the ideal structure of $L_R(E)$ is heavily influenced by the ideal structures of $R$ and $R[x,x^{-1}]$. Very little progress has been made towards unravelling the relationship between ideals of $R$, $R[x,x^{-1}]$, and $L_R(E)$. A notable exception is~\cite{clark2019ideals}, where the authors managed to do this quite successfully for row-finite graphs that satisfy Condition (K); that is, by dealing with the third complication in the absence of the first two.

In this paper, we are able to deal with all of the above complications simultaneously, and fully describe the ideal lattice of the Leavitt path algebra of an arbitrary graph over a commutative ring \(R\) with unit. We associate to each ideal \(I\trianglelefteq L_R(E)\) a pair of functions. The first function pinpoints the nearest graded ideal: it maps each admissible pair \((H,S)\) to the largest  ideal \(J\trianglelefteq R\) such that $Jv$ and $Jw^H$ (see \eqref{notation:v^H}) are contained in $I$, for all $v \in H$ and $w \in S$.  To capture the non-graded ideals, we attach to \(I\) a second function from a set of cycles to the set of ideals of \(R[x,x^{-1}]\). These two functions (which must of course satisfy some conditions and be compatible with each other in a precise sense) then characterise the ideal \(I\), as we prove in Theorem \ref{LPA ideal correspondence}. Moreover, to have a description of the infimum and supremum of the ideal lattice as well, we describe the pair of functions to which the sum, intersection, and even the product of ideals corresponds.

When working over a field, we not only have a classification of the ideals of Leavitt path algebras, but it is also known that they satisfy interesting arithmetic properties. For example, it was shown in \cite{rangaswamy} that the multiplication of ideals of Leavitt path algebras is commutative, and that Leavitt path algebras are arithmetical rings and multiplication rings. Moreover, a classification of the prime ideals is given in \cite{rangaswamyprime}, and it is shown in \cite{rangaswamy} that the irreducible ideals of Leavitt path algebras are exactly the prime ideals, and that the primary ideals are exactly the powers of prime ideals. While we have not been able to find analogous results for Leavitt path algebras over arbitrary rings, we do have partial results. For example, we will show that the multiplication of ideals remains commutative, and give conditions on the base ring which imply that Leavitt path algebras will still be arithmetical rings. We also partially generalise the classification of prime ideals from~\cite{rangaswamyprime} to give necessary conditions for ideals in Leavitt path algebras over arbitrary rings to be prime.

We start this paper with Section \ref{sec:preliminaries} by giving a brief introduction to Steinberg algebras and some module theory that  will be needed. Like in \cite{clark2019ideals}, our techniques for Steinberg algebras often rely on a standing assumption that the ample groupoid is Hausdorff. Theorem \ref{disassembly-theorem} is our first major theorem: it describes the ideals of Steinberg algebras by breaking them down, or ``disassembling" them, into ideals of the isotropy group algebras. At the end of Section~\ref{section:disassembly theorem}, we use that theorem to deduce some results from \cite{clark2019ideals} about Steinberg algebras of strongly effective groupoids. Section \ref{sec:lattice structure} is about operations on ideals in Steinberg algebras: we deduce some results about commutativity and distributivity of ideals, and a new theorem about prime Steinberg algebras.

Sections \ref{sec:Leavitt path algebras} and \ref{sec:ideals of LPAs} are the biggest and most important ones, as we give a complete description of the ideal structure of Leavitt path algebras of an arbitrary directed graph over a commutative ring with unit. To be able to use previous results of this paper, we have to start by describing how Leavitt path algebras arise as Steinberg algebras, and hence introduce the boundary path groupoid \(G_E\). As we want to use Theorem \ref{disassembly-theorem} specifically, we also describe which subsets of the topological space \(X(G_E)\) are open, and satisfy the conditions needed. We then define and describe the lattices that we need, so that we can finally prove the ideal correspondence theorem for Leavitt path algebras (Theorem \ref{LPA ideal correspondence}). And to finish, we end the paper with a few sections to describe the product of ideals, the graded ideals, and give necessary conditions for an ideal to be prime.

\section{Preliminaries on groupoids and Steinberg algebras} \label{sec:preliminaries}

A \textit{groupoid} is a small category in which every morphism is invertible. We use the following notation and terminology: the \textit{unit space} is $\G^{(0)} = \{xx^{-1} \mid x \in \G \} = \{x^{-1}x \mid x \in \G\}$, the \textit{source} map is $\src: \G \to \G^{(0)}$, $\src(x) = x^{-1} x$, the \textit{range} map is $\ran: \G \to \G^{(0)}$, $\ran(x) = xx^{-1}$,  and the set of \textit{composable pairs} is $\G^{(2)} = \{(x,y) \in \G \times \G \mid \src(x) = \ran(y) \}$. We use the notation $\G_u = \src^{-1}(u)$, $\G^v = \ran^{-1}(v)$, and $\G_u^v = \G_u \cap \G^v$, whenever $u,v \in \G^{(0)}$ are units. The group $\G_u^u$ is called the \textit{isotropy group} based at~$u$. The \textit{isotropy subgroupoid} of $\G$ is the bundle of groups $\Iso(\G) = \{x \in \G \mid \src(x) = \ran(x) \} = \bigcup_{u \in \G^{(0)}} \G_u^u$. We say that a groupoid $\G$ is \textit{transitive} if $\G_u^v \ne \varnothing$ for every $u, v \in \G^{(0)}$, and it is \textit{principal} if $\Iso(\G) = \G^{(0)}$. The \textit{orbit} of a unit $u \in G^{(0)}$ is the set $\mathcal{O}_v = \{u \in G^{(0)} \mid G_u^v \ne \varnothing\}$. A subset $V \subseteq G^{(0)}$ is called \emph{invariant} if $V = \src(\ran^{-1}(V)) = \ran(\src^{-1}(V))$, which is the same as saying that $V$ is a union of orbits.

A topological groupoid is a groupoid $G$ equipped with a topology such that inversion $G \to G$  and composition  $G^{(2)}\to G$ are continuous (where $G^{(2)}$ has the subspace topology from $G \times G$). An \textit{ample groupoid} is a topological groupoid $\G$ in which $\src$ is a local homeomorphism and $\G^{(0)}$ is locally compact, totally disconnected, and Hausdorff. Equivalently, an ample groupoid is a topological groupoid $\G$ in which $\src$ is an open map and $\G$ has a basis of \textit{compact open bisections}; that is, compact open subsets of $\G$ on which $\src$ and $\ran$ are injective. In any ample groupoid, $G^{(0)}$ is open and $\Iso(G)$ is closed in $G$. If $G$ is Hausdorff, then $G^{(0)}$ is both open and closed \cite[p.\!~29]{rigby}.

We shall assume throughout that {$G$ is a \textit{Hausdorff} ample groupoid, and $R$ is an arbitrary commutative ring with 1}. If a topology is given to $R$, it will always be the discrete topology.  We write $\Bisc(\G)$ for the set of compact open bisections in $\G$, and $\Bisc(\G^{(0)})$ for the set of compact open subsets of $\G^{(0)}$. For $A, B \in \Bisc(G)$, we define $AB = \{ab \mid a \in A, b \in B, \src(a) = \ran(b) \}$ and $A^{-1} = \{a^{-1} \mid a \in A \}$. Under these operations, $\Bisc(G)$ is an inverse semigroup and  $\Bisc(G^{(0)})$ is the set of idempotents in $\Bisc(G)$. For any $A \subseteq G$, let $\bm{1}_A: \G \to R$ be its characteristic function. For any function $f: G \to R$, the \emph{support} of $f$ is $\supp(f) = \{x \in G \mid f(x) \ne 0\}$. The term \emph{ideal} always means two-sided ideal, and we specify if we mean left or right ideals. If $A$ is a ring, we write $\mathcal{L}(A)$ for the set of ideals in $A$ and we sometimes write $I \trianglelefteq A$ to mean $I \in \mathcal{L}(A)$.

\begin{definition} \cite{steinberg0} \label{steinberg-definition}
Let $\G$ be a Hausdorff ample groupoid. The \textit{Steinberg algebra} $A_R(G)$ is the $R$-module of locally constant, compactly supported functions $f: G \to R$,
equipped with the \textit{convolution product}:
\begin{align} \label{convol}
f*g(x) = \sum_{y \in \G_{\src(x)}} f(xy^{-1})g(y) = \sum_{\substack{(z,y) \in \G^{(2)},\\ zy = x}} f(z)g(y) && \text{for all } f, g \in A_R(\G) \text{ and } x \in \G.
\end{align}
\end{definition}
Every $f \in A_R(G)$ is continuous and can be written as a sum $f = \sum_{B \in {D}} r_B \bm{1}_{B}$, where ${D}$ is some finite set of mutually disjoint compact open bisections, and each $r_B \in R$ \cite[Proposition~2.5]{rigby}.
If $A, B \in \Bisc(\G)$ then (\ref{convol}) yields $\bm{1}_A * \bm{1}_B = \bm{1}_{AB}$. The following lemma is straightforward to prove from the definitions:

\begin{lemma} \label{lem1} If $f \in A_R(G)$, $U \in \Bisc(G^{(0)})$, and $B \in \Bisc(G)$, then for all $x \in G$,
\begin{align*}
f* \bm{1}_U(x) &= f(x) \bm{1}_{U}(\src(x)),
& f*\bm{1}_B(x) &= \begin{cases} f(xy^{-1}) & \text{if } B \cap G_{\src(x)} = \{y\} \\0 & \text{if } B \cap G_{\src(x)} = \varnothing, \end{cases} \\
\bm{1}_U* f(x) &= f(x) \bm{1}_{U}(\ran(x)),
& \bm{1}_B*f(x) &= \begin{cases} f(y^{-1}x) & \text{if } B \cap G^{\ran(x)} = \{y\} \\0 & \text{if } B \cap G^{\ran(x)} = \varnothing. \end{cases}
\end{align*}
\end{lemma}

Steinberg algebras are locally unital rings in the sense that for every finitely generated subalgebra $S\subseteq A_R(G)$ there exists an idempotent $e\in A_R(G)$ such that $es = se = s$ for all $s \in S$. These idempotents can be chosen from the set $\{\bm{1}_U \mid U \in \Bisc(G^{(0)})\}$.

As an aside, to define the Steinberg algebra of a not necessarily Hausdorff ample groupoid, one needs to define $A_R(G)$ differently from how we have done it: it is the $R$-module generated by the functions $\bm{1}_K: G \to R$ where $K$ ranges over the set of compact open Hausdorff subsets of $G$ \cite[Definition 4.1]{steinberg0}. This is equivalent to Definition \ref{steinberg-definition} if $G$ is Hausdorff. But if $G$ is not Hausdorff, the functions in $A_R(G)$ are not necessarily locally constant and compactly supported.

\subsection{Module theory for Steinberg algebras} For our purposes, a left module over a locally unital ring $A$ is an abelian group $M$ with a ring homomorphism $\rho: A \to \End M$ such that $\rho(A)M = M$. Usually $\rho(a)m$ is written as $a \cdot m$ or even $am$. The condition that $\rho(A)M = M$ (or $A \cdot M = M$) is equivalent to the condition that for each $m \in M$ there is a local unit $e \in A$ such that $e\cdot m = m$. 
If $A$ is an $R$-algebra, then by defining $r \cdot m = (re) \cdot m$, we give $M$ the structure of an $R$-module. The \emph{annihilator} of $M$ is the ideal $\Ann(M) = \{a \in A \mid aM = 0\} \trianglelefteq A$.

\begin{notation} We shall use the following notational conventions. Let $u\in G^{(0)}$.
	\begin{enumerate}[(i)]
	\item By the notation $R\G_u^u$ we mean the group $R$-algebra of the isotropy group $\G_u^u$. Formally, we take $R\G_u^u$ to be the set of finitely supported functions $\alpha: \G_u^u \to R$ equipped with the convolution product $\alpha \cdot \beta(x) = \sum_{y \in G_u^u} \alpha(xy^{-1})\beta(y)$ for all $x \in G_u^u$ and all $\alpha, \beta \in RG_u^u$.
	\item For $f \in A_R(G)$, we write $f_u = f|_{G_u^u} \in R\G_u^u$.
	\item By the notation $RG_u$ (respectively, $RG^u$) we mean the $R$-module of finitely supported functions $\beta: G_u \to R$ (respectively, finitely supported functions $\beta: G^u \to R$). These are free $R$-modules.
	\item For $t \in G$ we write $1_t$   for the function sending $t$ to 1 and all other elements of its domain to 0. Thus we denote the generators of $RG_u^u$, $RG_u$, and $RG^u$ as $\{1_z: G_u^u \to R \mid z \in G^u_u\}$, $\{1_t : G_u \to R \mid t \in G_u \}$, and $\{1_t : G^u \to R \mid t \in G^u\}$ respectively. \end{enumerate}
\end{notation}

It is easy to show that the map $f \mapsto f_u$ is a surjective $R$-module homomorphism from $A_R(G)$ to $RG_u^u$. Likewise, $f \mapsto f|_{G_u}$ and $f \mapsto f|_{G^u}$ are surjective $R$-module homomorphisms from $A_R(G)$ to $RG_u$ and $RG^u$ respectively.

The $R$-module $RG_u$ has the additional structure of an $(A_R(G),RG_u^u)$-bimodule \cite[Proposition~7.8]{steinberg0}. The left action of $A_R(G)$ and the right action of $RG_u^u$ on $RG_u$ are defined as follows:
\begin{align} \label{laction}
	f \cdot 1_t &=  \sum_{y \in G_{\ran(t)}}f(y)1_{yt},& f \in A_R(G),\ t \in G_u;\\ \label{raction}
	1_t \cdot 1_z &= 1_{tz},& t \in G_u,\  z \in G_u^u.
\end{align}
For direct comparison with \cite[Proposition 7.8]{steinberg0}, one should perform a change of variables $y = xt^{-1}$ in~(\ref{laction}) and see that
 $f \cdot 1_t = \sum_{x \in G_u} f(xt^{-1})1_x$. 
Using this, one can also derive an alternative version of (\ref{laction}): if $\beta \in RG_u$ and $f\in A_R(G)$ then $f \cdot \beta = (f * h)|_{G_u}$ for all $h \in A_R(G)$ such that $h|_{G_u} = \beta$. An alternative version of (\ref{raction}) is that $\beta \cdot 1_z = (h*\bm{1}_{B})|_{G_u}$ for any $h \in A_R(G)$ such that $h|_{G_u} = \beta$, and any $B \in \Bisc(G)$ that contains $z$.

By the dual versions of (\ref{laction}) and (\ref{raction}), $RG^u$ is an $(RG_u^u,A_R(G))$-bimodule. If we nominate a single representative $j_v \in G_v^u$ for all $v \in \mathcal{O}_u$ and write $
\gamma_v = 1_{j_v^{-1}}$, the set $\{\gamma_v \mid v \in \mathcal{O}_u\}$ is a basis for $RG_u$ as a free right $RG_u^u$-module. 
 If we write $\delta_v = 1_{j_v}$, the set $\{\delta_v \mid v \in \mathcal{O}_u\}$ is a basis for $RG^u$ as a free left $RG_u^u$-module. We fix these bases throughout throughout this section, and it makes some calculations simpler if we assume that $j_u = u$.
 \label{basis}

\subsubsection{Induced modules}\cite{steinberg0}
Let $u \in G^{(0)}$. Given any left $RG_u^u$-module $M$, the \textit{induced module} is the left  $A_R(G)$-module:
\[
\textstyle \Indu M = RG_u \bigotimes_{RG_u^u}M.
\]
This gives a functor $\Indu$ from the category of left $RG_u^u$-modules to the category of left $A_R(G)$-modules. This functor is exact because it is defined by tensoring over a free module. It also sends simple modules to simple modules  and reflects isomorphisms \cite[Proposition 7.19]{steinberg0}.  Since $RG_u = \bigoplus_{v \in \mathcal{O}_u}\gamma_v RG_u^u$, we have $\Ind_u M =  \bigoplus_{v \in \mathcal{O}_u} \gamma_v \otimes M$. We can express the left action of $A_R(G)$ on $\Indu M$ in the co-ordinates of the basis $\{\gamma_v \mid v 
\in \mathcal{O}_u \}$ for $RG_u$ (compare with \cite[p.\!~4]{steinberg-eh}):

\begin{lemma}\label{ann}
	Let $m \in M$, $u \in G^{(0)}$, $v \in \mathcal{O}_u$, and $f \in A_R(G)$. For each $w \in \mathcal{O}_u$, let $B_w \in \Bisc(G)$ be a bisection containing $j_w$. Then:
	\[
	f \cdot (\gamma_v \otimes m) = \sum_{w \in \mathcal{O}_u} \gamma _w \otimes (\bm{1}_{B_w} * f * \bm{1}_{B_v^{-1}})_u \cdot m.
	\]
	Moreover, $\Ann(\Ind_u M) = \big\{f \in A_R(G) \mid (\bm{1}_{B_w} * f * \bm{1}_{B_v^{-1}})_u \in \Ann(M) \text{ for all } v,w \in \mathcal{O}_u\big\}$.
\end{lemma}
\begin{proof}
By definition of the $A_R(G)$-action (\ref{laction}), we can compute
\begin{align*} \label{basis-action}
	f \cdot (\gamma_v \otimes m) &= (f \cdot 1_{j_v^{-1}}) \otimes m = \sum_{y \in G_v} f(y) 1_{yj_v^{-1}}\otimes m \\
	&= \sum_{w \in \mathcal{O}_v} \sum_{y \in G_v^w} f(y) 1_{j_w^{-1}}1_{j_wyj_v^{-1}} \otimes m = \sum_{w \in \mathcal{O}_v} \gamma _w \otimes \sum_{y \in G_v^w}f(y)1_{j_w y j_v^{-1} }\cdot m.
\end{align*}
By applying Lemma \ref{lem1} and then a change of variables, we have for all $w \in\mathcal{O}_v$:  \begin{align*}
\sum_{y \in G_v^w}f(y)1_{j_w y j_v^{-1}} &= \sum_{y \in G_v^w} \bm{1}_{B_w} * f * \bm{1}_{B_v^{-1}}(j_w y j_v^{-1}) 1_{j_w y j_v^{-1}} \\ 	
&= \sum_{z \in G_u^u} \bm{1}_{B_w}* f * \bm{1}_{B_v^{-1}}(z) {1}_z = (\bm{1}_{B_w} * f * \bm{1}_{B_v^{-1}})_u. 
 \end{align*}
 Since the $\gamma_w$'s are $RG_u^u$-linearly independent, the ``moreover" part of the statement is clear.
 \end{proof}

\subsubsection{Limits of modules} \cite{steinberg}
For a unit $u \in G^{(0)}$, define the directed system $(\Delta_u, \le)$ where \[\Delta_u = \{U \in \Bisc(G^{(0)}) \mid u \in U\}\]
and $U \le V$ if and only if $V \subseteq U$.  Given a left $A_R(G)$-module $N$, we define a left $RG_u^u$-module $\mathcal{E}(N)_u$ by taking the direct limit of $R$-modules:
	\[\mathcal{E}(N)_u = \lim_{\substack{\longrightarrow \\U \in \Delta_u}} \bm{1}_U N.\]
	where the connecting homomorphism $\rho_V^U: \bm{1}_UN \to \bm{1}_V N$, for $U \le V$, is left-multiplication by $\bm{1}_V$. The image of $n \in \bm{1}_U N$ in the direct limit is denoted by $[n]_u$. Using the fact that $N = \bigcup_{U \in \Bisc(G^{(0)})} \bm{1}_U N$ one can show that $n \mapsto [n]_u$ defines a surjective $R$-module homomorphism $N \to \mathcal{E}(N)_u$. We equip $\mathcal{E}(N)_u$ with the structure of a left $RG_u^u$-module as follows:  for each $z \in G_u^u$, choose an arbitrary $B \in \Bisc(G)$ containing $z$, and define \begin{align*}
 	z \cdot [n]_u = [\bm{1}_{B} n]_u		\end{align*}
 		for all $[n]_u \in \mathcal{E}(N)_u$. This does not depend on the choice of $B$.
 		
 		This construction extends to a functor $\mathcal{E}(-)_u$ from the category of left $A_R(G)$-modules to the category of left $RG_u^u$-modules. Since direct limits over directed systems preserve exact sequences,  $\mathcal{E}(-)_u$ is an exact functor. A left ideal $I$ in $A_R(G)$ corresponds to an $RG_u^u$-submodule of $RG^u$, namely $\{f|_{G^u} : f \in I \}$, which turns out to be isomorphic to $\mathcal{E}(I)_u$:
 		
\begin{lemma} \label{NB}
	Let $I$ be a left ideal in $A_R(G)$. Then $\mathcal{E}(I)_u \cong \{f|_{G^u} : f \in I \}$ as left $RG_u^u$-modules. In particular, $\mathcal{E}(A_R(G))_u \cong RG^u$.
\end{lemma}

\begin{proof}
	Since $A_R(G)$ is locally unital, $\bigcup_{U \in \Bisc(G^{(0)})} \bm{1}_U I = I$, and we have $\mathcal{E}(I)_u = \{[f]_u \mid f \in I \}$. Define the $R$-module homomorphism
	\begin{align*}
	\chi:  \mathcal{E}(I)_u \to \{f|_{G^u} \in RG^u \mid f \in I \}, &&
	[f]_u \mapsto f|_{G^u} \text{ for all } f \in I.
	\end{align*}
	To check that this is well-defined, suppose $[f]_u = [g]_u$ for some $f, g \in I$. Then $[f-g]_u = [0]_u$, which means $\bm{1}_U*(f-g) = 0$ for some $U \in \Bisc(G^{(0)})$ containing $u$. By Lemma \ref{lem1}, $f|_{G^u} = g|_{G^u}$. Clearly $\chi$ is surjective. To check that it is injective, suppose $f|_{G^u} = g|_{G^u}$ for some $f, g \in I$. Then $\ran(\supp (f-g))$ is a compact subset of $G^{(0)}$ not containing $u$, so there is an open set $K \in \Bisc(G^{(0)})$ with  $u\in K$ and $K \cap \ran(\supp(f-g)) = \varnothing$. Then $\bm{1}_K*(f-g) = 0$, so $[f]_u = [g]_u$.   It is easy to check that $\chi$ is also an isomorphism of left $RG_u^u$-modules.
\end{proof}

Note that if $I$ is a two-sided ideal, then the image of $\chi$ is not only an  $RG_u^u$-submodule of $RG^u$, but also an $(RG^u_u, A_R(G))$-submodule of $RG^u$. Note also that if $I$ and $J$ are two ideals in $A_R(G)$ with $I \subseteq J$, then the corresponding homomorphism  $\mathcal{E}(I)_u \to \mathcal{E}(J)_u$ sends $[f]_u \mapsto [f]_u$ for all $f \in I$. One easily checks that $\chi$ is natural in the sense that the diagram $\mathcal{E}(I)_u \to \mathcal{E}(J)_u$ is isomorphic to the inclusion $\{f|_{G^u}: f \in I\} \rightarrow \{f|_{G^u}: f \in J\}$.

\begin{lemma} \label{fu=gu}
Let $u \in G^{(0)}$. If $f, g \in A_R(G)$ and $f_u = g_u$ in $RG_u^u$, then there exist  $V, W \in \Bisc(G^{(0)})$ with $u \in V \cap W$ such that $\bm{1}_{V'} * f * \bm{1}_{W'} = \bm{1}_{V'} * g * \bm{1}_{W'}$  for all open subsets \(V'\subseteq V\) and \(W'\subseteq W\). 
\end{lemma}

\begin{proof}
	The set $D= \supp(f-g)\cap G^u$ is finite because $f-g$ is compactly supported and $G^u$ is a discrete subspace of $G$. Since $u \notin \src(D)$, there exists a neighbourhood $W \in \Bisc(G^{(0)})$ of $u$ with $\src(D) \cap W = \varnothing$, and this implies ${(f*\bm{1}_W)}|_{G^u} = (g* \bm{1}_W)|_{G^u}$. It follows from Lemma \ref{NB} that  $\bm{1}_V * f * \bm{1}_W = \bm{1}_V * g * \bm{1}_W$ for some $V \in \Bisc(G^{(0)})$ containing $u$. Clearly this implies that $\bm{1}_{V'} * f * \bm{1}_{W'} = \bm{1}_{V'} * g * \bm{1}_{W'}$  for open subsets \(V'\subseteq V\) and \(W'\subseteq W\).
\end{proof}

In particular, the lemma above implies that $\displaystyle \lim_{ \substack{\longrightarrow \\  (U,V) \in \Delta_u\times \Delta_u}}\bm{1}_U * A_R(G) * \bm{1}_V \cong RG_u^u$.

\section{Disassembly of ideals} \label{section:disassembly theorem}

Define the set $X(G)$ to be the union of the group algebras of all the isotropy subgroups of $G$, \begin{equation}
	\label{XG} X(\G) = \bigcup_{u \in \G^{(0)}}{RG_u^u}.
	\end{equation}
Formally, $RG_u^u$ is the set of finitely supported functions $G_u^u \to R$, so (\ref{XG}) is a disjoint union. 
 Define the map $p: X(G) \to G^{(0)}$ by $p(\alpha) = u$ if $\alpha \in RG_u^u$. To each $f \in A_R(G)$ is associated a section of~$p$, namely  $s_f: G^{(0)} \to X(G)$, $u \mapsto f_u$.

\begin{lemma}
	The collection of sets $\{s_f(U) \mid f \in A_R(G),\ U \in \Bisc(G^{(0)})\}$ is a basis for a topology on $X(G)$ such that each $s_f$ is continuous and $p$ is a local homeomorphism.
\end{lemma}
\begin{proof}
	Let $f, g \in A_R(G)$ and $U, V \in \Bisc(G^{(0)})$, and suppose $\alpha \in s_f(U) \cap s_g(V)$.   Then $p(\alpha) \in U \cap V$ and $\alpha = f_{p(\alpha)} = g_{p(\alpha)}$. By Lemma \ref{fu=gu}, there exist $K, L \in \Bisc(G^{(0)})$ containing $p(\alpha)$ such that ${\bm{1}_K * f * \bm{1}_L = \bm{1}_K * g * \bm{1}_L}$, and then for $W = K \cap L$ we have $\bm{1}_W *f * \bm{1}_W = \bm{1}_W * g * \bm{1}_W$. So $\alpha \in s_f(W) = s_g(W) \subseteq s_f(U) \cap s_g(V)$. This proves that the collection of sets displayed in the lemma is a basis for a topology.

	If $u \in s_g^{-1}(s_f(U))$, then $g_u \in s_g(U) \cap s_f(U)$. We showed in the previous paragraph that $g_u \in s_g(W) =s_f(W) \subseteq s_g(U) \cap s_f(U)$ for some open set $W \subseteq U$ containing $u$. Then $u \in W = s_g^{-1}(s_g(W)) \subseteq s_g^{-1}(s_f(U))$, from which we conclude that $s_g^{-1}(s_f(U))$ is open. This proves that every $s_g$ is continuous. It is clear that $p$ is a local homeomorphism.
\end{proof}

For a given $t \in G_u^v$ there is an important isomorphism of $R$-algebras:
\begin{align} \label{translation}
	{\tau}_t: RG_u^u \to RG_v^v, && {\tau}_{t}(\alpha) := \alpha^t, && \text{where }\alpha^t(z) = \alpha(t^{-1}zt)
\end{align}
for all $\alpha \in RG_u^u$, $z \in G_v^v$. We call this the \textit{translation isomorphism} associated to $t$. Expressed in another way, if $\alpha = \sum {r_i}1_z$ where $z \in G^u_u$, then $\alpha^t = \sum {r_i}1_{tzt^{-1}}$.
If $\alpha = f_u \in RG_u^u$ for some $f \in A_R(G)$, and $t \in A$ for some $A\in \Bisc(G)$, then by Lemma \ref{lem1} we can express $\alpha^t = (\bm{1}_A * f * \bm{1}_{A^{-1}})_v$. 

\subsection{The disassembly map and its inverse}

Given an ideal $I \subseteq A_R(G)$, we define the \emph{disassembly map}:
\begin{align} \label{disint-map}
	I \mapsto Y_I:= \bigcup_{u \in G^{(0)}} I_u,&&
	\text{where} \quad I_u := \{f_u \mid f \in I\} \subseteq RG_u^u.
\end{align}
Our first major theorem is the following:

\begin{theorem} \label{disassembly-theorem}
	The disassembly map
	$
	I \mapsto Y_I= \bigcup_{u \in G^{(0)}} I_u
	$
	 defines a bijection from $\mathcal{L}(A_R(G))$ to the set $\Delta_{G,R}$ of all open subsets $Y\subseteq X(G)$ satisfying the conditions:
	 \begin{enumerate}[\rm (D1)] \label{D1} \item For all $u \in G^{(0)}$, $Y \cap RG_u^u$ is an ideal in $RG_u^u$;
	 \item	\label{D2}
 For all $u, v \in G^{(0)}$ and $t \in G_u^v$, the translation map $\tau_{t}$ restricts to an isomorphism $Y \cap RG_u^u \to Y \cap RG_v^v$.
 \item \label{D3} For all $\alpha \in Y$, there exists some $f \in A_R(G)$ such that $f_{p(\alpha)} =  \alpha$ and $(f * \bm{1}_C)_v \in Y$ for all $v \in G^{(0)}$ and all $C \in \Bisc(G)$.
 \end{enumerate}
\end{theorem}

We give the proof of Theorem \ref{disassembly-theorem} in the sequence of Lemmas \ref{disint-lem1}--\ref{disint-lem3}. Before taking on that task, we make a few remarks about the set $\Delta_{G,R}$ and the conditions that define it. Firstly, (\hyperref[D2]{D2}) implies $Y \cap RG_u^u\cong Y \cap RG_v^v$ whenever $u$ and $v$ are in the same orbit. Secondly, to check (\hyperref[D2]{D2}) it is sufficient to check that  $\tau_t(Y\cap RG_{u}^{u}) \subseteq Y \cap RG_{v}^{v}$ for all $u, v \in G^{(0)}$, $t \in G_u^v$, because we already know that $\tau_t: RG_u^u \to RG_v^v$ and its inverse $\tau_t^{-1} = \tau_{t^{-1}}: RG_v^v \to RG_u^u$ are isomorphisms. Thirdly, when applying this theorem, the most difficult of the three conditions to check is usually~(\hyperref[D3]{D3}). Fortunately, there are some special cases where we can show that (\hyperref[D1]{D1}) and (\hyperref[D2]{D2}) together imply~(\hyperref[D3]{D3}) for open subsets of $X(G)$. These special cases include when $G$ is strongly effective (Corollary \ref{se-d3}) and when $G$ is the boundary path groupoid of a graph (Lemma~\ref{D3 not necessary for LPA}). We are not sure whether or not (\hyperref[D1]{D1}) and (\hyperref[D2]{D2}) imply (\hyperref[D3]{D3}) in the general case.

\begin{lemma} \label{disint-lem1}
If $I$ is an ideal in $A_R(G)$, then $Y_I = \bigcup_{u \in G^{(0)}} I_u$ is an open subset of $X(G)$ satisfying conditions {\rm(\hyperref[D1]{D1})--(\hyperref[D3]{D3})}.	
\end{lemma}

\begin{proof}
	Throughout the proof, let $\alpha \in Y_I$ and write $u = p(\alpha)$. Then  $\alpha = f_u$ for some representative $f \in I$ (also remaining fixed throughout). For any open $U \subseteq G^{(0)}$, we have  $\alpha \in s_f(U) \subseteq Y_I$. Since $\alpha$ is arbitrary, this proves $Y_I$ is open.
	
	It is easy to see that $Y_I \cap RG_u^u = I_u$ is a subgroup of $RG_u^u$. Since $\supp(f)$ is compact and $G_u$ is a discrete subspace of $G$, there are only finitely many $w \in G^{(0)}$ such that $\supp(f) \cap G_u^w \ne \varnothing$. Since $G^{(0)}$ is Hausdorff, we can choose a neighbourhood $K \in \Bisc(G^{(0)})$ of $u$ such that $\supp(f) \cap G_u^w = \varnothing$ for all $w \in K \setminus \{u\}$. Now let $\gamma \in RG_u^u$, and pick a representative $g \in A_R(G)$ such that $\gamma = g_u$. For all $x \in G_u^u$, we have
	\begin{align*}
	g* \bm{1}_K * f (x) &= \sum_{y \in G_u} g* \bm{1}_K(xy^{-1})f(y) = \sum_{y \in G_u} g(xy^{-1})\bm{1}_K(\src(xy^{-1})) f(y)\\
	&=  \sum_{y \in G_u}g(xy^{-1})\bm{1}_K(\ran(y))f(y) = \sum_{y\in G_u^u} g(xy^{-1})f(y)\\ &=  \sum_{y \in G_u^u} \gamma(xy^{-1})\alpha(y) = \gamma \cdot \alpha (x).
	\end{align*}
	The second equality in this calculation is justified by Lemma \ref{lem1}. The fourth equality is justified by the observation that $y \in \supp(f) \cap G_u \cap \ran^{-1}(K)$ implies $y \in G_u^u$, because of how $K$ was chosen. 
	This shows that $(g*\bm{1}_K *f)_u = \gamma \cdot \alpha$. Since $f \in I$, we have $g*\bm{1}_K * f \in I$, and it follows that $\gamma \cdot \alpha \in I_u$. In the same way, one shows that $\alpha\cdot \gamma \in I_u$. Since $\alpha$ and $\gamma$ were arbitrary, the conclusion is that (\hyperref[D1]{D1}) holds.

Now suppose $t \in G_u^v$ for some $v \in G^{(0)}$. Pick some $A \in \Bisc(G)$ with $t \in A$. Then $\tau_{t}(\alpha) = (\bm{1}_A * f * \bm{1}_{A^{-1}})_v \in I_v$. Since $\alpha$ was arbitrary,  $\tau_{t}(I_u) \subseteq I_v$ for all $u,v\in G^{(0)}$, and (\hyperref[D2]{D2}) holds. Condition (\hyperref[D3]{D3}) is clearly implied by the fact that $I$ is an ideal.
\end{proof}

Given a subset $Y \subseteq X(G)$, we define the \emph{reassembly map},
\begin{equation} \label{reassembly-map}
Y \mapsto \mathcal{I}(Y) := \big\{ f \in A_R(G) \mid (\bm{1}_A * f * \bm{1}_B)_u \in Y \text{ for all } u \in G^{(0)} \text{ and all } A, B \in \Bisc(G) \big\}.
\end{equation}
If $Y \cap RG_u^u$ is an $R$-submodule of $RG_u^u$ for all $u \in G^{(0)}$, then clearly $\mathcal I(Y)$ is an $R$-submodule of $A_R(G)$. It is also an ideal: if $f \in  \mathcal I(Y)$ and $C,D \in \Bisc(G)$ then $(\bm{1}_A*(\bm{1}_C* f* \bm{1}_D)*\bm{1}_B)_u = (\bm{1}_{AC}* f * \bm{1}_{DB})_u\in Y$ for all $u \in G^{(0)}$ and $A, B \in \Bisc(G)$,  so clearly $\bm{1}_C * f * \bm{1}_D \in \mathcal I(Y)$. Since characteristic functions generate $A_R(G)$, it follows that $A_R(G)*\mathcal I(Y)*A_R(G) \subseteq \mathcal I(Y)$.

Given an ideal $I$ and a unit $u \in G^{(0)}$, consider the $RG_u^u$-module \[
	M_u = RG^u/I_u RG^u.
	\]

\begin{lemma}
If $I$ is an ideal in $A_R(G)$, then
\[
I = \bigcap_{u \in G^{(0)}} \Ann (\Ind_u M_u) = \mathcal{I}(Y_I).
\]	
\end{lemma}

\begin{proof} The proof plays out in three steps: we show that $I \subseteq \mathcal{I}(Y_I)$ and $ \bigcap_{u \in G^{(0)}} \Ann(\Ind_u M_u)\supseteq \mathcal{I}(Y_I) $ and finally that $I = \bigcap_{u \in G^{(0)}} \Ann(\Ind_u M_u) $. The first step is easy:
	if $f \in I$, then $\bm{1}_A * f * \bm{1}_B \in I$ for all $A, B \in \Bisc(G)$, so $(\bm{1}_A * f * \bm{1}_B)_u \in Y_I$ for all $u\in G^{(0)}$, and it follows that $f \in \mathcal{I}(Y_I)$.
	
	Now suppose $h \in \mathcal{I}(Y_I)$ and $u \in G^{(0)}$.  Let $m = \gamma_v \otimes(\beta + I_uRG^u) \in \Indu M_u$, where $\beta \in RG^u$ and $v \in \mathcal{O}_u$. By Lemma~\ref{ann}, we have
	\[
	h \cdot m = \sum_{w \in \mathcal{O}_u} \gamma_w \otimes \big( (\bm{1}_{B_w} * h * \bm{1}_{B_v^{-1}})_u \cdot \beta + I_uRG^u\big)= 0
	\]
	because $h \in \mathcal{I}(Y_I)$ implies $(\bm{1}_{B_w} * h * \bm{1}_{B_v^{-1}})_u \in I_u$. Therefore $h$ annihilates $\Ind_u M_u$ for all $u \in G^{(0)}$. This proves $\bigcap_{u \in G^{(0)}} \Ann (\Indu M_u)\supseteq \mathcal{I}(Y_I)$ and concludes the second step of the proof.
	
	By \cite[Theorem 5, Remark 6]{steinberg-eh}, we have that
	\begin{align*} 	I = \bigcap_{u \in G^{(0)}} \Ann\Big(\Indu \ \mathcal{E}(A_R(G)/I)_u\Big).
	\end{align*}
	The proof will be completed if we can prove that  $M_u \cong \mathcal{E}(A_R(G)/I)_u$ as $RG_u^u$-modules. Using the exactness of the functor $\mathcal{E}(-)_u$, it is sufficient to prove that the diagram $\mathcal E (I)_u \to \mathcal E(A_R(G))_u$ is isomorphic to the diagram $I_uRG^u \to RG^u$. Applying Lemma \ref{NB}, it is sufficient to prove that  \begin{align} \label{eq}
	I_uRG^u = \{f|_{G^u} \in RG^u \mid f \in I \}.\end{align}
	
Suppose that $\alpha \in I_u$, $t \in G^u_v$, and $t \in B \in \Bisc(G)$. Then $\alpha = f_u = \sum_{z \in G_u^u} f(z)1_z$ for some $f \in I$, and \[\alpha \cdot 1_t = \sum_{z \in G_u^u} f(z)1_{zt} = \sum_{z \in G^u_u} f*\bm{1}_{B}(zt) 1_{zt}.\]
	By choosing a neighbourhood $U \in \Bisc(G^{(0)})$ of $u$ small enough that $\supp(f) \cap G^u \cap \src^{-1}(U) \subseteq G_u^u$ (as in the proof of Lemma \ref{disint-lem1}) we can ensure that $f(x)\bm{1}_U(\src(x))\bm{1}_B(b) = 0$ for all $x \in G^u$ and all $b \in G^{\src(x)}$, except possibly if $x \in G_u^u$ and $b = t$. Consequently,
	\begin{align*}
	(f*\bm{1}_U*\bm{1}_B)|_{G^u} &= \sum_{y \in G^u}f* \bm{1}_U*\bm{1}_B(y)1_y = \sum_{y \in G^u}\Big(\sum_{\substack{x,u,b \in G\\xub = y}} f(x)\bm{1}_U(u) \bm{1}_B(b) \Big) 1_y\\ &= \sum_{y \in G^u}\Big(\sum_{\substack{x,b\in G\\xb=y}}f(x)\bm{1}_{U}(\src(x))\bm{1}_{B}(b)\Big)1_y  = \sum_{x\in G_u^u}f(x)1_{xt} = \alpha \cdot 1_t.
	\end{align*}

	This proves the $``\subseteq"$ part of (\ref{eq}). Now suppose that $h \in I$ and $\beta = h|_{G^u}$. We can write $\beta = \sum_{v \in \mathcal{O}_u} z_v \delta_v$ where  $\{\delta_v \mid v \in \mathcal{O}_u\}$ is the basis for $RG^u$ that was introduced in \S\ref{basis}, and $z_v \in RG_u^u$ (only finitely many of which are nonzero). Let $w \in \mathcal{O}_u$ be fixed for now. We can find a small enough compact open bisection  $B_w$ that contains $j_w$, where ``small enough" means that $v \notin \src(B_w)$ for all $v\in \mathcal{O}_u$ such that $v \ne w$ and $z_v \ne 0$. Then for all $v \ne w$ such that $z_v \ne 0$,\[\delta_v \cdot \bm{1}_{B_w^{-1}} = 1_{j_v} \cdot \bm{1}_{B_w^{-1}} = \sum_{y \in G^v} \bm{1}_{B_w^{-1}}(y)1_{j_v y} = 0.\]
	Therefore,
	\[
	\beta \cdot \bm{1}_{B_w^{-1}} = \sum_{v \in \mathcal{O}_u} z_v \delta_v \cdot \bm{1}_{B_w^{-1}} = z_w \delta_w \cdot \bm{1}_{B_w^{-1}} = z_w \cdot 1_u.
	\]
	This calculation shows that $z_w \cdot 1_u = (h*\bm{1}_{B_w^{-1}})|_{G^u}$. But $z_w \cdot 1_u (x) = z_w(x)$ for all $x \in G^u_u$, so it follows that $z_w = (h*\bm{1}_{B_w^{-1}})|_{G_u^u} \in I_u$. Since $w\in \mathcal{O}_u$ was arbitrary, we have $z_v\in I_u$ for all $v \in \mathcal{O}_u$, and $\beta \in I_uRG^u$. This proves the $``\supseteq"$ part of (\ref{eq}).
	\end{proof}
	
At this point, we have already shown that the disassembly map (\ref{disint-map}) is injective. Now we show that it is surjective onto $\Delta_{G,R}$.

\begin{lemma} \label{disint-lem3}
	If $Y$ is an open subset of $X(G)$ satisfying conditions {\rm(\hyperref[D1]{D1})--(\hyperref[D3]{D3})}, then $Y_{\mathcal{I}(Y)} = Y$.
\end{lemma}

\begin{proof}
	Suppose  $\alpha \in Y_{\mathcal{I}(Y)}$. Then $\alpha = f_{p(\alpha)}$ for some $f \in \mathcal{I}(Y)$. Since $A_R(G)$ is locally unital, there is a $U \in \Bisc(G^{(0)})$ such that $\bm{1}_U * f * \bm{1}_U = f$. Then $f \in \mathcal{I}(Y)$ implies $\alpha = (\bm{1}_U * f * \bm{1}_U)_{p(\alpha)} \in Y$. Therefore $Y_{\mathcal{I}(Y)} \subseteq Y$. 
	
	Now suppose that $\alpha \in Y$. By (\hyperref[D3]{D3}) there exists some $f \in A_R(G)$ such that $f_{p(\alpha)} = \alpha$ and $(f*\bm{1}_C)_u \in Y$ for all $u \in G^{(0)}$ and $C \in \Bisc(G^{(0)})$. Now let $A, B \in \Bisc(G)$ be arbitrary. If $(\bm{1}_A*f*\bm{1}_B)_u = 0$ then it is in $Y$ because every fibre of $p: Y \to G^{(0)}$ is an ideal and therefore contains the zero of $RG_u^u$. On the other hand, if $(\bm{1}_A* f * \bm{1}_B)_u \ne 0$ then $u \in \ran(A) \cap \src(B)$. Letting $A \cap G^u = \{t\}$, we have 
	\[
	(\bm{1}_A * f * \bm{1}_B)_u = (\bm{1}_A* (f * \bm{1}_{BA}) * \bm{1}_{A^{-1}})_u = \tau_t\big((f*\bm{1}_{BA})_{\src(t)}\big) \in Y
	\]
	because $(f*\bm{1}_{BA})_{\src(t)} \in Y$, and because of (\hyperref[D2]{D2}). Therefore $f \in \mathcal{I}(Y)$ and $\alpha \in Y_{\mathcal{I}(Y)}$. The conclusion is that $Y \subseteq Y_{\mathcal{I}(Y)}$.
\end{proof}

We have concluded the proof of Theorem \ref{disassembly-theorem} by showing that the reassembly map (\ref{reassembly-map}) maps $\Delta_{G,R}$ bijectively to $\mathcal{L}(A_R(G))$ and that this is the inverse of the disassembly map (\ref{disint-map}). We end the section with an observation that will be useful later.

\begin{proposition} \label{d3-prop}
		Let $Y \subseteq X(G)$ be an open subset that satisfies {\rm (\hyperref[D1]{D1})} and {\rm (\hyperref[D2]{D2})}. If for every $u \in G^{(0)}$, the ideal $Y \cap RG_u^u \trianglelefteq RG_u^u$ is of the form $J_u RG_u^u$ for some ideal $J_u \trianglelefteq R$, then $Y_{\mathcal{I}(Y)} = Y$. Consequently, if $Y$ has this property then it also satisfies {\rm (\hyperref[D3]{D3})}.
\end{proposition}
\begin{proof}
	When we proved that $Y_{\mathcal{I}(Y)} \subseteq Y$, we did not use (\hyperref[D3]{D3}), so we have this fact already. With the assumptions above, $Y \cap RG_u^u$ is generated as an ideal by elements of the form $\alpha = j1_u$ for $j \in J_u$. If we can prove that $\{j1_u \mid u \in G^{(0)}, j \in J_u\}\subseteq Y_{\mathcal{I}(Y)}$, then it implies $Y \subseteq Y_{\mathcal{I}(Y)}$. 
	
	Let $j1_u \in Y$ where $u \in G^{(0)}$ and $0 \ne j \in J_u$. Since $Y$ is open, we can find an $h \in A_R(G)$ and a $U \in \Bisc(G^{(0)})$ such that $u \in U$, $h_u = j1_u$, and $s_h(U) \subseteq Y$. Since $j1_u = (j\bm{1}_U)_u$, Lemma~\ref{fu=gu} implies that there is a $V \in \Bisc(G^{(0)})$ with $u\in V \subseteq U$ such that $\bm{1}_V * h * \bm{1}_V = j\bm{1}_{V}$. Then $s_{j\bm{1}_{V}}(V)  = s_{\bm{1}_V*h*\bm{1}_V}(V) = s_h(V) \subseteq s_h(U) \subseteq Y$. Therefore $(j\bm{1}_V)_v = j1_v \in Y \cap RG_v^v$ and $j \in J_v$, for all $v \in V$. By the translation-invariant property (\hyperref[D2]{D2}), $j \in J_w$ for all $w \in \src(\ran^{-1}(V))$.

	Now let $f = j\bm{1}_V$ and take $w \in G^{(0)}$ and $ C, D \in \Bisc(G)$ to be arbitrary. Then $(\bm{1}_C * f * \bm{1}_D)_w = (j\bm{1}_{CVD})_w = 0$ or $j1_t$ according as $CVD \cap G_w^w = \varnothing$ or $\{t\}$. In the first case, $(\bm{1}_C * f * \bm{1}_D)_w = 0 \in Y$. In the second case, $t = cvd \in G_w^w$ for some $v \in V$, $c \in C$, and $d \in D$, which implies that $w\in \src(\ran^{-1}(V))$ and therefore $j \in J_w$. This yields that $j1_t \in J_w RG_w^w \subseteq Y$. Therefore $f \in \mathcal{I}(Y)$ and $j1_u = f_u \in Y_{\mathcal{I}(Y)}$.
	\end{proof}

\subsection{The special case of strongly effective groupoids}

An ample groupoid $G$ is called \emph{effective} if the interior of $\Iso(G)$ is $G^{(0)}$. Given a subset $V \subseteq G^{(0)}$, we define $VGV = \ran^{-1}(V) \cap \src^{-1}(V)$. The groupoid $G$ is  \emph{strongly effective} if $VGV$ is effective for every nonempty closed invariant subset $V \subseteq G^{(0)}$.

\begin{proposition} \label{se-prop1}
If $G$ is strongly effective, then every open subset $Y \subseteq X(G)$ that satisfies {\rm (\hyperref[D1]{D1})} and {\rm (\hyperref[D2]{D2})} is of the form $\bigcup_{u \in G^{(0)}} \rho(u)  RG_u^u$ for some ideals $\rho(u) \trianglelefteq R$.
\end{proposition}

\begin{proof}
Suppose $G$ is strongly effective and $Y \subseteq X(G)$ is an open subset satisfying (\hyperref[D1]{D1}) and (\hyperref[D2]{D2}). Let $\alpha \in Y$. Then $\alpha \in s_f(U) \subseteq Y$ for some $f \in A_R(G)$ and $U \in \Bisc(G^{(0)})$. Replacing $f$ with $\bm{1}_U * f * \bm{1}_U$, we may write \[f = r_0\bm{1}_U + \sum_{i = 1}^{T} s_i \bm{1}_{B_i}\] where $B_i \in \Bisc(G)$, $B_i \cap G^{(0)} = \varnothing$, and $\src(B_i), \ran(B_i) \subseteq U$. By eliminating some $B_i$'s if necessary, we may assume that $p(\alpha) \in \src(B_i)\cap \ran(B_i)$. Applying \cite[Lemma 4.2]{clark2019ideals} to $B_T$, we know that there exists some $N \in \Bisc(G)$ such that $p(\alpha) \in \src(N)\subseteq U$, $\ran(N) \subseteq U$, and $N^{-1} B_T N = \varnothing$. Thus
\[
\bm{1}_{N^{-1}} * f * \bm{1}_{N} = r_0\bm{1}_{\src(N)} + \sum_{i = 1}^{T-1} s_i \bm{1}_{N^{-1}B_T N}
\]
and (\hyperref[D2]{D2}) implies that $0 \ne (\bm{1}_{N^{-1}} * f * \bm{1}_N)_u \in Y$. Applying this process inductively yields the conclusion that there exists $N\subseteq UGU$ with $p(\alpha) \in \src(N)$ and $(r_0 \bm{1}_{\src(N)})_{p(\alpha)}  = r_0{1}_{p(\alpha)} \in Y$. Every $\alpha \in X(G)$ can be written in a unique way as $\alpha  = r_0 1_{p(\alpha)} + \sum s_i 1_{z_i}$ for $r_0, s_i \in R$ and $z_i \in G_{p(\alpha)}^{p(\alpha)} \setminus \{p_{\alpha}\}$, and we have proved that $\alpha \in Y$ implies $r_0 1_{p(\alpha)} \in Y$. Moreover, $\alpha z_i^{-1} = s_i1_{p(\alpha)} + r_0 1_{z_i^{-1}} + \sum_{j \ne i} s_i 1_{z_j z_i^{-1}} \in Y$ and applying the same reasoning we have $s_i1_{p(\alpha)} \in Y$ for all $i$.

If we let $\rho(u) = \{\alpha(z) \mid  z \in G_u^u, \alpha \in Y\cap RG_u^u\}$, then clearly $\rho(u)$ is an ideal in $R$. The previous paragraph proves that $Y \cap RG_u^u = \rho(u) RG_u^u$ for all $u \in G^{(0)}$.
\end{proof}

Propositions \ref{d3-prop} and \ref{se-prop1} imply the following.

\begin{corollary} \label{se-d3}
	If $G$ is strongly effective, then every open subset $Y \subseteq X(G)$ that satisfies {\rm (\hyperref[D1]{D1})} and {\rm (\hyperref[D2]{D2})} also satisfies  {\rm (\hyperref[D3]{D3})}.
\end{corollary}

We can now give a new proof of \cite[Theorem 5.4]{clark2019ideals} using the general framework we have developed. The authors describe a topology on $\mathcal{L}(R)$  generated by the sets
	\[
	Z(F) = \{I \in \mathcal{L}(R) \mid F \subseteq I \}
	\]
	as $F$ ranges over finite subsets of $R$. They show in \cite[Lemma 5.2]{clark2019ideals} that a map $\rho: G^{(0)} \to \mathcal{L}(R)$ is continuous at a point $u \in G^{(0)}$ if and only if for all $a \in \rho(u)$ there exists an open neighbourhood $W \subseteq G^{(0)}$ of $u$ such that $a \in \rho(w)$ for all $w \in W$. A map $\rho:G^{(0)} \to \mathcal{L}(R)$ is defined to be \emph{$G$-invariant} if it is constant on orbits.

\begin{theorem}\cite[Theorem 5.4]{clark2019ideals}
Let $G$ be strongly effective. There is a one-to-one correspondence between $\mathcal{L}(A_R(G))$ and the set of continuous $G$-invariant maps $\rho: G^{(0)} \to \mathcal{L}(R)$, defined by sending $\rho$ to the ideal \begin{equation} \label{clark-ideal}
I = \Span_R\Big\{r\bm{1}_B \mid B \in \Bisc(G),\ r \in \bigcap_{u \in \src(B)}\rho(u) \Big\}.	
 \end{equation}
	\end{theorem}
	
\begin{proof}
	Combining Theorem \ref{disassembly-theorem} with Corollary \ref{se-d3}, we know that
	\begin{align*}
	I \mapsto Y_I = \bigcup_{u \in G^{(0)}} I_u \subseteq X(G)
	\end{align*}
	is a bijection from $\mathcal{L}(A_R(G))$ to the set of open subsets of $X(G)$ satisfying (\hyperref[D1]{D1}) and (\hyperref[D2]{D2}). We show that the set of open subsets of $X(G)$ satisfying (\hyperref[D1]{D1}) and (\hyperref[D2]{D2}) is in bijection with the set of continuous $G$-invariant maps $\rho: G^{(0)} \to \mathcal{L}(R)$. 
	 Proposition \ref{se-prop1} tells us that every such subset of $X(G)$ is of the form $Y = \bigcup_{u \in U} \rho(u)RG_u^u$ for a unique $\rho(u) \in \mathcal{L}(R)$.
	 	The map $\rho: G^{(0)} \to \mathcal{L}(R)$ is $G$-invariant because $Y$ satisfies (\hyperref[D2]{D2}). If $a \in \rho(u)$ then $a1_u \in Y$, so there exists $f \in A_R(G)$ and $U \in \Bisc(G^{(0)})$ with $a1_u \in s_f(U)\subseteq Y$.  The set $W = f^{-1}(a) \cap U$ is an open neighbourhood of $u$ with $a \in \rho(w)$ for all $w \in W$. Therefore $\rho$ is continuous.
	
	Conversely, given a continuous $G$-invariant map $\rho: G^{(0)} \to \mathcal{L}(R)$, we can define \[Y = \bigcup_{u \in G^{(0)}}\rho(u)RG_u^u \subseteq X(G).\] Clearly $Y$ satisfies (\hyperref[D1]{D1}) and (\hyperref[D2]{D2}). We claim that it is open in $X(G)$. Indeed, let $\alpha \in Y$ and $u = p(\alpha)\in G^{(0)}$.  Then $\alpha = \sum r_i 1_{z_i}$ for some $r_i \in \rho(u)$ and $z_i\in G_u^u$. For each $i$ we can pick $X_i \in \Bisc(G)$ containing $z_i$ in such a way that all of the $X_i$'s are disjoint. Letting $f = \sum r_i \bm{1}_{X_i}$, we have $f_u = \alpha$. Since $\rho$ is continuous at $u$, there are open neighbourhoods $W_i \in \Bisc(G^{(0)})$ of $u$ such that $r_i \in \rho(w)$ for all $w \in W_i$. Let $W = \bigcap W_i$. For any $w \in W$, we have $f_w \in \rho(w) RG_w^w$, so $s_f(W)$ is an open subset of $Y$ containing $\alpha$. Therefore $Y$ is open.
	
	It is straightforward to check that $Y_I = \bigcup_{u \in G^{(0)}}\rho(u) RG_u^u$ where $I$ is as in (\ref{clark-ideal}), which completes the proof.
\end{proof}

\section{The lattice and multiplicative structure of ideals} \label{sec:lattice structure}
 The image of the disassembly map \[\Delta_{G,R} = \big\{Y_I \subseteq X(G) \mid I \in \mathcal{L}(A_R(G))\big\}\]
  is partially ordered by inclusion. Since the disassembly map (\ref{disint-map}) and its inverse (\ref{reassembly-map}) are order-preserving, the partially ordered set $\Delta_{G,R}$
is in fact a lattice and it is isomorphic to the lattice $\mathcal{L}(A_R(G))$. The next proposition confirms that the join and meet operations on the lattice $\Delta_{G,R}$ are given by:
\begin{align*}
Y_1\vee Y_2 = \bigcup_{u \in G^{(0)}} Y_1\cap RG_u^u+ Y_2\cap RG_u^u, && 
Y_1\wedge Y_2 = Y_1\cap Y_2.
\end{align*}
The third item in the proposition, describing the multiplicative behaviour of the disassembly map, is quite surprising.
\begin{proposition} \label{product of ideals of isotropy groups}
 For all ideals \(I,J\trianglelefteq A_R(G)\) and for all \(u\in G^{(0)}\),\begin{align} (I+J)_u&=I_u+J_u;\label{additivity} \\ (I\cap J)_u&=I_u\cap J_u; \label{intersectionality} \\ (IJ)_u &= I_u J_u. \label{multiplicativity}\end{align}
\end{proposition}
\begin{proof}
(\ref{additivity}) 
We have $(I+J)_u =  \{(f+g)_u \mid f \in I, g \in J\} = \{f_u + g_u \mid f \in I, g \in J\} = I_u + J_u$.

(\ref{intersectionality})
It is clear that $(I \cap J)_u \subseteq I_u \cap J_u$. On the other hand, if $\alpha \in I_u \cap J_u$ then $\alpha = f_u = g_u$ for some $f \in I$ and $g \in J$. By Lemma \ref{fu=gu}, there exist $V, W \in \Bisc(G^{(0)})$ containing $u$ such that $\bm{1}_V * f * \bm{1}_W = \bm{1}_V * g * \bm{1}_W\in I \cap J$, and therefore $\alpha = (\bm{1}_V * f * \bm{1}_W)_u \in (I \cap J)_u$.

(\ref{multiplicativity}) First, suppose we have some \(\alpha\in I_u\) and \(\beta\in J_u\). Let \(f\in I\) and \(g\in J\) be representatives of \(\alpha\) and \(\beta\) respectively. As in the proof of Lemma \ref{disint-lem1}, we can find a neighbourhood $K \in \Bisc(G^{(0)})$ such that $ \alpha \cdot \beta (x) = f*\bm{1}_K * g(x)$ for all $x \in G_u^u$, so $\alpha \cdot \beta = ((f* \bm{1}_K)*g)_u \in (IJ)_u$. Therefore $I_u J_u \subseteq (IJ)_u$. Conversely, suppose we have some $f = \sum_i r_i \bm{1}_{B_i} \in I$ and $g = \sum_j s_j \bm{1}_{C_j} \in J$, where $B_i, C_j \in \Bisc(G)$. Then $f*g = \sum_i \sum_j r_is_j \bm{1}_{B_i C_j}$. Note that $f*g (x) \ne 0$ implies $x = bc$ for some $i,j$, $b \in B_i$, and $c \in C_j$ with $\ran(c) = \src(b)$.  Let $F\subseteq G^{(0)}$ be the finite set of all units $\ran(c) = \src(b)$ arising in this way. For each $v \in F$, pick any $D_v \in \Bisc(G)$ with $D_v \cap G_u^v \ne \varnothing$. Let $f' = f* \sum_{v \in F} \bm{1}_{D_v}$ and $g' = \sum_{v \in V} \bm{1}_{D_v^{-1}}* g$. Then $f_u' \in I_u$ and $g_u' \in J_u$, and it is straightforward to verify that $(f*g)_u = f_u' \cdot g_u'\in I_u J_u$. Therefore $(IJ)_u \subseteq I_u J_u$.
\end{proof}

An \emph{arithmetical ring} is a ring $A$ whose lattice of ideals is distributive; that is, $I \cap (J + K) = I \cap J + I \cap K$ for all ideals $I,J,K \in \mathcal{L}(A)$. This is equivalent to $I + (J \cap K) = (I + J) \cap (I + K)$ for all ideals $I, J, K \in \mathcal{L}(A)$. 

\begin{corollary} \label{commuting and distributing ideals}  \begin{enumerate}[\rm (1)]
\item 	If all the isotropy groups of $G$ are abelian, then $IJ = JI$ for all $I, J \in \mathcal{L}(A_R(G))$. 	
\item If $RG_u^u$ is an arithmetical ring for all $u \in G^{(0)}$, then $A_R(G)$ is an arithmetical ring.
\end{enumerate}
\end{corollary}
A theorem of  Tuganbaev \cite[Theorem 12.57]{tuganbaev} is a good companion to this corollary, as it gives a characterisation of abelian groups $\Gamma$ and rings $R$ such that $R\Gamma$ is arithmetical.

\subsection{Prime and semiprime Steinberg algebras}

Recall that a ring $A$ is \emph{prime} if $IJ = 0$ implies $I = 0$ or $J = 0$ for all ideals $I, J \trianglelefteq  A$, and $A$ is \emph{semiprime} if $I^2 = 0$ implies $I = 0$ for all ideals $I\trianglelefteq A$.

A groupoid $G$ is called \emph{topologically transitive} if for every pair of nonempty open subsets $U, V \subseteq G^{(0)}$ the set $\ran^{-1}(U) \cap \src^{-1}(V)$ is nonempty. In particular, every transitive groupoid is topologically transitive. Steinberg's \cite[Proposition 3.3]{steinberg-prime} shows that $G$ is topologically transitive if and only if $G^{(0)}$ is not a union of two proper closed invariant subsets. It is known that $G$ is topologically transitive if $A_R(G)$ is prime \cite[Proposition 4.3]{steinberg-prime}. Sufficient conditions for primeness are, however, much more delicate.

\begin{lemma} \label{U1U0}
    If $I$ is an ideal in $A_R(G)$, then: \begin{enumerate}[\rm (1)] \item For any $r \in R$, the set $U_r = \{u \in G^{(0)} \mid r1_u \in I_u\}$ is open and invariant in $G^{(0)}$. In particular, $U_1 = \{u \in G^{(0)} \mid I_u = RG_u^u\}$ is open and invariant.
		\item The set $U_{\rm null} = \{u \in G^{(0)} \mid I_u = 0\}$ is closed and invariant in $G^{(0)}$.
	\end{enumerate}
\end{lemma}

\begin{proof}
	(1) Suppose $u \in U_r$. Then $r{1}_u = f_u$ for some $f \in I$, and $f_u = (r\bm{1}_W)_u$ for every $W \in \Bisc(G^{(0)})$ containing $u$. By Lemma~\ref{fu=gu}, we have $\bm{1}_V * f * \bm{1}_V = r\bm{1}_V$ for some $V \in \Bisc(G^{(0)})$ containing $u$. Then $r\bm{1}_V \in I$ and $u \in V \subseteq U_r$, so $U_r$ is open. 
	
	(2)
	Fix $f \in A_R(G)$. We claim that $(f*\bm{1}_C)_u = 0$ for all $C \in \Bisc(G)$ if and only if $f*\bm{1}_C(u) = 0$ for all $C \in \Bisc(G)$.
	Indeed, suppose that $u \in G^{(0)}$ is such that $f * \bm{1}_C(u) = 0$ for all $C \in \Bisc(G)$. Let $x \in G_u^u$, and choose some $D \in \Bisc(G)$ containing $x$. Then for all $C \in \Bisc(G)$, we have
	$
	f * \bm{1}_C (x) = f* \bm{1}_{CD^{-1}}(u) = 0.
	$
	With this observation in hand, we can compute:
	\begin{align*} \label{U0}
	U_{\rm null} &= \big\{u \in G^{(0)} \mid I_u = 0 \big\} = \bigcap_{f \in I } \big\{u \in G^{(0)} \mid f_u = 0\big\} = \bigcap_{f \in I} \bigcap_{C \in \Bisc(G)} \big\{u \in G^{(0)} \mid (f*\bm{1}_C)_u = 0\big\} \\&= \bigcap_{f \in I} \bigcap_{C \in \Bisc(G)} \big\{u \in G^{(0)} \mid f*\bm{1}_C(u) = 0\big\} = \bigcap_{f \in I} \bigcap_{C \in \Bisc(G)} (f*\bm{1}_C)^{-1}(0) \cap G^{(0)}.
	\end{align*}
	Since $G$ is Hausdorff, both $G^{(0)}$ and $(f*\bm{1}_C)^{-1}(0)$ are closed in $G$, so $G^{(0)} \cap (f*\bm{1}_{C})^{-1}(0)$ is closed in $G$, and hence closed in $G^{(0)}$. Being an intersection of closed sets, $U_{\rm null}$ is closed in $G^{(0)}$.
	
	All the $U_r$'s and $U_{\rm null}$ are invariant because there is a translation $\tau_t$ that maps $I_u$ isomorphically to $I_v$, whenever $u$ and $v$ are in the same orbit.
\end{proof}

In the realm of Hausdorff ample groupoids, we are able to specify a weaker condition than the one in \cite[Theorem 4.9]{steinberg-prime} and show that it is still strong enough to imply primeness. We also give a streamlined proof of \cite[Theorem 4.10]{steinberg-prime}.

\begin{theorem}
	\begin{enumerate}[\rm (1)]
		\item If $G$ is topologically transitive and there is a dense subset $D \subseteq G^{(0)}$ such that $RG_d^d$ is prime for every $d \in D$, then $A_R(G)$ is prime.
		\item If there is a dense subset $D \subseteq G^{(0)}$ such that $RG_d^d$ is semiprime for every $d \in D$, then $A_R(G)$ is semiprime.
	\end{enumerate}
\end{theorem}
\begin{proof}
	(1) Suppose that $IJ = 0$ for some ideals $I, J \trianglelefteq  A_R(G)$. Then $(IJ)_u = I_u J_u = 0$ for all $u \in G^{(0)}$. Then for every $d \in D$ we have either $I_u = 0$ or $J_u = 0$. Let $K_1 = \{u \in G^{(0)} \mid I_u = 0\}$, and $K_2 = \{u \in G^{(0)} \mid J_u = 0\}$.  These sets are closed, by Lemma \ref{U1U0} (2). But $D \subseteq K_1 \cup K_2$ and $D$ is dense, so $G^{(0)} = \overline D \subseteq \overline {K_1 \cup K_2} = K_1 \cup K_2 \subseteq G^{(0)}$. Since $G$ is topologically transitive, $G^{(0)}$ cannot be a union of two proper closed invariant subsets. Clearly $K_1$ and $K_2$ are invariant, so one of them is equal to all of $G^{(0)}$. This implies that $I = 0$ or $J = 0$.
	
	(2) Suppose $I^2 = 0$ for some ideal $I \trianglelefteq A_R(G)$. Then $(I^2)_u = (I_u)^2 = 0$ for all $u \in G^{(0)}$, and we have $I_d = 0$ for all $d \in D$. Let $K = \{u \in G^{(0)} \mid I_u = 0 \}$. By Lemma \ref{U1U0} (2), $K$ is closed. Now $D \subseteq K$ and $G^{(0)} = \overline D \subseteq \overline K = K\subseteq G^{(0)}$. Therefore $I_u = 0$ for all $u \in G^{(0)}$, and this implies $I = 0$.
\end{proof}

The theorems of Connell and Passman \cite[Ch.\! 4 \S10]{Lam} are good companions to the above theorem because they characterise the rings $R$ and groups $\Gamma$ such that $R\Gamma$ is prime or semiprime.

\section{Leavitt path algebras} \label{sec:Leavitt path algebras}

Leavitt path algebras are algebras associated to directed graphs, and are a special kind of Steinberg algebras. We will give a brief introduction on their definition and how they become Steinberg algebras, but we refer to \cite{LPAbook,rigby} and the introduction for more information, applications, and motivations. The  purpose of this section is to specialise the contents of \S\ref{section:disassembly theorem} to boundary path groupoids, moving from a topological setting towards a more combinatorial setting and simplifying some aspects of the theory where it is possible to do so.

\subsection{Preliminaries on Leavitt path algebras}

To start, a \emph{directed graph} \(E\) is a quadruple \((E^0, E^1, r, s)\) consisting of a set \(E^0\) of \emph{vertices}, a set \(E^1\) of \emph{edges}, and two functions \(r,s: E^1\to E^0\) assigning to each vertex its \emph{range} and \emph{source} respectively. A vertex \(v\) is a \emph{sink} if it doesn't emit any edges (i.e.\! if \(s^{-1}(v)=\varnothing\)),   and an \emph{infinite emitter} if it emits infinitely many edges. In all other cases, \(v\) is a \emph{regular vertex}, and the set of all of these is denoted by $\operatorname{Reg}(E) \subseteq E^0$.

A \emph{path} \(\mu\) of length \(|\mu|>0\) is a finite sequence \(\mu=e_1 e_2\cdots e_{|\mu|}\) of edges such that \(r(e_i)=s(e_{i+1})\) for \(i = 1, \dots, \vert \mu \vert -1\). Vertices are considered to be paths of length 0. The set \(\{s(e_1), r(e_1), \ldots, r(e_{|\mu|})\}\) of vertices of \(\mu\) is denoted by \(\mu^0\), and the source and range of \(\mu\) are defined as \(s(\mu)=s(e_1)\) and \(r(\mu)=r(e_{|\mu|})\). For two paths \(\mu\) and \(\nu\) such that \(r(\mu)=s(\nu)\), we can define their \emph{concatenation} \(\mu\nu\) in the obvious way. The set of all paths will be denoted by~\(E^*\). We say \(\mu\in E^*\) is \emph{closed} if \(|\mu|>0\) and \(s(e_1)=r(e_{|\mu|})\). A closed path is called \emph{simple}, if \(s(e_1)=s(e_i)\) implies \(i=1\), and a \emph{cycle} if \(s(e_i)=s(e_j)\) implies \(i=j\). Clearly, every cycle is simple. An exit for a path \(e_1\cdots e_n\) is an edge \(e \in E^1\) for which there exists \(i\in \{1, \ldots, n\}\) such that \(s(e)=s(e_i)\) but \(e\neq e_i\). 
For each \(e\in E^1\), we formally define a \emph{ghost edge} \(e^*\), which we think about as an edge with the same endpoints as \(e\), but going in the opposite direction. The set of ghost edges is denoted by \((E^1)^*\), and for each path \(\mu =e_1\cdots e_n\) we have a corresponding sequence \(\mu^*=e_n^* \cdots e_1^*\), called a \emph{ghost path}.

\subsubsection{The classical definition: generators and relations}

\begin{definition} 
The \emph{Leavitt path algebra} \(L_R(E)\) of a directed graph $E$ over the ring \(R\) is the universal associative \(R\)-algebra generated by the set  \(E^0\cup E^1 \cup (E^1)^*\), satisfying the following relations: 
\begin{itemize}
    \item[(V)] \(vw=\delta_{vw}v\),
    \item[(E1)] \(s(e)e=e=er(e) \),    \item[(E2)] \(r(e)e^*=e^*=e^*s(e)\),    \item[(CK1)] \(e^*f=\delta_{ef}r(e) \),
         \item[(CK2)] \(u=\sum_{e\in s^{-1}(u)}ee^*\) 
\end{itemize}
for all $v, w \in E^0$, $e, f \in E^1$, and $u \in \operatorname{Reg}(E)$. 
\end{definition}
Note that every element of \(L_R(E)\) can be written (non-uniquely) as \(\sum_i r_i \alpha_i \beta_i^*\) for paths \(\alpha_i, \beta_i\). It was also proved in \cite[Proposition 3.4]{tomforde2011leavitt} that the elements in \(E^0\cup E^1\cup (E^1)^*\) are all nonzero. Every Leavitt path algebra is \(\ZZ\)-graded, with the grading induced by letting \(\deg(v)=0\) for \(v\in E^0\) and \(\deg(e)=1\) and \(\deg(e^*)=-1\) for \(e\in E^1\). So \(L_R(E)=\bigoplus_{n\in \ZZ} L_n\) with \(L_n=\{\sum_i r_i\alpha_i \beta_i^*\mid |\alpha_i|-|\beta_i|=n\}\) and $L_nL_m \subseteq L_{n+m}$ for all $n, m \in \ZZ$. An ideal \(I\trianglelefteq L_R(E)\) is \emph{graded} if \(I=\bigoplus_{n\in \ZZ} (I\cap L_n)\).

A subset \(H\subseteq E^0\) is called \emph{hereditary} if for each path \(\mu\), \(s(\mu)\in H\) implies \(r(\mu)\in H\), and \(H\) is called \emph{saturated} if for each regular vertex \(v\), \(r(s^{-1}(v))\subseteq H\) implies \(v\in H\). If \(H\) is both hereditary and saturated, we say \(H\) is \emph{hereditary saturated}. For any subset \(K\subseteq E^0\), we can define the \emph{hereditary saturated closure} \(\overline{K}\) of \(K\) as the smallest hereditary saturated subset of \(E^0\) containing \(K\).  See \cite[Lemma 2.0.7]{LPAbook} for an explicit construction of this hereditary saturated closure; alternatively, it is given by Formula (\ref{S-saturation}) with \(S = \varnothing\).
If \(H\subseteq E^0\) is a hereditary subset, a \emph{breaking vertex} for \(H\) is an infinite emitter \(v\in E^0\setminus H\), with \(0< |s^{-1}(v)\cap r^{-1}(E^0\setminus H)| <\infty\). In words, a breaking vertex is an infinite emitter not in \(H\), such that a finite but nonzero number of edges it emits have their range not in \(H\). For such a breaking vertex \(v\), we define an element \(v^H\in L_R(E)\), by
	\begin{equation} \label{notation:v^H}
	v^H = v- \sum_{\substack{e\in s^{-1}(v)\\ r(e)\notin H}} ee^*.
	\end{equation}

We have already defined paths in \(E\), which are of finite length by definition. To relate Leavitt path algebras to Steinberg algebras, we will also need the concept of an \emph{infinite path}, which is an infinite sequence \(e_1e_2e_3\ldots\) such that \(r(e_i)=s(e_{i+1})\) for \(i\geq 1\). Unless we specify otherwise, a \emph{path} will always mean a finite path. For a path \(\alpha\) and a (possibly infinite) path \(x\), we say \(\alpha\) is an \emph{initial subpath} of \(x\) if there exists a (possibly infinite) path \(x'\) such that \(x=\alpha x'\).

\subsubsection{The Steinberg algebra model} \label{sec:Steinberg algebra model}
For each graph \(E\), we want to define a Hausdorff ample groupoid \(G_E\) such that \(L_R(E)\cong A_R(G_E)\). This is done as follows: let \(\partial E\) be the set of infinite paths, together with the finite paths that end in a singular vertex (i.e.\! a sink or infinite emitter). These are called the \emph{boundary paths}. These will be the units of \(G_E\). For two boundary paths \(x,y\in \partial E\) and \(k\in \ZZ\), we say \(x\) and \(y\) are \emph{tail equivalent of lag \(k\)} if there exist \(\alpha, \beta \in E^*\) and \(z\in \partial E\) such that \(x=\alpha z\), \(y=\beta z\) and \(|\alpha|-|\beta|=k\). Note that two boundary paths can be tail equivalent of multiple lags. Clearly, being tail equivalent (without fixed lag) is an equivalence relation. A boundary path is called \emph{rational} if it is infinite and can be written as \(\rho c^\infty= \rho ccc\cdots\), for some \(\rho\in E^*\) and closed path \(c\). An \emph{irrational} path is an infinite path that is not rational.

Next, we have to define a topology on \(\partial E\). For each \(\alpha\in E^*\) and \(F\subseteq_{\text{finite}} s^{-1}(r(\alpha))\), let \(Z(\alpha) = \{\alpha x\in \partial E\mid s(x)=r(\alpha)\}\), and \(Z(\alpha, F) = Z(\alpha)\setminus \bigcup_{e\in F}Z(\alpha e)\). In particular, \(Z(\alpha, \varnothing) = Z(\alpha)\). Note that the collection of such sets, together with the empty set, is closed under finite intersections. We define a topology on \(\partial E\) by letting these sets form a basis. For this topology, \(\partial E\) is Hausdorff, locally compact, totally disconnected, and each \(Z(\alpha, F)\) is open and compact.

Now we can define the \emph{boundary path groupoid} \(G_E\) as the set of triples \((x,k,y)\in \partial E \times \mathbb{Z} \times \partial E\) where $x$ and $y$ are tail equivalent of lag $k$. The source and range maps are defined as \(\src((x,k,y))=y\) and \(\ran((x,k,y)) = x\) respectively, and composition and inversion are given by \((x,k,y)\cdot(y,l,z) = (x,k+l,z)\) and \((x,k,y)^{-1}=(y,-k,x)\). The unit space \(G_E^{(0)}\) is \(\{(x,0,x)\mid x\in \partial E\}\), and it can clearly be identified with \(\partial E\).    It is also not hard to see that the isotropy group based at \(u \in \partial E\) is infinite cyclic if \(u\) is rational, and trivial if \(u\) is irrational or finite.

\begin{notation}
From now on, we will usually identify \(\partial E\) and \(G_E^{(0)}\). We will clarify however, which notation we will use in what circumstances. If we pick a boundary path or do something that only works for units, such as taking isotropy groups or disassembling ideals, we will use \(u\in \partial E\). If we use an element of \(G_E\) that happens to be a unit, even when the isotropy group at that unit is trivial, we will use \((u,0,u)\in G_E^{(0)}\).
\end{notation}

For each \(\alpha, \beta\in E^*\) such that \(r(\alpha)=r(\beta)\) and \(F\subseteq_{\text{finite}} s^{-1}(r(\alpha))\), define 
\[\mathcal{Z}(\alpha, \beta) = \{(\alpha x, |\alpha|-|\beta|, \beta x)\mid x\in \partial E, s(x)=r(\alpha)\}
\] 
and 
\[
\mathcal{Z}(\alpha, \beta, F) = \mathcal{Z}(\alpha, \beta)\setminus \bigcup_{e\in F} \mathcal{Z}(\alpha e, \beta e)\]
Again we have \(\mathcal{Z}(\alpha, \beta, \varnothing) = \mathcal{Z}(\alpha, \beta)\), and the collection of these sets together with the empty set is closed under finite intersections. As with \(\partial E\), we give \(G_E\) the topology generated by such sets \(\mathcal{Z}(\alpha, \beta, F)\). For this topology, \(G_E\) is a Hausdorff ample groupoid and each set \(\mathcal{Z}(\alpha, \beta, F)\) is a compact open bisection \cite[Theorem 2.4]{rigby}. Note that for each \(\alpha\in E^*\) and \(F\subseteq_{\text{finite}} s^{-1}(r(\alpha))\), we have \(\mathcal{Z}(\alpha, \alpha, F) = Z(\alpha, F)\) once $G_E^{(0)}$ is identified with $\partial E$, so that identifying them is indeed a homeomorphism.

Lastly, we want to find an isomorphism \(L_R(E)\overset{\sim}{\to} A_R(G_E)\). This is done as follows: for each \(\alpha, \beta\in E^*\), map \(\alpha\beta^*\in L_R(E)\) to \(\bm{1}_{\mathcal{Z}(\alpha, \beta)}\in A_R(G_E)\). This map can be extended uniquely to an \(R\)-algebra homomorphism \(L_R(E)\to A_R(G_E)\), and \cite[Theorem 2.7]{rigby} shows that this is an isomorphism. In particular, \(A_R(G_E)\) is generated as an \(R\)-module by the set \(\{\bm{1}_{\mathcal{Z}(\alpha,\beta)} \mid \alpha,\beta\in E^*, r(\alpha) = r(\beta)\}\). Note that if \(v\in E^0\) is a breaking vertex for the hereditary subset \(H\subseteq E^0\), then under this isomorphism \(v^H\) corresponds to \(\bm{1}_{\mathcal{Z}(v,v,F)}\) where  \(F=s^{-1}(v)\cap r^{-1}(E^0\setminus H)\).

We will always assume that any summation is finite, i.e.\!  only finitely many terms are nonzero. Moreover, when working with rational paths, we will use the following notation:
\begin{notation}
Let \(\rho c^\infty\) be a rational path. Then we have \(\mathcal{Z}(\rho c, \rho)\mathcal{Z}(\rho c, \rho) = \mathcal{Z}(\rho c^2, \rho)\). We also have \(\mathcal{Z}(\rho c, \rho)^{-1} = \mathcal{Z}(\rho, \rho c)\). Consequently, for \(n\in \ZZ\setminus \{0\}\), we have \(\mathcal{Z}(\rho c, \rho)^n = \mathcal{Z}(\rho c^n, \rho)\) if \(n>0\) and \(\mathcal{Z}(\rho c, \rho)^{n} = \mathcal{Z}(\rho, \rho c^{-n})\) if \(n<0\). We will extend this notation to \(n\in \ZZ\), by letting \(\mathcal{Z}(\rho c, \rho)^0 = \mathcal{Z}(\rho, \rho)\).
\end{notation}

\subsubsection{Multiplicative ideal theory for Leavitt path algebras}

Using this Steinberg algebra model for Leavitt path algebras, we can already find some properties of the ideal lattice of \(L_R(E)\). It has been established (e.g., in \cite[Corollary 2.8.17]{LPAbook}) that  ideal multiplication in Leavitt path algebras over fields is commutative. In Proposition \ref{multiplication of Leavitt ideals}, we will give a visibly commutative formula for multiplying ideals in Leavitt path algebras over arbitrary rings. However, note that it follows already from Corollary \ref{commuting and distributing ideals} that ideal multiplication is commutative.
 
Even more surprisingly, in \cite[Theorem 4.3]{rangaswamy} it was proved that Leavitt path algebras over fields are arithmetical rings.  Corollary \ref{commuting and distributing ideals} provides a groupoid-theoretic explanation for this phenomenon: if $R$ is von Neumann regular and all of $G$'s isotropy groups are either trivial or infinite cyclic, then $A_R(G)$ is an arithmetical ring simply because $R$ and $R[x,x^{-1}]$ are \cite[Theorem 12.4]{tuganbaev}.

\begin{example}\label{toeplitz graph groupoid}
To illustrate the concept of a boundary path groupoid, let us consider the \emph{Toeplitz graph}
\[T=
 	\xymatrix{\bullet_u \ar@(ul,dl)_e \ar[r]^f & \bullet_v}.
\]
In this case the elements of \(\partial T\) are given by \(v\), \(e^nf\) for \(n\in \NN\), and \(e^\infty\). Moreover, given the topology described above, the open subsets are exactly the unions of sets of the form \(\{v\}\), \(\{e^nf\}\) for \(n\in \NN\), or \(\{e^\infty, e^nf\mid n\geq m\}\) for \(m\in \NN\).

Similarly, we can describe the elements of \(G_{T}\) as \((v,0,v)\), \((e^nf, n+1,v)\) and the corresponding inverses for \(n\in \NN\), \((e^mf,m-n,e^nf)\) for \((m,n)\in \ZZ^2\) and \((e^\infty, m, e^\infty)\) for \(m\in \ZZ\). In this case the open subsets are unions of sets of the form \(\{(v,0,v)\}\), \(\{(e^nf, n+1, v)\}\) and \(\{(v, -n-1, e^nf)\}\) for \(n\in \NN\), \(\{(e^mf, m-n, e^nf)\}\) for \((m,n)\in \ZZ^2\), or \(\{(e^\infty, m-n, e^\infty), (e^{m+k} f, m-n, e^{n-k}f)\mid k\in \NN\}\) for \((m,n)\in \ZZ^2\). 
\end{example}

Throughout the rest of this paper, we will sometimes go back to this graph to illustrate other concepts.

\subsection{Correspondence between the ideals of \(A_R(G_E)\) and the disassembled ideals} \label{sec:ideal-disassembly correspondence}

For the rest of this paper, let \(E\) be an arbitrary directed graph and let \(G_E\) be its boundary path groupoid. To avoid too cumbersome notation, for \(u\in \partial E\), we will denote the isotropy group of \(G_E\) at \(u\in \partial E\) by \(G_u^u\), instead of \((G_E)_u^u\).

To start this section, we will show how an ideal \(I\trianglelefteq A_R(G_E)\) is characterised by its disassembled ideals \(I_u\), as defined in \eqref{disint-map}. The same characterisations will also hold if we consider \(I\) as an ideal of \(L_R(E)\) under the isomorphism \(L_R(E) \cong A_R(G_E)\). We will only really need these characterisations later on, but keeping these in mind can help to understand why the lattice \(\mathscr{D}_{E,R}\) will be defined as it is. Recall that \(RG_u^u\cong R\) if \(u\) is irrational or finite, and \(RG_u^u\cong R[x, x^{-1}]\) if \(u\) is rational (i.e.\! \(u=\rho c^\infty\) for some paths \(\rho\) and \(c\)). In the latter case, the isomorphism is given by \(\sum_n r_n 1_{(u, n|c|, u)} \mapsto \sum_n r_n x^n\).

\begin{proposition}\label{correspondence: proposition}
Let \(I\trianglelefteq A_R(G_E)\) be an ideal. Then the following statements hold:
\begin{enumerate}[\rm (1)]
    \item \label{correspondence: ends in sink} If \(u\in \partial E\) is finite and ends in a sink, then for each \(r\in R\), we have \(r\notrational{u}\in I_u\) if and only if \(r\bm{1}_{\mathcal{Z}(r(u), r(u))}\in I\). 
    \item \label{correspondence: irrational} Let \(u\in \partial E\) be irrational and \(r\in R\). Then we have \(r\notrational{u} \in I_u\) if and only if there exists a vertex \(v\in E^0\) on \(u\) such that \(r\bm{1}_{\mathcal{Z}(v,v)}\in I\).
    \item \label{correspondence: rational} Let \(u=\rho c^\infty \in \partial E\) be rational, with \(v=s(c)=r(c)\). The Laurent polynomial \(\sum_{n\in \mathbb{Z}} r_n1_{(u,n|c|,u)}\) is an element of \(I_u\) if and only if the function \(\sum_{n\in \ZZ} r_n \bm{1}_{\mathcal{Z}(c, v)^n}\) is an element of \(I\).
    \item \label{correspondence: ends in infinite emitter} If \(u\in \partial E\) is finite and ends in an infinite emitter, we have \(r\notrational{u} \in I_u\) if and only if \(r\bm{1}_{\mathcal{Z}(r(u),r(u),F)}\in I\), for some finite \(F\subseteq s^{-1}(r(u))\), for all \(r\in R\).
    \item \label{correspondence: vertices} For a vertex \(v\) and \(r\in R\), we have \(r\bm{1}_{\mathcal{Z}(v,v)}\in I\) if and only if \(r\notrational{u}\in I_u\) for all \(u\in \partial E\) such that \(v\in u^0\).
\end{enumerate}
\end{proposition}
\begin{proof}
(1) If $u$ ends in a sink, then $r\notrational{u} \in I_u$ if and only if $r\notrational{r(u)} \in I_{r(u)}$, because $u$ and $r(u)$ are in the same orbit. This holds if and only if $r\notrational{r(u)} = f_{r(u)}$ for some $f \in I$, which holds if and only if $r\bm{1}_B \in I$ for some compact open neighbourhood $B$ of $r(u)$, by an application of Lemma~\ref{fu=gu}. But $\mathcal{Z}(r(u), r(u)) = \{r(u)\}$ is contained in every such neighbourhood $B$, so this holds if and only if $r\bm{1}_{\mathcal{Z}(r(u), r(u))} \in I$.

(2) Let $u = e_1 e_2 \dots$ be an irrational path. Since  $\{\mathcal{Z}(e_1\cdots e_n, e_1\cdots e_n) \mid n \ge 1 \}$ is a neighbourhood base for $u$, we have $r\notrational{u} \in I_u$ if and only if $r\bm{1}_{\mathcal{Z}(e_1\cdots e_n, e_1\cdots e_n)} \in I$ for some $n \ge 1$, again by Lemma \ref{fu=gu}. Let $B = \mathcal{Z}(r(e_n),e_1\cdots e_n)$. Then clearly $\bm{1}_{B}* r\bm{1}_{\mathcal{Z}(e_1\cdots e_n, e_1\cdots e_n)} * \bm{1}_{B^{-1}} = r\bm{1}_{\mathcal{Z}(r(e_n), r(e_n))} \in I$ if and only if $\bm{1}_{B^{-1}} *r \bm{1}_{\mathcal{Z}(r(e_n), r(e_n))} * \bm{1}_{B} = r\bm{1}_{\mathcal{Z}(e_1\dots e_n, e_1\cdots e_n)} \in I$. 

(3) Suppose $u = \rho c^\infty$, $v = r(c) = s(c)$, and $\sum_{n \in \ZZ} r_n1_{(u,n\vert c\vert ,u)} \in I_u$. Then there is an $f \in I$ such that $f(u,n\vert c \vert, u) = r_n$ for all $n \in \ZZ$. Applying Lemma \ref{fu=gu}, we may replace $f$ with $\bm{1}_V * f * \bm{1}_V$ for some $V \in \Bisc(G_E^{(0)})$ if necessary, and assume without loss of generality that $f = \sum_{n \in \ZZ} r_n \bm{1}_{B_n} \in I$ for some sets $B_n \in \Bisc(G_E)$ such that $B_n$ is a nonempty neighbourhood of $(u,n \vert c \vert, u)$ if $r_n \ne 0$ and $B_n = \varnothing$ otherwise. We may even assume that the nonempty $B_n$'s are of the form $\mathcal Z(\gamma_n, \delta_n)$ where $\gamma_n, \delta_n$ are initial subpaths of $u$ such that $\vert \gamma_n \vert - \vert \delta_n \vert = n \vert c \vert$, since sets of this kind are a neighbourhood base for $(u, n \vert c \vert, u)$.  Now choose an initial subpath \(y\) of \(u\) such that \(r(y) = v\) and \(y\) is longer than any \(\gamma_{n}\) or \(\delta_{n}\). Then \(\bm{1}_{\mathcal{Z}(v,y)}* f*\bm{1}_{\mathcal{Z}(y,v)} = \sum_{n \in \ZZ} r_n\bm{1}_{\mathcal{Z}(c,v)^n} \in I\).

Conversely, suppose \(\sum_{n\in \ZZ} r_n \bm{1}_{\mathcal{Z}(c,v)^n}\in I\). Then we also have \(\bm{1}_{\mathcal{Z}(\rho,v)}*\sum_{n\in \mathbb{Z}} r_n\bm{1}_{\mathcal{Z}(c,v)^n}*\bm{1}_{\mathcal{Z}(v,\rho)} = \sum_{n \in \ZZ} r_n \bm{1}_{\mathcal{Z}(\rho c, c)^n}\in I\). Since \((u,n|c|,u)\in  \mathcal{Z}(\rho c, \rho)^m\) if and only if \(n=m\), we find that \(f(u,n\vert c\vert,u)=r_n\) for all $n \in \ZZ$. Consequently, \(f_u=\sum_{n\in \mathbb{Z}} r_n 1_{(u,n|c|,u)}\in I_u\).
	
	(4) This is analogous to (2), using the fact that $\{\mathcal{Z}(r(u),r(u),F) \mid F \subseteq_{\rm finite} s^{-1}(r(u)) \}$ is a neighbourhood base of the infinite emitter $r(u)$.
	
	(5) Assume  we have \(r\bm{1}_{\mathcal{Z}(v, v)}\in I\), with $v \in E^0$ and \(r\in R\), and let \(u\in \partial E\) be a path that contains \(v\). Let \(\mu\in E^*\) be the initial subpath of \(u\) ending at the first occurrence of \(v\). Then we have \(\bm{1}_{\mathcal{Z}(\mu,v)} * r\bm{1}_{\mathcal{Z}(v, v)} * \bm{1}_{\mathcal{Z}(v,\mu)} = r\bm{1}_{\mathcal{Z}(\mu, \mu)} \in I\). Since \(u\in Z(\mu)\), it follows that \(r\notrational{u} \in I_u\).

	Now assume that $r\mayberational{u} \in I_u$ for all $u \in \partial E$ such that $v \in u^0$. In particular, $r \mayberational{u} \in I_u$ for all $u \in Z(v)$. Then for each $u\in Z(v)$ there exists a $B_u \in \Bisc(G_E^{(0)})$ with $u\in B_u \subseteq Z(v)$ such that $r\bm{1}_{B_u} \in I$. Since $\{B_u \mid u \in Z(v)\}$ covers $Z(v)$, which is compact, there is a finite subcover $\{B_1, \dots, B_n\}$ also covering $Z(v)$. For compact open subsets $C \subseteq B_i$, we also have $r\bm{1}_{C} = \bm{1}_{C}*r\bm{1}_{B_i} \in I$ and $-r\bm{1}_C \in I$. By the principle of inclusion-exclusion,
	\[
	\bm{1}_{\mathcal{Z}(v, v)} = \bm{1}_{B_1 \cup \dots \cup B_n} = \sum_{k=1}^n (-1)^{i-1} \sum_{\substack{I \subseteq \{1, \dots, n\} \\ |I|=k }}\bm{1}_{\bigcap_{i \in I} B_i},
	\]
	and consequently $r\bm{1}_{\mathcal{Z}(v,v)} \in I$.\end{proof}

\begin{corollary}\label{ideal is generated by vertices and cycles}
Every ideal of \(L_R(E)\) is generated by elements of the following form:
\begin{itemize} \item Scalar multiples of vertices, 
\item Scalar multiples of elements of the form \(v^H = v- \sum_{e\in s^{-1}(v), r(e)\notin H} ee^*\), with \(v\) a breaking vertex for a hereditary subset \(H\subseteq E^0\), 
\item Laurent polynomials of cycles. 
\end{itemize}
\end{corollary}
\begin{proof}
Proposition \ref{correspondence: proposition} and the fact that the disassembly map is injective imply that an ideal~$I$ is uniquely determined by the elements of this form contained in $I$.
\end{proof}

\subsection{Characterising open subsets of \(X(G_E)\)}

The disassembly map gives us a lattice isomorphism between the ideals of \(A_R(G_E)\) and the set $\Delta_{G_E,R}$ of open subsets of \(X(G_E)\) that satisfy three other conditions. For a given subset of \(X(G_E)\) it is relatively easy to see if it satisfies (\hyperref[D1]{D1}) and (\hyperref[D2]{D2}), but harder to check (\hyperref[D3]{D3}) and openness. As such, we will use this subsection to find an equivalent formulation of openness, given a subset that satisfies (\hyperref[D1]{D1}) and (\hyperref[D2]{D2}), and show that this all implies (\hyperref[D3]{D3}).

\begin{theorem}\label{open equivalence}
Let \(Y\subseteq X(G_E)\) be a subset that satisfies the conditions {\rm (\hyperref[D1]{D1})} and {\rm (\hyperref[D2]{D2})}. Then \(Y\) is open in \(X(G_E)\) if and only if \(Y\) satisfies the following conditions.
\begin{enumerate}[\rm(L1)]
    \item \label{L1} Suppose \(u\in \partial E\) is irrational and \(r\notrational{u} \in Y\cap RG_u^u\) for some \(r\in R\). Then there exists a vertex \(v_r\in u^0\) such that \(r\mayberational{z} \in Y\cap RG_z^z\), for all \(z\in \partial E\) with \(v_r\in z^0\).
    \item \label{L2} Suppose \(u=\rho c^\infty \in \partial E\) is a rational path, and \(z\in \partial E\) contains a vertex of c, but is not tail equivalent to \(u\). Then \(r\mayberational{z}\in Y\cap RG_z^z\), for all \(r\in R\) such that \(r\) is a coefficient of a Laurent polynomial of \(Y\cap RG_u^u\).
    \item \label{L3} Suppose \(u\in \partial E\) is a finite path, and \(r(u)\) is an infinite emitter. If \(r\notrational{u}\in Y\cap RG_u^u\) for some \(r\in R\), then there exists a finite set \(F\subseteq s^{-1}(r(u))\) such that if \(z\in \partial E\) is a boundary path that contains a vertex of \(r(s^{-1}(r(u))\setminus F)\), then \(r \mayberational{z} \in Y\).
\end{enumerate}
Moreover, every open subset $Y \subseteq X(G_E)$ that satisfies  {\rm (\hyperref[D1]{D1})} and {\rm (\hyperref[D2]{D2})} also satisfies {\rm (\hyperref[D3]{D3})}.
\end{theorem}

The proof will be given in the Lemmas \ref{open_implies}--\ref{D3 not necessary for LPA}. The first three lemmas show the equivalence between \(Y\) being open and satisfying (\hyperref[L1]{L1})--(\hyperref[L3]{L3}) under the assumption that (\hyperref[D1]{D1}) and (\hyperref[D2]{D2}) hold, and the last lemma shows that this all implies (\hyperref[D3]{D3}). Note that this means that we can ignore condition (\hyperref[D3]{D3}) when applying Theorem \ref{disassembly-theorem} to Leavitt path algebras.

\begin{lemma}\label{open_implies}
Let \(Y\subseteq X(G_E)\) be an open subset satisfying {\rm (\hyperref[D1]{D1})} and {\rm (\hyperref[D2]{D2})}. Then \(Y\) satisfies {\rm (\hyperref[L1]{L1})--(\hyperref[L3]{L3})}.
\end{lemma}
\begin{proof}
(\hyperref[L1]{L1}) Let \(u\in \partial E\) be irrational. Since \(Y\) is open, there exist \(f\in A_R(G_E)\), \(\alpha\in E^*\) and \(F\subseteq_{\text{finite}} s^{-1}(r(\alpha))\) such that \(r\notrational{u} \in s_f(Z(\alpha,F))\subseteq Y\). Clearly, \(\alpha\) must be an initial subpath of \(u\), and \(e\notin F\), where \(e\in E^1\) is the first edge of \(u\) after \(\alpha\). As \(f_u=(r \bm{1}_{\mathcal{Z}(\alpha, \alpha, F)})_u\), Lemma \ref{fu=gu} tells us that there exists an initial subpath \(\beta\) of \(u\), strictly longer than \(\alpha\), such that \( s_{r\bm{1}_{\mathcal{Z}(\beta, \beta)}} (Z(\beta)) = s_f(Z(\beta) )\subseteq Y\). In particular, for any \(z\in Z(\beta)\), we have \(r\mayberational{z} \in Y\). Since any path that contains the vertex \(r(\beta)\) is tail equivalent to some path in \(Z(\beta)\), the fact that the translation maps restrict to isomorphisms implies that \(r(\beta)\) satisfies the conditions needed for~\(v_r\).
    
(\hyperref[L2]{L2}) Now assume \(u=\rho c^\infty\in \partial E\) is rational. Since \(Y\cap RG_u^u\) is an ideal, any coefficient of a Laurent polynomial in \(Y\cap RG_u^u\) can be chosen to be the coefficient of \(1_{(u,0,u)}\) for some Laurent polynomial in $Y \cap RG_u^u$. As such, it is enough to prove that \(r_0\mayberational{z}\in Y\) for any \(\sum_{n\in \mathbb{Z}} r_n1_{(u,n|c|,u)}\in Y\), so take such a Laurent polynomial in \(Y\cap RG_u^u\). Since \(Y\cap RG_u^u\) and \(Y\cap RG_{u'}^{u'}\) are isomorphic if \(u\) and \(u'\) are tail equivalent, we may choose \(u\) to be equal to \(c^\infty\). Again, since \(Y\) is open, there exist \(f\in A_R(G_E)\), \(\alpha\in E^*\) and \(F\subseteq_{\text{finite}} s^{-1}(r(\alpha))\) such that \(\sum_{n\in \mathbb{Z}}r_n 1_{(u,n|c|,u)}\in s_f(Z(\alpha,F))\subseteq Y\). And again, we know that \(\alpha\) must be an initial subpath of \(u\), and that \(e\notin F\), with \(e\) being the first edge of \(u\) after \(\alpha\). Since \(Z(\beta)\subseteq Z(\alpha,F)\), if \(\alpha\) is an initial subpath of \(\beta\) and the first edge of \(\beta\) after \(\alpha\) is not in \(F\), we can choose \(\alpha = c^m\) and \(F=\varnothing\), for some \(m\in \mathbb{N}\). By Lemma \ref{fu=gu}, we can also assume \(f=\sum_n r_n \bm{1}_{X_n}\), for some open bisections \(X_n\) with \(X_n\cap G_u^u=\{(u,n|c|,u)\}\). If \(z\in Z(c^m)\) is irrational or finite, we immediately have \(s_f(z)=r_0 \notrational{z}\). On the other hand, suppose \(z\in Z(c^m)\) is rational and that there exists \(n\neq 0\) with \(X_n\cap G_z^z\neq \varnothing\). Take \(\gamma_n\in E^*\) such that \(X_n=\mathcal{Z}(\gamma_n c, \gamma_n)^n\). Then \(z\) would have to be equal to \(\gamma_n c^n c^n\cdots\), which is tail equivalent to \(u\). Consequently, for any \(z\in Z(c^m)\) that is not tail equivalent to \(u\), we have \(r_0 1_{(z,0,z)} = s_f(z)\in s_f(Z(c^m))\subseteq s_f(Z(\alpha, F))\subseteq Y\). But any \(z'\in \partial E\) that contains a vertex of \(c\) is tail equivalent to a path \(z\) in \(Z(c^m)\), and \(z'\) is tail equivalent to \(u\) if and only if \(z\) is. Since the ideals \(Y\cap RG_z^z\trianglelefteq RG_z^z\) and \(Y\cap RG_{z'}^{z'}\trianglelefteq RG_{z'}^{z'}\) are naturally isomorphic, condition (\hyperref[L2]{L2}) holds.
    
(\hyperref[L3]{L3}) Finally, assume \(u\in \partial E\) is finite, and that \(r(u)\) is an infinite emitter. Like before, we can choose \(f\in A_R(G_E)\), \(\alpha\in E^*\) and \(F\subseteq_{\text{finite}} s^{-1}(r(\alpha))\) such that \(r\notrational{u}\in s_f(Z(\alpha,F)) \subseteq Y\), with \(\alpha\) an initial subpath of \(u\) and if \(\alpha\) is strictly shorter than \(u\), the first edge of \(u\) after \(\alpha\) is not in \(F\). If \(\alpha\) is strictly shorter than \(u\), however, we get \(Z(u,\varnothing)\subseteq Z(\alpha,F)\) and \(r\notrational{u}\in s_f(Z(u, \varnothing))\), so we can assume that \(\alpha\) is equal to \(u\). Now, let \(f=\sum_i r_i\bm{1}_{B_i} + \sum_j r_j\bm{1}_{C_j}\), with all \(B_i\) and \(C_j\) of the form \(\mathcal{Z}(\gamma,\delta)\), and \((u,0,u)\in B_i\) and \((u,0,u)\notin C_j\) for all \(i\) and \(j\). In particular, we have \(\sum_i r_i = r\). Now, consider those \(C_{j'}\), for which \(\gamma_{j'}=\delta_{j'}\), of which \(u\) is a (strict) initial subpath. Let \(F'\) be the union of \(F\), with the set of edges consisting of the first edge of \(\gamma_{j'}\) after \(u\), for all these \(j'\). Then we have \(Z(u,F')\subseteq Z(u,F)\). But for any \(z\in Z(u,F')\), we have \(r \mayberational{z} = s_f(z)\in s_f(Z(u,F'))\subseteq s_f(Z(u,F))\subseteq Y\). Since paths that contain a vertex of \(r(s^{-1}(r(u))\setminus F')\) are tail equivalent with a path in \(Z(u,F')\), \(F'\) satisfies the requirements for~(\hyperref[L3]{L3}).
\end{proof}

\begin{lemma}\label{lemma to prove some condition}
Let \(Y\subseteq X(G_E)\) be a subset satisfying {\rm (\hyperref[D1]{D1})}, {\rm(\hyperref[D2]{D2})}, {\rm (\hyperref[L1]{L1})} and {\rm (\hyperref[L2]{L2})}. If \(u=\rho c^\infty \in \partial E\) is a rational path  such that there are multiple closed simple paths based at \(s(c)=r(c)\), then there exists an ideal \(J\trianglelefteq R\) such that \(Y\cap RG_u^u=JG_u^u\).
\end{lemma}
\begin{proof}
Let \(u=\rho c^\infty\in \partial E\) be a rational path, and \(\gamma\) and \(\delta\) two distinct closed simple paths based at \(s(c)=r(c)\). We will show that for any \(\sum_{n\in \mathbb{Z}}r_n 1_{(u,n|c|,u)}\in Y\cap RG_u^u\), \(r_01_{(u,0,u)}\) is an element of \(Y\). Since \(Y\cap RG_u^u\) is an ideal in \(RG_u^u\) by {\rm (\hyperref[D1]{D1})}, the same will hold for any \(r_n\), and it will follow that \(Y\cap RG_u^u = JG_u^u\), where \(J\) is the ideal containing the coefficients of all Laurent polynomials of \(Y\cap RG_u^u\).

So consider \(\sum_{n\in \mathbb{Z}}r_n 1_{(u,n|c|,u)}\in Y\cap RG_u^u\), and let \(x= \rho \gamma \delta \gamma^2\delta \gamma^3\delta \dots \in \partial E\). Clearly \(x\) is irrational, so it cannnot be tail equivalent to \(u\). So by (\hyperref[L2]{L2}), \(r_0 \notrational{x}\in Y\cap RG_x^x\). Let \(u'=\rho\gamma\delta c^\infty\in \partial E\) be a rational path tail equivalent to \(u\). Since all vertices of \(x\) lie on \(u'\), (\hyperref[L1]{L1}) tells us that \(r_01_{(u',0,u')}\in Y\cap RG_{u'}^{u'}\). But since \(u\) and \(u'\) are tail equivalent, {\rm(\hyperref[D2]{D2})} gives a translation isomorphism \(Y\cap RG_{u'}^{u'}\to Y\cap RG_u^u\) which sends \(r_01_{(u',0,u')}\) to \(r_01_{(u,0,u)}\), so \(r_01_{(u,0,u)}\) is also an element of \(Y\).
\end{proof}

\begin{lemma}\label{implies_open} 
If \(Y\subseteq X(G_E)\) satisfies {\rm (\hyperref[D1]{D1})}, {\rm (\hyperref[D2]{D2})}, and {\rm (\hyperref[L1]{L1})--(\hyperref[L3]{L3})}, then \(Y\) is open in \(X(G_E)\).
\end{lemma}
\begin{proof}
We will show that for any \(\alpha\in Y\), there exist \(f\in A_R(G_E)\) and \(U\in \Bisc(G_E^{(0)})\) such that \(\alpha\in s_f(U)\subseteq Y\). Let \(u\in \partial E\) be the unique boundary path such that \(\alpha\in Y\cap G_u^u\). If \(u\) is finite and ends in a sink, we have \(\alpha = r\notrational{u}\) for some \(r\in R\), and \(\{\alpha\}=s_{r\bm{1}_{\mathcal{Z}(u, u)}}(Z(u)) \subseteq Y\).

Now, suppose \(u\) is irrational. Then \(\alpha= r\notrational{u}\) for some \(r\in R\). Let \(\mu\) be the initial subpath of \(u\) ending at the first occurrence of \(v_r\), with \(v_r\) as in (\hyperref[L1]{L1}). Let \(f=r\bm{1}_{\mathcal{Z}(\mu, \mu)}\), so that \(\alpha\in s_f(Z(\mu))\). Now, for any \(z\in Z(\mu)\), we have \(s_f(z)=r\mayberational{z}\). But by (\hyperref[L1]{L1}), this is an element of \(Y\).

For rational \(u=\rho c^{\infty}\), there exist scalars \(r_n\in R\) such that \(\alpha=\sum_{n\in \mathbb{Z}}r_n 1_{(u,n|c|,u)}\). Let \(f=\sum_{n\in \mathbb{Z}}r_n \bm{1}_{\mathcal{Z}(\rho c,\rho)^n}\), so that clearly \(\alpha\in s_f(Z(\rho))\). We want to show that \(s_f(Z(\rho))\subseteq Y\), so let \(z\in Z(\rho)\). If \(z\) is irrational or finite, we have \(s_f(z)=r_0 \notrational{z}\), which is an element of \(Y\) by (\hyperref[L2]{L2}). If \(z\) is tail equivalent to \(u\), but not equal to \(u\), this means that there are multiple closed simple paths based at \(s(c)=r(c)\). In this case, \(s_f(z)\) is some Laurent polynomial with coefficients in \(\{r_n\mid n\in \mathbb{Z}\}\), which is an element of \(Y\) by Lemma \ref{lemma to prove some condition}, since we know that \(Y\cap RG_u^u = Y\cap RG_z^z\). Lastly, if \(z\) is rational but not tail equivalent to \(u\), \(s_f(z)\) will again be a Laurent polynomial with coefficients in \(\{r_n\mid n\in \mathbb{Z}\}\), which is an element of \(Y\) by (\hyperref[L2]{L2}).

Finally, suppose \(u\) is finite with \(r(u)\) is an infinite emitter, and \(\alpha = r\notrational{u}\), for some \(r\in R\). Let \(f=r\bm{1}_{\mathcal{Z}(u, u,F)}\), with \(F\) found by (\hyperref[L3]{L3}). Clearly, we have \(u\in Z(u,F)\) and \(s_f(u)=\alpha\). If we take any other \(z\in Z(u,F)\), this path contains a vertex of \(r(s^{-1}(r(u))\setminus F)\), so \(s_f(z)\in Y\) by (\hyperref[L3]{L3}), and we find \(s_f(Z(u,F))\subseteq Y\).
\end{proof}

\begin{lemma}\label{D3 not necessary for LPA}
Let \(Y\subseteq X(G_E)\) be an open subset that satisfies {\rm (\hyperref[D1]{D1})} and {\rm (\hyperref[D2]{D2})}. Then \(Y\) also satisfies {\rm (\hyperref[D3]{D3})}.
\end{lemma}
\begin{proof}
For (\hyperref[D3]{D3}), we must check that for each \(u\in \partial E\) and \(\alpha\in Y\cap RG_u^u\), there exists \(f\in A_R(G_E)\) such that \(f_u=\alpha\) and  \((f * \bm{1}_C)_v \in Y\) for all \(v \in \partial E\) and \(C \in \Bisc(G_E)\). As the compact open bisections of the form \(\mathcal{Z}(\gamma, \delta, F)\) form a basis for the topology on \(G_E\), closed under finite intersections, it is enough to check this only for these compact open bisections. As we also have \(\mathcal{Z}(\gamma, \delta, F) = \mathcal{Z}(\gamma, \delta)\setminus \bigcup_{e\in F} \mathcal{Z}(\gamma e, \delta e)\) (where the union is a disjoint union), it is even enough to only check for compact open bisections of the form \(\mathcal{Z}(\gamma, \delta)\), with \(\gamma\) and \(\delta\) paths in \(E\). Note that by Lemma \ref{open_implies}, we know that \(Y\) satisfies (\hyperref[L1]{L1})--(\hyperref[L3]{L3}). Now, let \(u\in \partial E\), and let \(\alpha\in Y\cap RG_u^u\). 

If \(u\) is finite or irrational, \(\alpha\) must be of the form \(r \notrational{u}\) for some \(r\in R\). If \(u\) is finite and ends in a sink, choose \(f=r\bm{1}_{\mathcal{Z}(u,u)}\). If \(u\) is irrational, let \(u'\) be the initial subpath of \(u\) until the first occurrence of \(v_r\) as in (\hyperref[L1]{L1}), and choose \(f=r\bm{1}_{\mathcal{Z}(u',u')}\). If \(u\) ends in an infinite emitter, let \(f=r\bm{1}_{\mathcal{Z}(u,u,F)}\), with \(F\) as in (\hyperref[L3]{L3}). Finally, if \(u=\rho c^\infty\) is rational, \(\alpha\) will be of the form \(\sum_n r_n1_{(u,n|c|, u)}\) for some \(r_n\in R\), so let \(f=\sum_n r_n\bm{1}_{\mathcal{Z}(\rho c^n, \rho)}\). Then one can check by applying the conclusions of (\hyperref[L1]{L1})--(\hyperref[L3]{L3}) that \(f\) satisfies the condition needed for (\hyperref[D3]{D3}).
\end{proof}

\section{Ideals of Leavitt path algebras} \label{sec:ideals of LPAs}

In \cite[Theorem 2.8.10]{LPAbook}, it is shown that each ideal \(I\) of \(L_K(E)\), with \(K\) a field, is determined by the set of vertices \(H\) that it contains, the breaking vertices \(w\) of \(H\) for which \(w^H\in I\), and the Laurent polynomials \(p\in K[x, x^{-1}]\) with \(p(c)\in I\) for certain cycles \(c\) in \(E\). In the case of Leavitt path algebras over an arbitrary commutative ring \(R\) with unit, we have the additional difficulty that the ideal structure of \(R\) comes into play as well. It is clear that for any ideal \(I\trianglelefteq L_R(E)\) and for any vertex \(v\in E^0\), the set \(\{r\in R\mid rv\in I\}\) is an ideal of \(R\). So one way to resolve the problem about the ideal structure of \(R\) would be to assign to each ideal \(J\trianglelefteq R\) the pair \((H,S)\), with \(H=\{v\in E^0\mid rv\in I \text{ for all } r\in J\}\) and \(S\) the set of breaking vertices \(w\) of \(H\) for which \(rw^H\in I\) for each \(r\in J\). Such a pair would then automatically be an \emph{admissible pair}, a concept we will define shortly. If we similarly assign to each ideal of \(R[x,x^{-1}]\) a certain set of cycles as well, this would determine the ideal \(I\) completely.

The approach outlined in the previous paragraph would succeed in cataloguing the ideals of $L_R(E)$ using graphical data and ideals of $R$. But this is not exactly the approach we followed. It turns out to be better to describe an ideal \(I\trianglelefteq L_R(E)\) by instead assigning to each admissible pair an ideal of~\(R\), and to certain cycles an ideal of \(R[x,x^{-1}]\). In particular, this will make it easier to describe  the sum, intersection, and product of ideals. 

{Using this approach, each ideal of \(L_R(E)\) will correspond to a function from the set of admissible pairs of \(E\) to the lattice of ideals of \(R\), and a function from a certain set of cycles of~\(E\) to the lattice of ideals of \(R[x,x^{-1}]\). We will start this section by describing the functions that we will need, and the structure of the lattice that they form. After that, we will use Theorem~\ref{disassembly-theorem} to show that these functions indeed classify the ideals of \(L_R(E)\). Finally, to end this paper, we will describe what the product of ideals looks like, and which pairs of functions correspond to graded ideals.}

\subsection{Some auxiliary lattices} \label{section: lattices}
Before we can describe a lattice to which the lattice of ideals of \(L_R(E)\) is isomorphic, we will need some other lattices. The first of these is the lattice \(\mathscr{T}_E\) of \emph{admissible pairs}, as described in \cite[Definition 2.5.3]{LPAbook}. An admissible pair is a pair \((H,S)\) where \(H\subseteq E^0\) is hereditary saturated and \(S\) is a set of breaking vertices for \(H\). A partial order on $\mathscr{T}_E$ is given by \[(H_1,S_1)\leq (H_2,S_2) \Longleftrightarrow H_1\subseteq H_2 \text{ and } S_1\subseteq H_2\cup S_2.\]

The partially ordered set \(\mathscr{T}_E\) is a lattice, with the following supremum and infimum; see \cite[Proposition 2.5.6]{LPAbook}:
\begin{align*}(H_1, S_1)\vee (H_2, S_2) &= \left(\overline{H_1 \cup H_2}^{S_1\cup S_2},\ (S_1\cup S_2)\setminus \overline{H_1\cup H_2}^{S_1\cup S_2}\right),\\
(H_1, S_1) \wedge (H_2,S_2) &= \Big(H_1\cap H_2, \ (S_1\cap S_2) \cup \big((S_1\cup S_2)\cap (H_1\cup H_2)\big)\Big),\end{align*}
where, for a hereditary set of vertices \(H\) and a set of vertices \(S \subseteq H\cup B_H\), the set \(\overline{H}^S\) is the \emph{\(S\)-saturation} of \(H\), defined as the smallest subset \(H'\) of vertices that satisfies the following:
\begin{itemize}
    \item \(H\subseteq H'\),
    \item \(H'\) is hereditary saturated,
    \item If \(v\in S\) and \(r(s^{-1}(v))\subseteq H'\), then \(v\in H'\).
\end{itemize}

This \(S\)-saturation can be constructed as follows: let \(\Lambda_0^S(H) = H\) and  
\begin{equation} \label{S-saturation}\Lambda_n^S(H) = \Lambda_{n-1}^S(H) \cup \Big\{v\in E^0\setminus \Lambda_{n-1}^S(H)\mid v\in \operatorname{Reg}(E)\cup S\text{ and } r(s^{-1}(v)) \subseteq \Lambda_{n-1}^S(H)\Big\}.	
\end{equation}
Then \(\overline{H}^S = \bigcup_{n\in \mathbb{N}} \Lambda_n^S(H)\). See also \cite[Definition 2.5.5]{LPAbook}.

We can make the following observations about this construction, which will be useful later on:

\begin{observation}\label{observation: saturation} 
 \begin{enumerate}
     \item \label{observation: infinite emitter} If \(v\in \overline{H}^S\) is an infinite emitter, then we must already have \(v\in H\cup S\).
     \item \label{observation: vertex on cycle} If the vertex \(v\in \overline{H}^S\) lies on a cycle, then \(v\) must already be in \(H\).
     \item \label{observation: vertices receiving paths} We have $v \in \Lambda_n^S(H)$ if and only if every regular vertex receiving a path from $v$ of length $m$ is in $H$ and every singular  vertex receiving a path from $v$ of length $\le m$ is in $H \cup S$. This can be proved by induction on $n$.
 \end{enumerate}
\end{observation}

\begin{example} \label{graph example}
As an example, let us again consider the Toeplitz graph
\[T=
 	\xymatrix{\bullet_u \ar@(ul,dl)_e \ar[r]^f & \bullet_v}
\]
from Example \ref{toeplitz graph groupoid}. The corresponding Leavitt path algebra is given by the free associative \(R\)-algebra \(R\langle X,Y\rangle\) modulo the single relation \(XY=1\), as was shown in \cite[Proposition 1.3.7]{LPAbook} for the case where \(R\) is a field. Let us determine the set \(\mathscr{T}_{T}\). Clearly, \(T\) has no infinite emitters, and any hereditary set of vertices containing \(u\) must also contain \(v\). It is then easy to verify that \(\mathscr{T}_{T}=\big\{(\varnothing, \varnothing), (\{v\}, \varnothing), (T^0, \varnothing)\big\}\).
\end{example}

In \cite[Theorem 2.5.8]{LPAbook}, it is proved that \(\mathscr{T}_E\) is isomorphic to the lattice of graded ideals of \(L_K(E)\), for any field \(K\). Since these graded ideals form a complete lattice, \(\mathscr{T}_E\) is also complete. On the groupoid side, $\mathscr{T}_E$ is isomorphic to the lattice of open invariant subsets of $G_E^{(0)}$ \cite[Theorem 3.3]{clark2016decomposability}. The supremum of an arbitrary number of admissible pairs is given in the following proposition.

\begin{proposition} \label{prop:supremum in TE}
For any collection of admissible pairs $\{(H_i,S_i) \}_i$, we have \[\displaystyle \bigvee_i (H_i,S_i) = \left(\overline{\bigcup_i H_i}^{\bigcup_i S_i},\ \Big(\bigcup_i S_i\Big)\setminus \overline{\bigcup_i H_i}^{\bigcup_i S_i}\right).\]
\end{proposition}
\begin{proof} 
We essentially use the proof of \cite[Proposition 2.5.6]{LPAbook}, but generalise it to infinite index sets. We immediately find that \(\overline{\bigcup_i H_i}^{\bigcup_i S_i}\) is hereditary saturated, by definition of the saturation. {Any vertex in  \((\bigcup_i S_i)\setminus\overline{\bigcup_i H_i}^{\bigcup_i S_i}\) must also be a breaking vertex for \(\overline{\bigcup_i H_i}^{\bigcup_i S_i}\)}.  So \((\overline{\bigcup_i H_i}^{\bigcup_i S_i}, (\bigcup_i S_i)\setminus\overline{\bigcup_i H_i}^{\bigcup_i S_i})\) is an admissible pair, and it is clearly greater than   or equal to any \((H_i,S_i)\), so it must be greater than or equal to \(\bigvee_i (H_i,S_i)\).

Now suppose we have some \((H,S)\in \mathscr{T}_E\) such that \((H,S)\geq (H_i,S_i)\), or in other words \(H_i\subseteq H\) and \(S_i\subseteq H\cup S\), for all \(i\). Then we also have \(\bigcup S_i\subseteq H\cup S\), so \((\bigcup_i S_i)\setminus\overline{\bigcup_i H_i}^{\bigcup_i S_i} \subseteq H\cup S\). Using the construction of the saturation given above, we will show that \(\Lambda_n^{\bigcup_i S_i} (\bigcup_i H_i) \subseteq H\) for each \(n\in \mathbb{N}\), from which it will follow that \(\overline{\bigcup_i H_i}^{\bigcup_i S_i} \subseteq H\). The case \(n=0\) is just \(\bigcup H_i \subseteq H\), which follows immediately from the fact that \(H_i \subseteq H\), for all \(i\). So suppose \(\Lambda_{n-1}^{\bigcup_i S_i} (\bigcup_i H_i) \subseteq H\), for some \(n\geq 1\). Suppose we have some vertex \(v\in E^0\setminus \Lambda_{n-1}^{\bigcup_i S_i} (\bigcup_i H_i)\). If \(v\) is regular and \(r(s^{-1}(v)) \subseteq \Lambda_{n-1}^{\bigcup_i S_i} (\bigcup_i H_i)\), then \(v\) must be in \(H\), since \(H\) is saturated. On the other hand, if \(v\in \bigcup_i S_i\) and \(r(s^{-1}(v)) \subseteq \Lambda_{n-1}^{\bigcup_i S_i} (\bigcup_i H_i)\), \(v\) must also be an element of \(H\): we know that it is an element of \(H\cup S\), but it cannot be in \(S\), as $r(s^{-1}(v)\subseteq H$ implies $v$ is not a breaking vertex for \(H\). It follows that \(\big(\overline{\bigcup_i H_i}^{\bigcup_i S_i},\ (\bigcup_i S_i)\setminus\overline{\bigcup_i H_i}^{\bigcup_i S_i}\big) \leq (H,S)\), so \(\bigvee_i (H_i,S_i) = \big(\overline{\bigcup_i H_i}^{\bigcup_i S_i},\ (\bigcup_i S_i)\setminus\overline{\bigcup_i H_i}^{\bigcup_i S_i}\big) \).
\end{proof}

Another property of \(\mathscr{T}_E\) is the following:
\begin{lemma}\label{admissible pairs distributive}
Let \(\{(H_\alpha, S_\alpha)\}_\alpha\) and \(\{(H_\beta, S_\beta)\}_\beta\) be sets of admissible pairs such that \(\bigvee_\alpha (H_\alpha, S_\alpha) = (H_1,S_1)\) and \(\bigvee_\beta (H_\beta, S_\beta) = (H_2,S_2)\).  Then we have \(\bigvee_{\alpha, \beta} ((H_\alpha, S_\alpha) \wedge (H_\beta, S_\beta)) = (H_1, S_1)\wedge (H_2, S_2)\). In particular, \(\mathscr{T}_E\) is a distributive lattice.
\end{lemma}
\begin{proof}
We will write the supremum \(\bigvee_{\alpha, \beta} \big((H_\alpha, S_\alpha) \wedge (H_\beta, S_\beta)\big)\) as \((H,S)\). Note that this means $H = \overline{H'}^{S'}$ and $S = S'\setminus H$, where $H' = \bigcup_{\alpha, \beta} H_\alpha \cap H_\beta$ and $S' = \bigcup_{\alpha, \beta}(S_\alpha \cap S_\beta)\cup ((S_\alpha \cup S_\beta)\cap (H_\alpha \cup H_\beta))$.
Since \((H_\alpha, S_\alpha)\wedge (H_\beta, S_\beta) \leq (H_i,S_i)\) for each \(\alpha\) and \(\beta\) and for \(i\in \{1,2\}\), it follows that \((H,S) \leq (H_1, S_1)\wedge (H_2, S_2)\). For the reverse inequality, recall that \[(H_1, S_1)\wedge (H_2, S_2)=\big(H_1\cap H_2, (S_1\cap S_2)\cup ((S_1\cup S_2)\cap (H_1\cup H_2))\big).\] Suppose we have some vertex \(v\in H_1\cap H_2\). Using the construction given above, we can find \(m,n\in \mathbb{N}\) such that \(v\in \Lambda_m^{\bigcup_{\alpha} S_\alpha} (\bigcup_\alpha H_\alpha)\) and \(v\in \Lambda_n^{\bigcup_{\beta} S_\beta} (\bigcup_\beta H_\beta)\).  This means (see Observation \ref{observation: saturation} (\ref{observation: vertices receiving paths})) that all regular vertices that receive a path from $v$ of length \(m\) are contained in \(\bigcup_\alpha H_\alpha\) and all singular vertices that receive a path from \(v\) of length \(\le m\) are contained in \(\bigcup_\alpha H_\alpha\cup S_\alpha\). In particular, all such vertices lie in \(\bigcup_\alpha H_\alpha \cup S_\alpha\). The same can be said for \(n\) and \(\bigcup_\beta H_\beta \cup S_\beta\). Suppose that \(n\geq m\), the other case being analogous. Since each \(H_\alpha\) is hereditary, we also find that all regular vertices that receive a path of length \(n\) from \(v\) and all singular vertices that receive a path from \(v\) of length  \(\le n\) are elements of \(\bigcup_\alpha H_\alpha \cup S_\alpha\). So each such vertex is contained in some \((H_\alpha \cup S_\alpha)\cap (H_\beta \cup S_\beta)\). It follows from Observation \ref{observation: saturation} (\ref{observation: vertices receiving paths}) that \(v\in \Lambda_n ^{S'}(H') \subseteq H\), so \(H_1\cap H_2\subseteq H\).

Now suppose \(v\) is a vertex in \((S_1\cap S_2) \cup ((S_1\cup S_2) \cap (H_1\cup H_2)) = (S_1\cap S_2)\cup (S_1\cap H_2) \cup (H_1\cap S_2)\). If \(v\) is an element of \(S_1\), there must exist \(\alpha\) such that \(v\in S_\alpha\). On the other hand, if \(v\in H_1\), \(v\) must either be in some \(H_\alpha\) or in some \(S_\alpha\), as \(v\) is not a regular vertex. So if \(v\in S_1\cap S_2\), we can find \(\alpha\) and \(\beta\) such that \(v\in S_\alpha\cap S_\beta\), so \(v\) is in the set of breaking vertices of the admissible pair \((H_\alpha, S_\alpha) \wedge (H_\beta, S_\beta)\), so \(v\in H\cup S\). If \(v\in S_1\cap H_2\) we can find \(\alpha\) and \(\beta\) such that \(v\in S_\alpha\) and \(v\in H_\beta \cup S_\beta\), so again, so \(v\) is in the set of breaking vertices of \((H_\alpha, S_\alpha) \wedge (H_\beta, S_\beta)\), so \(v\in H\cup S\). The case \(v\in H_1\cap S_2\) is similar, and we conclude that \(v\) must be an element of \(H\cup S\). It follows that \((H_1, S_1)\wedge (H_2, S_2) \leq (H,S)\), which proves the lemma.
\end{proof}

If we define the set \(\mathscr{T}_E^* = \mathscr{T}_E\setminus \{(\varnothing,\varnothing)\}\), this forms a join-semilattice under the same partial order and with the same supremum, since \((\varnothing, \varnothing)\) is the least element of \(\mathscr{T}_E\). Remember also that for a ring \(A\), the lattice of ideals with the usual partial order is denoted by \(\mathcal{L}(A)\). The following definition is a generalisation of an idea in \cite[Theorem 6.1 (b)]{clark2019ideals}.

\begin{definition}\label{def:saturated function}
The set \(\mathscr{F}_{E,R}\) of \emph{saturated functions} is the set of functions \(f: \mathscr{T}_E^* \to \mathcal{L}(R)\) that satisfy \[f\Big(\bigvee_{i\in I} (H_i,S_i)\Big)= \bigcap_{i\in I} f\big((H_i,S_i)\big)\] for all families \(\{(H_i,S_i)\}_{i \in I}\) of admissible pairs in \(\mathscr{T}_E^*\). The set \(\mathscr{F}_{E,R}\) comes with a natural partial order, defined by: \(f_1\leq f_2\) if and only if \(f_1(H,S) \subseteq f_2(H,S)\) for all \((H,S)\in \mathscr{T}_E^*\).
\end{definition}

Notice that saturated functions are order-reversing, i.e.\! if \((H_1,S_1)\leq (H_2,S_2)\) then \(f(H_2,S_2)\subseteq f(H_1,S_1)\). If we reverse the partial order of \(\mathscr{T}_E^*\) and denote the new partially ordered set by \((\mathscr{T}_E^*)^{\rm op}\), then \(\mathscr{F}_{E,R}\) is the set of morphisms  \((\mathscr{T}_E^*)^{\rm op}\to \mathcal{L}(R)\) of complete meet-semilattices. 

\begin{notation}
Given two $\mathcal{L}(R)$-valued functions \(f_1\) and \(f_2\) with the same domain, we define the functions \(f_1+ f_2\), \(f_1\cap f_2\) and \(f_1f_2\), given by the pointwise sum, intersection, or multiplication of ideals respectively. We can similarly define sums and intersections of arbitrarily many $\mathcal{L}(R)$-valued functions.
\end{notation}

To show that \(\mathscr{F}_{E,R}\) is a lattice, we will be interested in \(f_1\cap f_2\) and \(f_1+f_2\), for  functions \(f_1,f_2\in \mathscr{F}_{E,R}\). It is clear that the intersection of two saturated functions is also saturated. Unfortunately, while the sum \(f_1+f_2\) will be order-reversing, it need not be saturated, as the following example shows:

\begin{example}\label{saturation example}
Consider the graph
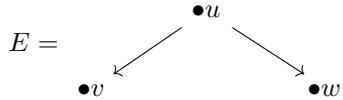
\[E=\begin{tikzcd}
&\bullet u \arrow[dl] \arrow[dr] & \\
\bullet v & & \bullet w.
\end{tikzcd}\]
Then the functions
\[f_1: \begin{cases}
\{v\} &\mapsto R\\
\{w\} &\mapsto 0\\
E^0 &\mapsto 0
\end{cases} \qquad \text{ and } \qquad 
f_2: \begin{cases}
\{v\} &\mapsto 0\\
\{w\} &\mapsto R\\
E^0 &\mapsto 0
\end{cases}
\]
are saturated. Their sum, however, is not saturated, as 
\[(f_1 + f_2)(E^0) = 0 \neq R= (f_1 + f_2)(\{v\}) \cap (f_1+f_2)(\{w\}).\]
\end{example}

To address this problem, we will introduce the following definition.

\begin{definition} \label{def:saturation}
Given an order-reversing function \(f:\mathscr{T}_E^*\to \mathcal{L}(R)\), we will define the \emph{saturation} of \(f\) as the smallest function \(\overline{f}\in \mathscr{F}_{E,R}\) such that \(f(H,S)\subseteq \overline{f}(H,S)\) for all \((H,S)\in \mathscr{T}_E^*\). A standard argument proves that \(\overline f\) exists, and we give a construction of it below.
\end{definition}

\begin{remark}
Note that in Example \ref{saturation example}, the saturation of $f_1 + f_2$ will be the function sending every hereditary saturated subset to \(R\). This is because the saturation of \(\{v,w\}\) is \(E^0\). This can be seen to hold in more generality; namely,  the saturation of functions in \(\mathscr{F}_{E,R}\) taking values in \(\{0,R\}\) is a generalisation of the saturation of hereditary sets of vertices of \(E\).
\end{remark}

It is clear that \(\mathscr{F}_{E,R}\) is a complete lattice, with \(\bigwedge f_i = \bigcap_i f_i\) and \(\bigvee_i f_i = \overline{\sum_i f_i}\). We have the following explicit construction for the saturation of an order-reversing function:
\begin{proposition}\label{construction saturation}
If \(f:\mathscr{T}_E^*\to \mathcal{L}(R)\) is order-reversing, the saturation \(\overline{f}\) of \(f\) is given by \[\overline{f}(H,S) = \bigcup_{\substack{A\subseteq \mathscr{T}_E^* \\ \scalebox{0.70}{$\displaystyle \bigvee_{\alpha\in A}(H_\alpha, S_\alpha) = (H,S)$}}}\left( \bigcap_{\alpha\in A} f(H_\alpha, S_\alpha)\right).\]
\end{proposition}
\begin{proof}
Denote the expression on the right-hand side as $f^*(H,S)$. We will show that \(f^*(H,S) \in \mathcal{L}(R)\) and that the function \(f^*\) is saturated, i.e.\! \(f^* \in \mathscr{F}_{E,R}\). Once we have this, it is immediate that it must be the least element of \(\mathscr{F}_{E,R}\) such that \(f(H,S) \subseteq f^*(H,S)\) for all \((H,S) \in \mathscr{T}_E^*\).
 
We must first show that \(f^*(H,S)\) is an ideal of \(R\). Since each \(f(H_\alpha,S_\alpha)\) is an ideal of \(R\), we only have to show that, if \(x,y\in f^*(H,S)\), we also have \(x+y\in f^*(H,S)\). So consider a subset \(A \subseteq \mathscr{T}_E^*\) such that \(\bigvee_{\alpha\in A}(H_\alpha, S_\alpha) = (H,S)\) and \(x\in \bigcap_{\alpha\in A}f(H_\alpha, S_\alpha)\), and a similar subset \(B\) for \(y\). Consider the subset \(C\subseteq \mathscr{T}_E^*\) given by \(C=\{(H_\alpha, S_\alpha) \wedge (H_\beta, S_\beta)\mid \alpha\in A, \beta\in B\}\). By Lemma~\ref{admissible pairs distributive}, we see that \(\bigvee_{\gamma\in C}(H_\gamma, S_\gamma) = (H,S)\). Since each \((H_\gamma,S_\gamma)\) is less than some \((H_\alpha,S_\alpha)\) and less than some \((H_\beta,S_\beta)\), each \(f(H_\gamma,S_\gamma)\) contains both \(x\) and \(y\), and consequently also \(x+y\). It follows that \(x+y\in f^*(H,S)\).

To show that \(f^*\) is order-reversing, consider some \((H_1,S_1)\leq (H_2,S_2)\) in \(\mathscr{T}_E^*\), and let \(r\in \bigcap_{\alpha \in A} f(H_\alpha, S_\alpha) \subseteq f^*(H_2,S_2)\) for a certain \(A\subseteq \mathscr{T}_E^*\) with \(\bigvee_{\alpha \in A} (H_\alpha, S_\alpha) = (H_2, S_2)\). Now consider the subset \(B\subseteq \mathscr{T}_E^*\) given by 
\[
B=\{(H_\alpha, S_\alpha)\wedge (H_1,S_1) \mid \alpha\in A \text{ such that } (H_\alpha, S_\alpha)\wedge (H_1,S_1) \neq (\varnothing, \varnothing)\}.
\]
By Lemma \ref{admissible pairs distributive} we have \(\bigvee_{\beta\in B} (H_\beta, S_\beta) = (H_1,S_1)\). Since \(f\) is order-reversing, it follows that \(r\in \bigcap_{\beta\in B} f(H_\beta, S_\beta)\), so \(r\in f^*(H_1,S_1)\).

Finally, we must show that \(f^*(\bigvee_{i\in I}(H_i,S_i)) = \bigcap_{i\in I} f^*(H_i,S_i)\), for an arbitrary number of admissible pairs \(\{(H_i, S_i)\}_{i\in I}\). The inclusion \(\subseteq\) follows from the order-reversing property of \(f^*\). On the other hand, if we have \(r\in \bigcap_{i\in I} f^*(H_i,S_i)\) for some \(r\in R\), there exist subsets \(A_i\subseteq \mathscr{T}_E^*\) such that \(\bigvee_{\alpha_i\in A_i}(H_{\alpha_i}, S_{\alpha_i}) = (H_i,S_i)\) and \(r\in \bigcap_{\alpha_i\in A_i} f(H_{\alpha_i}, S_{\alpha_i})\) for all \(i\). Since \(\bigvee_{i\in I, \alpha_i\in A_i} (H_{\alpha_i}, S_{\alpha_i}) = \bigvee_{i\in I} (H_i,S_i)\) and \(r\in \bigcap_{i\in I, \alpha_i\in A_i} f(H_{\alpha_i}, S_{\alpha_i})\), it follows from the definition of \(f^*\) that \(r\in f^*(\bigvee_{i\in I}(H_i,S_i))\).
\end{proof}

The following corollary will also be useful later on. Recall that \(\overline{c^0}\) is the hereditary saturated closure of the set of vertices of \(c\).
\begin{corollary}\label{saturation function for cycles}
If \(c\) is a cycle in \(E\) and \(f:\mathscr{T}_E^*\to \mathcal{L}(R)\) is an order reversing function, we have \(\overline{f}(\overline{c^0},\varnothing) = f(\overline{c^0}, \varnothing)\).
\end{corollary}
\begin{proof}
The inclusion \(f(\overline{c^0}, \varnothing) \subseteq \overline{f}(\overline{c^0},\varnothing)\) is obvious. For the reverse inclusion, suppose we have a subset \(A\subseteq \mathscr{T}_E^*\) such that \(\bigvee_{\alpha\in A} (H_\alpha, S_\alpha) = (\overline{c^0}, \varnothing)\) and \(r\in \bigcap_{\alpha\in A}f(H_\alpha, S_\alpha)\). Since \(c\) is a cycle, there must be some pair \((H_\alpha, S_\alpha)\) such that \(H_\alpha\) contains a vertex of \(c\), by Observation \ref{observation: saturation} \eqref{observation: vertex on cycle}. But then we have \((\overline{c^0}, \varnothing) \leq (H_\alpha, S_\alpha)\), so \(r\in f(H_\alpha, S_\alpha) \subseteq f(\overline{c^0}, \varnothing)\).
\end{proof}

\subsection{The ideal correspondence for Leavitt path algebras}

In this section we introduce the lattice \(\mathscr{D}_{E,R}\). This lattice will be the most important one, as it will be isomorphic to the lattice of ideals of \(L_R(E)\). It may interest the reader to skip ahead to Theorem \ref{LPA ideal correspondence} and see the isomorphism before reading through the details of how we derived it.

\begin{notation} 
\begin{enumerate}[(i)]
\item We say two cycles \(c=e_1 \cdots e_n\) and \(c'=e'_1 \cdots e'_m\) are \emph{equivalent} if \(n=m\) and there exists \(i\in \{0, \ldots, n-1\}\) such that \(e_1 \cdots e_n = e'_{i+1} \cdots e'_{i+n}\), where the indices are taken modulo \(n\). So two cycles are equivalent if they contain the same edges, in the same order, but possibly have different basepoints.
\item Let \(C_u(E)\) be the set of equivalence classes of cycles \(c\) in \(E\) such that there is exactly one closed simple path based at \(s(c) = r(c)\), as in \cite[2.8.1]{LPAbook}. In general we will ignore the choice of basepoint, and write \(c\in C_u(E)\) for some representative cycle \(c\). (Cycles in \(C_u(E)\) were originally called \emph{cycles without K} in \cite{rangaswamyprime}, and have also been called \emph{exclusive cycles} in \cite{ararangaswamy}.)
\item For any cycle \(c\) in \(E\), let \(c_\downarrow\) be the hereditary saturated closure of the ranges of the exits of \(c\).
\end{enumerate}
\end{notation}

In particular, we have \(c_\downarrow = \overline{c^0}\) if and only if \(c\notin C_u(E)\) and for \(c\in C_u(E)\), \((c_\downarrow, \varnothing)\) is the largest admissible pair strictly smaller than \((\overline{c^0}, \varnothing)\).

\begin{definition} \label{D_E,R}
The set \(\mathscr{D}_{E,R}\) is the set of pairs \((f,g)\) where \(f\in \mathscr{F}_{E,R}\) and \(g:C_u(E)\to \mathcal{L}(R[x,x^{-1}])\) is a map such that \(g(c)\cap R = f(\overline{c^0},\varnothing)\) and \(g(c) \subseteq f(c_{\downarrow},\varnothing)[x,x^{-1}]\) if \(c_\downarrow\neq \varnothing\).
\end{definition}

The reason we only require $g$ to be defined on \(C_u(E)\), and not all the cycles of \(E\), is that for a cycle \(c\notin C_u(E)\), we have \(c_\downarrow = \overline{c^0}\). But then \(g(c)\) would have to be \(f(\overline{c^0}, \varnothing)[x,x^{-1}]\), so it is redundant to define \(g\) for cycles not in \(C_u(E)\). A consequence of Definition {\ref{D_E,R}} is that for each \((f,g)\in \mathscr{D}_{E,R}\) and \(c\in C_u(E)\), we have \(f(\overline{c^0}, \varnothing)[x,x^{-1}] \subseteq g(c) \subseteq f(c_\downarrow, \varnothing)[x,x^{-1}]\), but note that this alone does not imply \(g(c)\cap R = f(\overline{c^0}, \varnothing)\).

\begin{notation}\label{drawing functions}
To make it easier to describe elements of \(\mathscr{D}_{E,R}\), we introduce the following notation. Suppose \(E\) has a finite number of vertices and edges. Then we can visualise a pair \((f,g)\in \mathscr{D}_{E,R}\) by drawing the graph \(E\), and attaching to each vertex \(v\in E^0\) the ideal \(f(\overline{\{v\}}) \trianglelefteq R\) and to each cycle \(c\in C_u(E)\) the ideal \(g(c) \trianglelefteq R[x,x^{-1}]\). As \(E\) is finite there are no breaking vertices, and every hereditary saturated subset \(H\) is  the supremum (in the lattice \(\mathscr{T}_E\)) of the sets of the form \(\overline{\{v\}}\), where \(v\) runs over the vertices in \(H\). So as \(f\) is saturated, it is uniquely determined by its values at sets of the form \(\overline{\{v\}}\).

Note however that if we were to draw \(E\) and attach arbitrary ideals of \(R\) to the vertices and ideals of \(R[x,x^{-1}]\) to the cycles in \(C_u(E)\), this would not necessarily correspond to an element of \(\mathscr{D}_{E,R}\). We will give an example of this notation in the example below.
\end{notation}

\begin{example}
Using the same graph \(T\) from Example \ref{graph example}, let us work out the set \(\mathscr{D}_{T,\ZZ}\). To determine \(\mathscr{F}_{T,\ZZ}\), recall that \(\mathscr{T}_{T}^*\) contains exactly two admissible pairs. The condition \(f(\bigvee_{i\in I} (H_i,S_i))= \bigcap_{i\in I} f((H_i,S_i))\) is then equivalent to \(f(T^0, \varnothing)\subseteq f(\{v\}, \varnothing)\). So each \(f\in \mathscr{F}_{T, \ZZ}\) is determined uniquely by a pair of integers \((a,b)\) with \(a\mid b\), by sending \((\{v\}, \varnothing)\) to \(a\ZZ\), and \((T^0, \varnothing)\) to \(b\ZZ\).
		
Next, note that \(e\) is the only cycle in \(T\), and we have \(\overline{e^0}=T^0\) and \(e_\downarrow = \{v\}\). So for a pair \((f,g)\in \mathscr{D}_{T,\ZZ}\), we must have \(f\in \mathscr{F}_{T, \ZZ}\) as above, and \(g(e)=b\ZZ[x,x^{-1}] + aI\), where \(I\trianglelefteq \ZZ[x,x^{-1}]\) is any ideal with \(I\cap \ZZ\subseteq \frac{b}{a}\ZZ\). Each such pair \((f,g)\) is clearly an element of \(\mathscr{D}_{T, \ZZ}\), so we have determined the set \(\mathscr{D}_{T, \ZZ}\) completely.
	
To illustrate the notation introduced above, we can describe such a pair \((f,g)\) as follows:
\[\xymatrix{b\ZZ \ar@(ul,dl)_{g(e)} \ar[r] & a\ZZ}.\]
\end{example}

The following proposition will tell us how we can give the set $\mathscr{D}_{E,R}$ a lattice structure.

\begin{proposition}\label{lattice of pairs of functions}
The relation
\begin{align*}
 (f_1,g_1)\leq (f_2,g_2) \iff \begin{cases} f_1(H,S)\subseteq f_2(H,S) & \text{for all } (H,S)\in \mathscr{T}_E^*,\text{ and} \\
 g_1(c)\subseteq g_2(c) &\text{for all } c\in C_u(E)\end{cases}
\end{align*}
is a partial order on \(\mathscr{D}_{E,R}\). With this partial order, \(\mathscr{D}_{E,R}\) forms a lattice, with the meet and join operations given by:
 \begin{align*}(f_1,g_1)\wedge (f_2,g_2) &= (f_1\cap f_2, g_1\cap g_2), \\ 
	(f_1,g_1)\vee (f_2,g_2) &= (\overline{f_1+ f_2 + G_c^+}, g_1+ g_2),
 \end{align*}
  where \(G_c^+: \mathscr{T}_E^*\to \mathcal{L}(R)\) is the function that maps \((\overline{c^0},\varnothing)\) to \((g_1(c)+g_2(c))\cap R\) for \(c\in C_u(E)\), and the rest to the zero ideal.
\end{proposition}
\begin{proof}
Clearly, \(\leq\) is a partial order. For the meet, we only have to show that \((f_1\cap f_2, g_1\cap g_2)\) is an element of \(\mathscr{D}_{E,R}\): if this is the case, it is easy to see that it is the greatest lower bound of \((f_1,g_1)\) and \((f_2,g_2)\). For \(c\in C_u(E)\) we have
\[
(g_1\cap g_2)(c)\cap R = g_1(c)\cap R\cap g_2(c)\cap R = f_1(\overline{c^0},\varnothing) \cap f_2(\overline{c^0}, \varnothing) = (f_1\cap f_2)(\overline{c^0},\varnothing).
\]
If \(c_\downarrow \neq \varnothing\), then
\[
g_1(c)\cap g_2(c) \subseteq f_1(c_\downarrow,\varnothing) [x,x^{-1}] \cap f_2(c_\downarrow, \varnothing)[x,x^{-1}] = (f_1\cap f_2) (c_\downarrow, \varnothing)[x,x^{-1}]
\]
 since \(g_1(c)\cap g_2(c)\subseteq g_i(c)\subseteq f_i(c_\downarrow, \varnothing) [x,x^{-1}]\) for \(i= 1,2\). It follows that \((f_1\cap f_2, g_1\cap g_2) \in \mathscr{D}_{E,R}\), so \((f_1,g_1)\wedge (f_2,g_2) = (f_1\cap f_2, g_1\cap g_2)\).

For the supremum of \((f_1,g_1)\) and \((f_2,g_2)\) we will start by showing that \(f_1+f_2 + G_c^+\) is order-reversing, so its saturation \(\overline{f_1+f_2 + G_c^+}\) is defined and is an element of \(\mathscr{F}_{E,R}\); see Definition~\ref{def:saturation}. Since \(f_1\) and \(f_2\) are order-reversing, we only need to verify that, if \((H,S)< (\overline{c^0},\varnothing)\) for some \((H,S)\in \mathscr{T}_E^*\) and \(c\in C_u(E)\), then \((g_1(c)+g_2(c))\cap R \subseteq f_1(H,S)+ f_2(H,S)\). If we have such an admissible pair, we know that \((H,S)\leq (c_\downarrow,\varnothing)\), so {this} follows from the fact that \(g_i(c)\subseteq f_i(c_\downarrow, \varnothing)[x,x^{-1}]\), with \(i\in \{1,2\}\), i.e. all the coefficients of the Laurent polynomials in \(g_i(c)\) are elements of \(f_i(c_\downarrow, \varnothing)\).

Next, we will show that \((\overline{f_1+ f_2 + G_c^+}, g_1+ g_2)\in \mathscr{D}_{E,R}\). By definition, \((g_1+g_2)(c) \cap R = G_c^+(\overline{c^0}, \varnothing)\), so \((g_1+g_2)(c)\cap R \subseteq \overline{f_1+ f_2+ G_c^+}(\overline{c^0}, \varnothing)\). The reverse inclusion follows from Corollary~\ref{saturation function for cycles} and the fact that
\begin{align*}
(f_1+f_2+G_c^+)(\overline{c^0}, \varnothing) &= f_1(\overline{c^0}, \varnothing) + f_2(\overline{c^0}, \varnothing) + G_c^+(\overline{c^0}, \varnothing) \\ &= g_1(c)\cap R + g_2(c)\cap R + (g_1(c)+g_2(c))\cap R \subseteq (g_1+g_2)(c)\cap R).
\end{align*}
Lastly, if \(c_\downarrow \neq \varnothing\), we find
\begin{align*}
g_1(c)+g_2(c) &\subseteq f_1(c_\downarrow, \varnothing)[x,x^{-1}] + f_2(c_\downarrow, \varnothing)[x,x^{-1}] \\&= (f_1+f_2)(c_\downarrow, \varnothing)[x,x^{-1}] \subseteq \overline{f_1+ f_2+ G_c^+}(c_\downarrow, \varnothing)[x,x^{-1}],
\end{align*}
 so \((\overline{f_1+ f_2+ G_c^+}, g_1+g_2)\) is an element of \(\mathscr{D}_{E,R}\).

Finally we need to show that \((\overline{f_1+ f_2 + G_c^+}, g_1+ g_2)\) is indeed the least upper bound of \((f_1, g_1)\) and \((f_2, g_2)\), but this is clear from its definition.
\end{proof}

From now on, for each \(u\in \partial E\), we will naturally identify the group algebra \(RG_u^u\) with \(R\) if \(u\) is irrational or finite, and with \(R[x,x^{-1}]\) if \(u\) is rational. We introduce the shorthand \((H,\{v\}_\varnothing)\) to mean \((H,\{v\})\) if \(v\notin H\), and \((H,\varnothing)\) if \(v\in H\); i.e. \((H, \{v\}_\varnothing) = (H, \{v\}\setminus H)\). Recall that \(\Delta_{G_E,R}\) is the image of the disassembly map from Theorem \ref{disassembly-theorem}. We shall define the following two maps: 
\begin{align} \label{Psi}
	\Psi: \mathscr{D}_{E,R} &\to  \Delta_{G_E,R} \\ \notag
\end{align}
where
\begin{align*}
	\small \Psi(f,g)\cap RG_u^u= 
    \begin{cases}
    \bigcup_{v\in u^0} f\Big(\overline{\small\{v \small\}}, \varnothing\Big) & \text{for } u \text{ irrational} \\
    f\Big(\overline{\{ r(u) \}},\varnothing\Big) & \text{for } u \text{ finite, } r(u) \text{ a sink} \\
    \bigcup_{F\subseteq_{\rm finite} s^{-1}(r(u))}f\Big(\overline{r(s^{-1}(r(u))\setminus F)}, \{r(u)\}_\varnothing\Big) & \text{for } u \text{ finite, } r(u) \text{ an inf.\! emitter} \\
    g(c) & \text{for } u=\rho c^\infty, c\in C_u(E)\\
    f\big(\overline{c^0},\varnothing\big) [x,x^{-1}] & \text{for } u=\rho c^\infty, c\notin C_u(E),
  \end{cases}
\end{align*}
 and
\begin{align} \label{Phi}
\Phi: \Delta_{G_E,R} &\to \mathscr{D}_{E,R}\\ \notag
Y &\mapsto (f,g)
\end{align}
where
\begin{align*}
f(H,S) &= \bigg(\bigcap_{\substack{u \in \partial E \\ u^0 \cap H \ne \varnothing }} (Y \cap RG_u^u \cap R)\bigg)\cap\bigg(\bigcap_{v \in S}(Y \cap RG_v^v)\bigg) & \text{for } (H,S) \in \mathscr{T}_E^*,\\
g(c) &= Y \cap RG_{c^\infty}^{c^\infty} & \text{for } c \in C_u(E).
\end{align*}
We will sometimes write $\Phi(Y) = (\Phi(Y)_f, \Phi(Y)_g)$.

{Note that for \((f,g)\in \mathscr{D}_{E,R}\) and \(u\) a finite boundary path that ends in a sink, we have \[\Psi(f,g) \cap RG^u_u = f\Big(\overline{\{ r(u) \}},\varnothing\Big) = \bigcup_{v\in u^0} f\Big(\overline{\small\{v \small\}}, \varnothing\Big),\] by the order-reversing property of \(f\). In particular, this is the same expression as in the case where~\(u\) is an irrational boundary path. In some proofs we will need to distinguish between different kind of boundary paths, and this observation will allow us to treat the cases of irrational paths and finite paths that end in a sink together.}

The following theorem is the main reason we introduce these maps:

\begin{theorem}\label{phi is lattice isomorphism}
The maps $\Phi$ and $\Psi$ are mutually inverse lattice isomorphisms.
\end{theorem}

The proof of this theorem will be given in the Lemmas \ref{range psi}--\ref{phi_psi}. The work lies in showing that \(\Psi\) and \(\Phi\) are mutually inverse bijections, since it can easily be checked that they preserve the order.

\begin{lemma}\label{range psi}
If $(f,g) \in \mathscr{D}_{E,R}$, then $\Psi(f,g) \in \Delta_{G_E,R}$.
\end{lemma}
\begin{proof}
By Theorem \ref{open equivalence}, we have to check the conditions (\hyperref[D1]{D1}), (\hyperref[D2]{D2}) and (\hyperref[L1]{L1})--(\hyperref[L3]{L3}). Let \(u\in \partial E\) be a unit. It is clear that \(\Psi(f,g) \cap RG_u^u\trianglelefteq RG_u^u\), because $\Psi(f,g) \cap RG_u^u$ is either an ideal in the image of $f$ or $g$, or the union of an ascending chain of ideals in the image of $f$.

Next we will show that \(\Psi(f,g)\cap RG_u^u\) and \(\Psi(f,g)\cap RG_v^v\) are canonically isomorphic if \(u\) and \(v\) are tail equivalent. If \(u\) (and thus \(v\)) is rational or ends in an infinite emitter, this follows by noting that \(\Psi(f,g)\cap RG_u^u\) only depends on the recurring cycle or the infinite emitter at which \(u\) ends. If \(u\) is irrational or ends in a sink, it follows from the order-reversing property of \(f\), since \(f(\overline{\{v_1\}},\varnothing) \subseteq f(\overline{\{v_2\}},\varnothing)\) if there is a path from the vertex \(v_1\) to the vertex \(v_2\).

For (\hyperref[L1]{L1}), if \(u\) is irrational and $r \in \Psi(f,g)\cap RG_u^u$, we can choose the first vertex \(v\) on \(u\) such that \(r\in f(\overline{\{v\}},\varnothing)\). Indeed, let \(z\) be a boundary path such that \(v\in z^0\). If \(z\) is irrational or ends in a sink, we have \(r\in f(\overline{\{v\}}, \varnothing)\subseteq \Psi(f,g)\cap RG_z^z\). If \(z = \rho c^\infty\) is rational, we see that \((\overline{c^0}, \varnothing) \leq (\overline{\{v\}}, \varnothing)\), so \(r\in f(\overline{\{v\}}, \varnothing) \subseteq f(\overline{c^0}, \varnothing) = \Psi(f,g)\cap RG_z^z\cap R\), regardless of whether \(c\) is an element of \(C_u(E)\) or not. And if \(z\) ends in an infinite emitter, the similar statement follows from the fact that \((\overline{r(s^{-1}(r(z)))}, \{r(z)\}_\varnothing)\leq (\overline{\{v\}}, \varnothing)\).

To show (\hyperref[L2]{L2}), let \(u=\rho c^\infty\in \partial E\) be rational, \(z\in \partial E\) a boundary path not tail equivalent to \(u\) such that \(z^0\cap c^0\neq \varnothing\), and \(r\in R\) a coefficient of some Laurent polynomial in \(\Psi(f,g)\cap RG_u^u\). If \(z\) is irrational or ends in a sink, \(z\) must contain an exit \(e\) of \(c\), so \(r(e)\in z^0\). If \(c\in C_u(E)\) we see that \(r\in f(c_\downarrow, \varnothing)\) since \(g(c)\subseteq f(c_\downarrow, \varnothing)[x,x^{-1}]\), and if \(c\notin C_u(E)\), then \(r\in f(\overline{c^0},\varnothing) = f(c_\downarrow, \varnothing)\). Since \(r(e)\in c_\downarrow\), the order-reversing property of \(f\) implies that \(r\in f(\overline{r(e)},\varnothing)\subseteq Y\cap RG_z^z\). If \(z=\rho \gamma^\infty\) is rational but not tail equivalent to \(u\), \(z\) must also contain an exit of \(c\), so \((\overline{\gamma}, \varnothing) \leq (c_\downarrow, \varnothing)\), and we can conclude the same as above, since \(g(\gamma)\cap R = f(\overline{\gamma^0}, \varnothing)\) if \(\gamma\in C_u(E)\). Lastly, suppose \(z\) ends in an infinite emitter and that there exists a vertex \(v\in z^0\cap c^0\). If \(v \ne r(z)\), we can use the same reasoning as above. On the other hand, suppose \(v=r(z)\), and let \(e\) be the unique edge of \(c\) such that \(s(e)=v\). Then we have \((r(s^{-1}(v)\setminus\{e\}), \{e\}_\varnothing) \leq (c_\downarrow, \varnothing)\), and we can conclude that \(r\in f(c_\downarrow,\varnothing) \subseteq \Psi(f,g)\cap RG_z^z\).

Finally, suppose \(u\) ends in an infinite emitter and \(r\in f(\overline{r(s^{-1}(r(u))\setminus F)}, \{r(u)\}_\varnothing)\) for some finite \(F\subseteq s^{-1}(r(u))\). Then (\hyperref[L3]{L3}) follows in all cases from the fact that \((\overline{v}, \varnothing)\leq (\overline{r(s^{-1}(r(u))\setminus F)}, \{r(u)\}_\varnothing)\) if \(v\in r(s^{-1}(r(u))\setminus F)\), and making use of the fact that \((\overline{c^0}, \varnothing)\leq (\overline{v}, \varnothing)\) for \(z=\rho c^\infty\).
\end{proof}

\begin{lemma}
If $Y \in \Delta_{G_E,R}$, then $\Phi(Y) \in \mathscr{D}_{E,R}$
\end{lemma}
\begin{proof}
As \(\Phi(Y)_f(H,S)\) is the intersection of ideals of \(R\), \(\Phi(Y)_f(H,S)\) is itself an ideal of \(R\), for any \((H,S)\in \mathscr{T}_E^*\). Consequently, \(\Phi(Y)_f\) is a function from \(\mathscr{T}_E^*\) to \(\mathcal{L}(R)\). We want to show that \(\Phi(Y)_f\) is saturated, i.e.\! that for any number of admissible pairs \((H_i,S_i)\in \mathscr{T}_E^*\) the equality \(\Phi(Y)_f(\bigvee_i (H_i,S_i)) = \bigcap_i \Phi(Y)_f(H_i,S_i)\) holds. We know that for each \((H,S)\in \mathscr{T}_E^*\), \(\Phi(Y)_f = \bigcap_u (Y\cap RG_u^u \cap R)\), where \(u\) ranges over the boundary paths that contain a vertex of \(H\), together with the vertices of \(S\) considered as boundary paths. As such, it will be enough to show that the set of boundary paths considered in the intersection for \(\Phi(Y)_f(\bigvee_i (H_i,S_i))\) is the same as the union of the sets of boundary paths considered for each \(\Phi(Y)_f(H_i,S_i)\). For simplicity, denote \(\bigvee_i (H_i, S_i)\) by \((H',S')\). Clearly, if \(u\) is a boundary path that contains a vertex of some \(H_i\) or is a vertex in some \(S_i\), the same can be said for \(H'\) or \(S'\). Conversely, if \(u\) is a vertex in \(S'\), it must be a vertex in some \(S_i\), and if \(u\) is a boundary path that contains a vertex of \(H'\), it follows from the construction of the \(S'\)-saturation (see (\ref{S-saturation})) that \(u\) also contains a vertex of some \(H_i\) or some~\(S_i\). Consequently, \(\Phi(Y)_f\in \mathscr{F}_{E,R}\).

Since \(\Phi(Y)_g(c) = Y\cap RG_{c^\infty}^{c^\infty}\) for each \(c\in C_u(E)\), \(\Phi(Y)_g\) is a function from \(C_u(E)\) to \(\mathcal{L}(R[x,x^{-1}])\). We want to show that for a cycle \(c\in C_u(E)\) the ideals \(\Phi(Y)_g(c)\cap R = Y\cap RG_{c^\infty}^{c^\infty} \cap R\) and \(\Phi(Y)_f(\overline{c^0}, \varnothing) = \bigcap_{u\in \partial E: u^0\cap \overline{c^0}\neq \varnothing} Y \cap RG_u^u\cap R\) are the same. Since \(c^\infty\) obviously contains a vertex of \(\overline{c^0}\), this gives us the inclusion \(\Phi(Y)_f(\overline{c^0}, \varnothing)\subseteq \Phi(Y)_g(c) \cap R\). For the reverse inclusion, suppose we have \(r\in \Phi(Y)_g(c)\), for some \(r\in R\), and \(u\in \partial E\) such that \(u^0\cap \overline{c^0}\neq \varnothing\). If \(u\) and \(c^\infty\) are not tail equivalent, (\hyperref[L2]{L2}) tells us that \(r\in Y\cap RG_u^u\cap R\), and if \(u\) and \(c^\infty\) are tail equivalent, the same follows, as \(Y\cap RG_u^u\) and \(Y\cap RG_{c^\infty}^{c^\infty}\) are naturally isomorphic.

Finally, suppose we have a cycle \(c\in C_u(E)\) such that \(c_\downarrow\neq \varnothing\). We need to show that the coefficients of the Laurent polynomials in \(\Phi(Y)_g(c)\) are elements of \(\Phi(Y)_f(c_\downarrow, \varnothing) = \bigcap_{u\in\partial E: u^0\cap c_\downarrow \neq \varnothing} Y\cap RG_u^u\cap R\). But this follows immediately from (\hyperref[L2]{L2}), as \(c\in C_u(E)\) implies that \(c\) cannot be tail equivalent to a boundary path \(u\) that is considered in this intersection.
\end{proof}

\begin{lemma}\label{psi_phi}
For all $Y \in \Delta_{G_E,R}$, we have $\Psi(\Phi(Y) = Y$.
\end{lemma}
\begin{proof}
We will show that for all \(u\in \partial E\), \(\Psi(\Phi(Y))\cap RG_u^u = Y\cap RG_u^u\). If \(u\) is irrational or finite with \(r(u)\) a sink, we have
\[
\Psi(\Phi(Y))\cap RG_u^u = \bigcup_{v\in u^0} \bigg(\bigcap_{\substack{u'\in \partial E\\ (u')^0\cap \overline{\{v\}}\neq \varnothing}} (Y\cap RG_{u'}^{u'}\cap R)\bigg).
\]
If \(r\in Y\cap RG_u^u\), there exists a vertex \(v\in u^0\) such that \(r\in Y\cap RG_z^z\) for all \(z\in \partial E\) that contain~\(v\) (for finite paths this is \(r(u)\), and for infinite paths we get it by (\hyperref[L1]{L1})). But then we have \(r\in \bigcap_{u'\in \partial E: (u')^0\cap \overline{\{v\}}\neq \varnothing} (Y\cap RG_{u'}^{u'}\cap R)\), since every path that contains a vertex of \(\overline{\{v\}}\) is tail equivalent to a path that contains \(v\), so \(r\in \Psi(\Phi(Y))\cap RG_u^u\). On the other hand, if \(r\in \Psi(\Phi(Y))\cap RG_u^u\), there exists a vertex \(v\in u^0\) such that \(r\in \bigcap_{u'\in \partial E: (u')^0\cap \overline{\{v\}}\neq \varnothing} (Y\cap RG_{u'}^{u'}\cap R)\). But clearly \(u^0\cap \overline{\{v\}}\neq \varnothing\), so \(r\in Y\cap RG_u^u\).

If \(u\) is finite and \(r(u)\) is an infinite emitter, we have
\[
\Psi(\Phi(Y))\cap RG_u^u = \bigcup_{F\subseteq_{\rm finite} s^{-1}(r(u))} \bigg(\bigcap_{u'} (Y\cap RG_{u'}^{u'}\cap R)\bigg),
\]
where \(u'\) ranges over all boundary paths that contain a vertex of \(\overline{r(s^{-1}(r(u))\setminus F)}\), and the path \(r(u)\) itself. So if there is such a finite set \(F\) for which this intersection contains some \(r\in R\), then \(r\) will be an element of \(Y\cap RG_{r(u)}^{r(u)}\). Since \(u\) is tail equivalent with \(r(u)\), \(r\) will also be in \(Y\cap RG_u^u\). On the other hand, if \(r\in Y\cap RG_u^u\), then (\hyperref[L3]{L3}) gives us such a finite set \(F\) for which \(r\) will be in the corresponding intersection.

For \(u=\rho c^\infty\in \partial E\) rational with \(c\in C_u(E)\) we have \(\Psi(\Phi(Y)) \cap RG_{u}^u=Y\cap RG_{c^\infty}^{c^\infty}\). But this is the same as \(Y\cap RG_{\rho c^\infty}^{\rho c^\infty}\), since \(u\) and \(c^\infty\) are tail equivalent.

Lastly, if \(u=\rho c^\infty\) is rational with \(c\notin C_u(E)\), we have \(\Psi(\Phi(Y))\cap RG_u^u= \Phi(Y)_f(\overline{c^0},\varnothing)[x,x^{-1}]\). If \(r\in \Phi(Y)_f(\overline{c^0},\varnothing)\), we also have \(r\in Y\cap RG_u^u\), since \(u^0\cap c^0\neq \varnothing\). On the other hand, if we have some Laurent polynomial in \(Y\cap RG_u^u\), its coefficients will be in \(\Phi(Y)_f(\overline{c^0},\varnothing)\): this follows from (\hyperref[L2]{L2}) and Lemma \ref{lemma to prove some condition}.
\end{proof}

\begin{lemma}\label{phi_psi}
For all \((f,g)\in \mathscr{D}_{E,R}\), we have \(\Phi(\Psi(f,g))=(f,g)\).
\end{lemma}
\begin{proof}
Suppose we have some pair \((H,S)\in \mathscr{T}_E^*\). Then
\[
\Phi(\Psi(f,g))_f(H,S) = \bigcap_u (\Psi(f,g)\cap RG_u^u\cap R),\]
where \(u\) ranges over the boundary paths that contain a vertex of \(H\), and the vertices in \(S\), considered as boundary paths. Suppose we have \(r\in f(H,S)\) for some \(r\in R\), and suppose firstly that $u \in \partial E, v \in u^0 \cap H$. We proceed to show that $r \in \Psi(f,g)\cap RG_u^u\cap R$.
If \(u\) is irrational or finite with \(r(u)\) a sink, then we have \(f(H,S)\subseteq f(\overline{\{v\}},\varnothing)\), as \(f\) is order-reversing. Since \(\Psi(f,g) \cap RG_u^u\) contains \(f(\overline{\{v\}},\varnothing)\) as a subset, it also contains \(r\). If \(u\) is finite and ends in an infinite emitter \(r(u) \in H\), we can conclude the same thing by choosing \(F=\varnothing\) in the definition of \(\Psi\) and observing that $\overline{r(s^{-1}(r(u)))} \subseteq H$, so $(\overline{r(s^{-1}(r(u)))}, \{r(u)\}_\varnothing) \le (H,S)$. If \(u=\rho c^\infty\) is rational and contains a vertex \(v\) of \(H\), then we must have \(\overline{c^0}\subseteq H\), so \(f(H,S)\subseteq f(\overline{c^0},\varnothing)\). If \(c\notin C_u(E)\) this is enough, and otherwise we can just note that \(f(\overline{c^0},\varnothing) = g(c)\cap R\) and \(\Psi(f,g)\cap RG_u^u = g(c)\). Lastly, if \(u \in S \), the fact that \(u\) is a breaking vertex for \(H\) implies that \((r(s^{-1}(u)\setminus F), \{u\}_\varnothing )\leq (H,S)\) for some finite  \(F = s^{-1}(u)\cap r(E^0\setminus H)\), so \(r\in f(H,S)\subseteq f(r(s^{-1}(u)\setminus F),\{u\}_\varnothing )\subseteq \Psi(f,g)\cap RG_u^u\). We can conclude that \(f(H,S)\subseteq \Phi(\Psi(f,g))_f(H,S)\).

Next, suppose we have \(r\in \bigcap _u (\Psi(f,g)\cap RG_u^u\cap R)\), where \(u\) ranges over the boundary paths that contains a vertex of \(H\), and the vertices in \(S\). We will show that \(r\in \bigcap_\alpha f(H_\alpha,S_\alpha)\) for some pairs \((H_\alpha,S_\alpha)\in \mathscr{T}_E^*\) with \(\bigvee_\alpha(H_\alpha,S_\alpha) = (H,S)\), and conclude that $r \in f(H,S)$ using the property from Definition \ref{def:saturated function}.

First, by the definition of $\Psi$, for any \(w\in S\), we can find a finite set \(F\subseteq s^{-1}(w)\) such that \(r\in f(\overline{r(s^{-1}(w)\setminus F)}, \{w\}_\varnothing)\). If we define \(H_w\) as \(r(s^{-1}(w)\setminus F)\cap H\), we find that \((H_w,\{w\})\leq (r(s^{-1}(w)\setminus F),\{w\}_\varnothing )\), so \(r\in f(\overline{r(s^{-1}(w)\setminus F)}, \{w\}_\varnothing) \subseteq f(H_w,\{w\})\). Secondly, for any boundary path \(u\) that contains a vertex of \(H\) and does not end in an infinite emitter, we can find a vertex \(v\in u^0\cap H\) such that \(r\in f(\overline{\{v\}},\varnothing)\): for irrational paths and finite paths that end in a sink this follows from the definition of \(\Psi\). For rational paths we can choose \(v = s(c)\), since \(\overline{\{s(c)\}} = \overline{c^0}\) and \(r \in g(c)\cap R = f(\overline{c^0},\varnothing)\) if \(c\in C_u(E)\). Let \(V\) be the set of vertices \(v\in H\) such that \(r\in f(\overline{\{v\}}, \varnothing)\), and consider the set \(\overline{V}\). In particular, we know that \(V\), and thus also \(\overline{V}\), contains a vertex of each boundary path having a vertex in \(H\) and not ending at an infinite emitter. Also, $r \in f(\overline{V},\varnothing)$ by the defining property of a saturated function. Lastly, for any infinite emitter \(z\in H\), we can find a finite set \(F\subseteq s^{-1}(z)\) such that \(r\in f(H_z, \{z\}_\varnothing)\), with \(H_z=r(s^{-1}(z)\setminus F)\).

So far we have found that
\[
r\in \bigcap_{w\in S} f\big(H_w,\{w\}\big) \ \cap\ f\big(\overline{V},\varnothing\big)\ \cap\ \smashoperator{\bigcap_{\substack{z \in H \\ \text{inf.\! emit.} }}}f\big(H_z, \{z\}_\varnothing\big).
\]

If we can prove that
\begin{equation*}
\label{supremum}(H,S) = \bigvee_{w\in S} (H_w,\{w\})\ \vee\ (\overline{V},\varnothing) \ \vee \ \smashoperator{\bigvee_{\substack{z \in H \\ \text{inf.\! emit.} }}} f(H_z,\{z\}_\varnothing),
\end{equation*}
the saturatedness of \(f\) will imply \(\Phi(\Psi(f,g))_f = f\). Since each of the pairs over which we take the supremum is less than or equal to \((H,S)\), their supremum is also less than or equal to \((H,S)\). Denote this supremum by \((H',S')\), and suppose that there exists a vertex \(v\in H\) that is not in \(H'\). Then \(v\) cannot be a sink, since all sinks in \(H\) must be in \(V\subseteq H'\). If \(v\) is regular, there must be an edge \(e_1\in s^{-1}(v)\) such that \(r(e)\notin \overline{V}\), because \(\overline{V}\) is saturated. If \(v\) is an infinite emitter, there must also be an edge \(e_1\in s^{-1}(v)\) such that \(r(e)\notin H'\), and thus \(r(e)\notin \overline{V}\), because of the saturation property of the supremum of such pairs. Continuing in this fashion, we can construct an infinite path with vertices in \(H\setminus V\), but this contradicts the definition of \(V\). Consequently, \(H=H'\), and it is easy to see that \(S=S'\).

Finally, for each cycle \(c\in C_u(E)\), we have \(\Phi(\Psi(f,g))_g(c) = \Psi(f,g)\cap RG_{c^\infty} ^{c^\infty} = g(c)\), by definition of \(\Phi\) and \(\Psi\). So \(\Phi(\Psi(f,g))_g = g\), which concludes the proof.
\end{proof}

Theorems \ref{disassembly-theorem} and \ref{phi is lattice isomorphism} and the fact that $L_R(E) \cong A_R(G_E)$ give us lattice isomorphisms
\begin{equation} \label{composing isomorphisms}	
	\begin{tikzcd}
	\mathcal{L}(L_R(E)) \ar[r,"\sim"] & \mathcal{L}(A_R(G_E)) \ar[rr,"\sim","\text{\tiny disassembly}" below] && \Delta_{G_E,R} \ar[r,"\Phi" below, "\sim" above] &\mathscr{D}_{E,R}.
	\end{tikzcd}
\end{equation}
The following theorem gives an explicit description of the isomorphism $\mathcal{L}(L_R(E))  \overset{\sim}{\longrightarrow} \mathscr{D}_{E,R}$.

\begin{notation}
For a hereditary saturated subset \(H\) and a breaking vertex \(w\) for \(H\), recall the notation \(w^H = w- \sum_{e\in s^{-1}(w), r(e)\notin H} ee^*\). In general, if \(A\) is a ring or algebra, and \(S\subseteq A\) a subset, the ideal of \(A\) generated by \(S\) will be denoted by \(\langle S\rangle\).
We use the convention that if $c$ is a cycle based at~$v$, and $p(x) = \sum_{n \in \ZZ} a_n x^n \in R[x,x^{-1}]$, then $p(c) = \sum_{n < 0} a_{n}(c^*)^{-n} + a_0 v + \sum_{n > 0}a_n c^n \in L_R(E)$.
\end{notation}

Our main theorem is the following.

\begin{theorem} \label{LPA ideal correspondence}
There is a lattice isomorphism \(\mathcal{L}(L_R(E)) \to \mathscr{D}_{E,R}\), that sends an ideal \(I\) to the pair \((f,g)\), with
\begin{align} \label{f}
f(H,S)&=\big\{r\in R\mid rv\in I\ \forall v\in H \text{ and } rw^H\in I \ \forall w\in S\big\}, \\
\label{g} g(c)&=\big\{p(x)\in R[x,x^{-1}]\mid p(c)\in I\big\}.	
\end{align}
 The inverse of this map is given by 
\[(f,g)\longmapsto \left\langle \bigcup_{(H,S)\in \mathscr{T}_E^*} \big\{rv, rw^H\mid r\in f(H,S), v\in H, w\in S\big\}\cup \big\{p(c)\mid c\in C_u(E), p(x)\in g(c)\big\}\right\rangle.\]
\end{theorem}
\begin{proof}
The lattice isomorphism is gotten by composing the arrows in line (\ref{composing isomorphisms}). Let \((f,g)\in \mathscr{D}_{E,R}\) be the image of an ideal \(I\) under this lattice isomorphism. Then we have \(g(c)=\{p(x)\in R[x,x^{-1}]\mid p(c)\in I\}\) by definition of \(\Phi\) (see (\ref{Phi})), and by Proposition \ref{correspondence: proposition} (\ref{correspondence: rational}). This explains (\ref{g}). Next, we will explain (\ref{f}).

 For \((H,S)\in \mathscr{T}_E^*\), the definition of \(\Phi\) gives
\[
f(H,S)=\bigcap_{\substack{u\in \partial E\\ u^0\cap H\neq \varnothing}} (Y\cap RG_u^u\cap R) \cap \bigcap_{v\in S} (Y\cap RG_v^v),
\]
 where \(Y\) is the image of \(I\) under the disassembly map. We can rewrite this as
 \[
 \bigcap_{v\in H} \bigg( \bigcap_{\substack{u\in \partial E\\v\in u^0}}(Y\cap RG_u^u\cap R) \bigg)\cap \bigcap_{w\in S}(Y\cap RG_w^w).
 \] By Proposition \ref{correspondence: proposition} (\ref{correspondence: vertices}) we know that \(\bigcap_{v\in H} \left( \bigcap_{u\in \partial E:v\in u^0}(Y\cap RG_u^u\cap R) \right)\) is the set of scalars \(r\) such that \(rv\in I\), for all \(v\in H\). Now suppose we have \(r\in Y\cap RG_w^w\), with \(w\in S\) and \(rv\in I\) for all \(v\in H\). By Proposition \ref{correspondence: proposition} (\ref{correspondence: ends in infinite emitter}), there exists a finite set \(F\subseteq s^{-1}(w)\) such that \(r\bm{1}_{\mathcal{Z}(w,w,F)}\in I\), or equivalently, \(r(w-\sum_{f\in F}ff^*)\in I\) in the language of Leavitt path algebras. If there were to exist an edge \(e\in F\) such that \(r(e)\in H\), we would also get \(rer(e)e^*=ree^*\in I\), and consequently \(r(w-\sum_{f\in F\setminus \{e\}} ff^*)\in I\), so we may assume that \(r(f)\notin H\) for all \(f\in F\). Now let \(F'\subseteq s^{-1}(w)\) be the finite set of edges such that \(r(e)\notin H\), for all \(e\in F'\). In particular, we have \(F\subseteq F'\). Since \(r(w-\sum_{f\in F}f f^*)\) is an element of \(I\), \(r(w-\sum_{f\in F} ff^*)(w-\sum_{e \in F'}ee^*)\) must also be in \(I\). But this is equal to \[r(w^2 - w{\sum_{e\in F'}}ee^* -\sum_{f\in F} ff^* w + \smashoperator{\sum_{e\in F',f\in F}} ff^*ee^*) = r(w-\sum_{e\in F'}ee^* - \sum_{f\in F}ff^* + \sum_{f\in F} ff^*) = rw^H,\] since \(F\subseteq F'\). The converse also holds: if \(rw^H\in I\), with \(r\in R\) and \(w\in S\), then \(r\in Y\cap RG_w^w\) by Proposition \ref{correspondence: proposition} (\ref{correspondence: ends in infinite emitter}). It follows that \(f(H,S)\) is exactly the set of scalars \(r\in R\) for which \(rv\in I\) and \(rw^H\in I\), for all \(v\in H\) and \(w\in S\).

The inverse map is then given as in the statement of the theorem, making use of Corollary~\ref{ideal is generated by vertices and cycles}.
\end{proof}

\subsubsection{Special cases of the ideal correspondence theorem}

If we restrict to the case where \(E\) is row-finite and \(C_u(E)=\varnothing\), we can use this theorem to get back \cite[Theorem 6.1 (b)]{clark2019ideals}. 

In the case where \(R=K\) is a field, the lattice of ideals of \(L_K(E)\) is described in \cite[Theorem 2.8.10]{abrams2005leavitt}: each ideal \(I\trianglelefteq L_K(E)\) corresponds to a triple \(((H,S), C, P)\), where \((H,S)\) is an admissible pair, \(C\subseteq C_u(E)\) a set of cycles whose exits have their range in \(H\), and \(P\) assigns to each \(c\in C\) a non constant polynomial \(p_c(x)\in K[x]\) with constant term 1. For readers who are familiar with this theorem, we explain briefly how we can deduce it from our Theorem \ref{LPA ideal correspondence}.

Consider an ideal \(I\), corresponding to \((f,g)\in \mathscr{D}_{E,R}\). We want to show \(I\) corresponds to a unique triple \(((H,S),C,P)\). For such a triple, the set \(H\) of vertices would consist of the vertices \(v\) for which \(v\in I\), and \(S\) would contain the breaking vertices \(w\) of \(H\) for which \(w^H\in I\). So we need to set \((H,S) = \max\left\{(H',S') \in \mathscr{T}_E^* \mid f(H',S')=K \right\}\). (This maximum exists by the saturation property of \(f\).) Next, the set \(C\) must consist of those cycles \(c\) whose vertices are not contained in \(I\), but for which there exists a nonzero polynomial \(p\in K[x]\) such that \(p(c)\in I\). So we can set \(C\subseteq C_u(E)\) to be the cycles which are mapped by \(g\) to a proper nonzero ideal of \(K[x,x^{-1}]\). For such a cycle \(c\), as \(K[x,x^{-1}]\) is a principal ideal domain, this ideal is uniquely determined by some non-constant polynomial in \(K[x]\) with constant term 1. This should  be exactly the polynomial in \(P\) corresponding to \(c\), and we easily determine it by looking at \(g(c)\).

Conversely, if we have such a triple \(((H,S),C,P)\), the corresponding ideal will correspond to the pair \((f,g)\) defined by
\begin{align*}f(H',S') = \begin{cases}
K & \text{if } (H',S')\leq (H,S)\\
0 & \text{else.}
\end{cases} && 
g(c) = \begin{cases}
K[x,x^{-1}] & \text{if } c^0 \subseteq H\\
\langle p_c(x)\rangle & \text{if } c\in C\\
0 & \text{else.}
\end{cases}
\end{align*}

\subsection{Products of ideals}

Since we have a lattice isomorphism, if we know which pairs \((f_I,g_I)\) and \((f_J,g_J)\) correspond to the ideals \(I\) and \(J\) respectively, Proposition \ref{lattice of pairs of functions} tells us which pairs correspond to the ideals \(I+J\) and \(I\cap J\). An obvious next step would be to find the pair corresponding to \(IJ\), the product of \(I\) and \(J\). This is done in the next proposition. Recall that the product \(f_1f_2\) of two $\mathcal{L}(R)$-valued functions \(f_1\) and \(f_2\) is defined as the pointwise multiplication of \(f_1\) and \(f_2\).

\begin{proposition} \label{multiplication of Leavitt ideals}
Let \(I\) and \(J\) be ideals in \(L_R(E)\), with corresponding pairs \((f_I,g_I)\) and \((f_J,g_J)\). Then the pair corresponding to the ideal \(IJ\) is \((\overline{f_If_J + G_c^\times},g_Ig_J)\), with \(G_c^\times: \mathscr{T}_E^*\to \mathcal{L}(R)\) the function that maps \((\overline{c^0},\varnothing)\) to \((g_1(c)g_2(c))\cap R\) for \(c\in C_u(E)\), and the rest to the zero ideal.
\end{proposition}

\begin{proof}
First of all, we will show that \(f_If_J + G_c^\times\) is order-reversing, so its saturation \(\overline{f_If_J + G_c^\times}\) is defined and is an element of \(\mathscr{F}_{E,R}\). Since \(f_If_J\) is clearly order-reversing, it is enough to show that if \((H,S) < (\overline{c^0}, \varnothing)\) for some \((H,S)\in \mathscr{T}_E^*\) and \(c\in C_u(E)\), then \((g_I(c)g_J(c)) \cap R \subseteq f_I(H,S)f_J(H,S)\). But for such \((H,S)\) we have \((H,S)\leq (c_\downarrow, \varnothing)\), so this follows from the fact that \(g_I(c)\subseteq f_I(c_\downarrow, \varnothing)[x,x^{-1}]\), i.e. all the coefficients of the Laurent polynomials in \(g_I(c)\) are elements of \(f_I(c_\downarrow, \varnothing)\), and similarly for \(J\).

To show that \((\overline{f_If_J+ G_c^\times}, g_Ig_J)\) is an element of \(\mathscr{D}_{E,R}\), consider some \(c\in C_u(E)\). Clearly, we have \(g_Ig_J(c)\cap R = G_c^\times(\overline{c^0}, \varnothing) \subseteq \overline{f_If_J + G_c^\times}(\overline{c^0}, \varnothing)\). For the reverse inclusion, we can apply Corollary \ref{saturation function for cycles}, and note that
\begin{align*}
(f_If_J + G_c^\times) (\overline{c^0}, \varnothing) &= f_I(\overline{c^0}, \varnothing) f_J(\overline{c^0}, \varnothing) + G_c^\times(\overline{c^0}, \varnothing) \\ &= (g_I(c)\cap R)(g_J(c)\cap R) + (g_I(c)g_J(c))\cap R \subseteq (g_I(c)g_J(c))\cap R.
\end{align*}
Finally, suppose \(c_\downarrow \neq \varnothing\). Then we have
\begin{align*}
g_Ig_J(c) \subseteq f_I(c_\downarrow, \varnothing)[x,x^{-1}]f_J(c_\downarrow, \varnothing)[x,x^{-1}] \subseteq (f_If_J)(c_\downarrow, \varnothing)[x,x^{-1}] \subseteq (\overline{f_If_J+G_c^\times})(c_\downarrow, \varnothing)[x,x^{-1}].
\end{align*}
So we can conclude that \((\overline{f_If_J + G_c^\times}, g_Ig_J)\in \mathscr{D}_{E,R}\).
To show that this pair corresponds to the ideal \(IJ\), we will make use of the map \(\Psi\). Consider the images \(Y_I\), \(Y_J\) and \(Y\) under \(\Psi\) of \((f_I,g_I)\), \((f_J,g_J)\) and \((\overline{f_If_J + G_c^\times}, g_Ig_J)\) respectively. If we can show that, for each \(u\in \partial E\), the ideals \((Y_I\cap RG_u^u)(Y_J\cap RG_u^u)\) and \(Y\cap RG_u^u\) are equal, the proposition will follow from Proposition \ref{product of ideals of isotropy groups}, line (\ref{multiplicativity}).

First suppose \(u\in \partial E\) is irrational, or ends in a sink. Then we have \(Y_I\cap RG_u^u = \bigcup_{v\in u^0} f_I(\overline{\{v\}}, \varnothing)\), \(Y_J\cap RG_u^u = \bigcup_{v\in u^0} f_J(\overline{\{v\}}, \varnothing)\), and \(Y\cap RG_u^u = \bigcup_{v\in u^0} (\overline{f_If_J + G_c^\times})(\overline{\{v\}}, \varnothing)\). Note that, if \(v_0,v_1,v_2, \ldots\) are the vertices of \(u\), then \(f(\overline{\{v_0\}}, \varnothing), f(\overline{\{v_1\}}, \varnothing), f(\overline{\{v_2\}}, \varnothing), \ldots\) is an ascending chain of ideals, for any \(f\in \mathscr{F}_{E,R}\). It immediately follows that \((Y_I\cap RG_u^u) (Y_J\cap RG_u^u)\subseteq Y\cap RG_u^u\). For the reverse inclusion, suppose we have a vertex \(v\in u^0\) and an element \(r\in R\) such that \(r\in (\overline{f_If_J + G_c^\times})(\overline{\{v\}}, \varnothing)\). By Proposition~\ref{construction saturation}, this means we can find a set of admissible pairs \(\{(H_\alpha, S_\alpha)\}_\alpha \subseteq \mathscr{T}_E^*\) such that \(\bigvee_\alpha (H_\alpha, S_\alpha) = (\overline{\{v\}}, \varnothing)\) and \(r\in \bigcap_{\alpha} (f_If_J + G_c^\times)(H_\alpha, S_\alpha)\). Since \(\bigvee_\alpha (H_\alpha, S_\alpha) = (\overline{\{v\}}, \varnothing)\), there must exist some \(\beta\) such that \(H_\beta\) contains a vertex \(w\in u^0\). But then we have \((\overline{\{w\}}, \varnothing) \leq (H_\beta, S_\beta)\), so \(r\in (f_If_J + G_c^\times)(H_\beta, S_\beta) \subseteq (f_If_J + G_c^\times)(\overline{\{w\}}, \varnothing) = f_If_J(\overline{\{w\}}, \varnothing)\). This means we can find \(a_n\in f_I(\overline{\{w\}}, \varnothing)\) and \(b_n\in f_J(\overline{\{w\}}, \varnothing)\) such that \(r=\sum_n a_nb_n\). Since \(f_I(\overline{\{w\}}, \varnothing) \subseteq Y_I\cap RG_u^u\), and similarly for \(J\), we can conclude that \(Y\cap RG_u^u\subseteq (Y_I\cap RG_u^u)(Y_J\cap RG_u^u)\).

Now suppose \(u\) is finite and ends in an infinite emitter. The inclusion \((Y_I\cap RG_u^u) (Y_J\cap RG_u^u) \subseteq Y\cap RG_u^u\) follows similarly as in the previous paragraph, by noting that, for \(F_I,F_J \subseteq _{\text{finite}} s^{-1}(r(u))\), we have \(f(\overline{r(s^{-1}(r(u))\setminus F_I)}, \{r(u)\}_\varnothing)\subseteq f(\overline{r(s^{-1}(r(u))\setminus (F_I\cup F_J)}, \{r(u)\}_\varnothing)\), and similarly for \(F_J\), for any \(f\in \mathscr{F}_{E,R}\). For the reverse inclusion, let \(r\in Y\cap RG^u_u\). Then using that \(Y=\Psi(\overline{f_If_J + G_c^\times}, g_Ig_J)\) and Proposition \ref{construction saturation}, we can find a finite set  \(F\subseteq s^{-1}(r(u))\) and a set of admissible pairs \(\{(H_\alpha, S_\alpha)\}_{\alpha \in A} \subseteq \mathscr{T}_E^*\) with 
\[
\bigvee_{\alpha \in A} (H_\alpha, S_\alpha) = \Big(\overline{r(s^{-1}(r(u))\setminus F)}, \{r(u)\}_\varnothing\Big),
\]
such that \(r\in \bigcap_{\alpha \in A} (f_If_J+ G_c^\times)(H_\alpha, S_\alpha)\). The equation above and Proposition \ref{prop:supremum in TE} imply that
\[
r(u) \in \Big(\overline{\bigcup_{\alpha \in A}H_\alpha}^{\bigcup_{\alpha \in A}S_\alpha}\Big)\cup\Big( \bigcup_{\alpha \in A}S_\alpha\Big).
\] Since \(r(u)\) is an infinite emitter, this means that there exists \(\beta \in A\) such that \(r(u)\in H_\beta \cup S_\beta\), by Observation \ref{observation: saturation} \eqref{observation: infinite emitter}. Consequently, we can find a finite set \(F'\subseteq s^{-1}(r(u))\) such that \((\overline{r(s^{-1}(r(u))\setminus F')}, \{r(u)\}_\varnothing) \leq (H_\beta, S_\beta)\), so we have 
	\begin{align*}r\in (f_If_J + G_c^\times) (H_\beta, S_\beta) &\subseteq (f_If_J+ G_c^\times)(\overline{r(s^{-1}(r(u))\setminus F')}, \{r(u)\}_\varnothing)\\ &= f_If_J(\overline{r(s^{-1}(r(u))\setminus F')}, \{r(u)\}_\varnothing).
	\end{align*}
	Again, we can find \(a_n\in f_I(\overline{r(s^{-1}(r(u))\setminus F')}, \{r(u)\}_\varnothing)\) and \(b_n\in f_J(\overline{r(s^{-1}(r(u))\setminus F')}, \{r(u)\}_\varnothing)\) such that \(r = \sum_n a_nb_n\), from which it follows that \(r\in (Y_I\cap RG_u^u)(Y_J\cap RG_u^u)\). 

Finally, suppose \(u=\rho c^\infty\) is rational. If \(c\in C_u(E)\), we have \(Y_I\cap RG_u^u = g_I(c)\), \(Y_J\cap RG_u^u = g_J(c)\) and \(Y\cap RG_u^u = g_Ig_J(c)\), so we immediately find \((Y_I\cap RG_u^u)(Y_J\cap RG_u^u) = Y\cap RG_u^u\). On the other hand, if \(c\notin C_u(E)\), we have \(Y_I\cap RG_u^u = f_I(\overline{c^0}, \varnothing)[x,x^{-1}]\), \(Y_J\cap RG_u^u = f_J(\overline{c^0}, \varnothing)[x,x^{-1}]\) and \(Y\cap RG_u^u = \overline{f_If_J+ G_c^\times} (\overline{c^0}, \varnothing)[x,x^{-1}] = f_If_J(\overline{c^0}, \varnothing)[x,x^{-1}]\), where we made use of Corollary~\ref{saturation function for cycles}. Since \(A[x,x^{-1}]B[x,x^{-1}]=AB[x,x^{-1}]\) holds in general for \(A,B \trianglelefteq R\), it follows that \(Y\cap RG_u^u= (Y_I\cap RG_u^u)(Y_J\cap RG_u^u)\), which concludes the proof.
\end{proof}

\begin{remark}
The reason we introduced the function \(G_c^\times\) in the previous proposition, is that for two ideals \(I,J\trianglelefteq R[x,x^{-1}]\), we know that \((I\cap R)(J\cap R) \subseteq (IJ\cap R)\), but equality does not hold in general. If we are working with a ring $R$ for which this equality holds for all such ideals, e.g.\! if $R$ is a field, then we can restate the previous proposition omitting all mention of \(G_c^\times\).

As an example of a ring for which this property does not hold, consider \(R=\CC[y,z]/\langle y^2, z^2 \rangle = \CC 1\oplus \CC y\oplus \CC z\oplus \CC yz \). Let \(I\) and \(J\) be the principal ideals of \(R[x,x^{-1}]\) generated by \(y+zx^{-1}\) and \(z+yx\) respectively. Then one can show that 
\begin{align*}
I&=\CC[x,x^{-1}]   yz \oplus \CC[x,x^{-1}]  ( y+zx^{-1} ) \\
J&=   \CC[x,x^{-1}]  yz \oplus \CC[x,x^{-1}] ( z+yx )
\end{align*}
and that \(I\cap R=J\cap R=\CC yz\). Consequently, \((I\cap R) (J\cap R) = 0\). However, as \((y+zx^{-1})(z+yx) = 2yz\in IJ\cap R\), there is a proper containment \((I\cap R)(J\cap R) \subsetneq IJ\cap R\).
\end{remark}

\subsection{Graded ideals}
The following theorem tells us how to see whether an ideal is graded or not by looking at the corresponding pair of functions.

\begin{theorem}
For any pair \((f,g)\in \mathscr{D}_{E,R}\) with corresponding ideal \(I\trianglelefteq L_R(E)\), the following are equivalent:
\begin{enumerate}[\rm(1)]
    \item \(I\) is graded.
    \item For each \(c\in C_u(E)\), \(g(c)=f(\overline{c^0})[x,x^{-1}]\).
    \item For each \(c\in C_u(E)\), \(g(c)\) is a graded ideal of \(R[x,x^{-1}]\).
\end{enumerate}
\end{theorem}
\begin{proof}
\((1)\Rightarrow (2)\) Suppose there is some \(c\in C_u(E)\) such that \(g(c)\neq f(\overline{c^0}, \varnothing)[x,x^{-1}]\). Since we know that \(g(c)\cap R = f(\overline{c^0}, \varnothing)\), it follows that there is some Laurent polynomial in \(g(c)\), such that its coefficients are not all in \(g(c)\). But the elements of \(g(c)\) are just the Laurent polynomials \(p(x)\) such that \(p(c)\in I\), where \(c^n = (c^*)^{-n}\) for \(n< 0\) and \(c^0 = s(c) = r(c)\). Consequently, there exists an element of \(I\) of which the homogeneous components are not in \(I\), so \(I\) is not graded. This gives a contradiction.

\((2)\Rightarrow (3)\) is clear.

\((3)\Rightarrow (1)\) If \(g(c)\) is graded for all \(c\in C_u(E)\), the description of the lattice isomorphism of Theorem \ref{LPA ideal correspondence} shows that \(I\) can be generated by homogeneous elements. As in any graded ring or algebra, this implies that \(I\) is graded.
\end{proof}

\begin{corollary}\label{classification of graded ideals}
The lattice \(\mathcal{L}_{\rm gr}(L_R(E))\) of graded ideals of \(L_R(E)\) is isomorphic to \(\mathscr{F}_{E,R}\).
\end{corollary}
\begin{proof}
It is easy to see that the sum and intersection of two graded ideals will also be graded. Suppose we have two graded ideals \(I\) and \(J\), with corresponding pairs \((f_I,g_I)\) and \((f_J,g_J)\). Then the intersection of \(I\) and \(J\) corresponds to the pair \((f_I\cap f_J, g_I\cap g_J)\), while the sum of \(I\) and \(J\) corresponds to \((\overline{f_I+ f_J+ G_c^+}, g_I+ g_J) = (\overline{f_I+ f_J}, g_I+g_J)\). The last equality follows from the fact that, for \(c\in C_u(E)\) we have \(G_c^+(\overline{c^0}, \varnothing) = (f_I(\overline{c^0}, \varnothing)[x,x^{-1}] + f_J(\overline{c^0}, \varnothing)[x,x^{-1}])\cap R \subseteq (f_I+ f_J) (\overline{c^0}, \varnothing)\). The statement then follows from the fact that the second function in a pair corresponding to a graded ideal is redundant, and that the sum and intersection of ideals correspond to the lattice \(\mathscr{F}_{E,R}\) as described in Section \ref{section: lattices}.
\end{proof}

\begin{corollary}
Let \(I\) be an ideal in \(L_R(E)\) and \((f,g)\) the corresponding pair in \(\mathscr{D}_{E,R}\). Then the ideal corresponding to the pair \((f,f_C)\), with \(f_C(c) = f(\overline{c^0}, \varnothing)[x,x^{-1}]\) for \(c\in C_u(E)\), is the largest graded ideal contained in \(I\).
\end{corollary}

\begin{remark}

When working with Leavitt path algebras over an arbitrary commutative ring with unit, the ideal lattice of so-called graded \emph{basic} ideals is known to be isomorphic to the lattice of admissible pairs of \(E\): see \cite[Theorem 7.9 (1)]{tomforde2011leavitt} for row-finite graphs, and \cite[Theorem 3.10 (4)]{larki} for the general case. A graded ideal \(I \in \mathcal{L}(L_R(E))\) is said to be \emph{basic} if \(rx\in I\) implies \(x\in I\), for any nonzero \(r\in R\) and any \(x\in L_R(E)\) of the form \(x = v\) or \(x=v-\sum_{i=1}^n e_ie_i^*\), with \(v\in E^0\) and \(s(e_i)=v\).
In our classification of graded ideals, Corollary \ref{classification of graded ideals}, it is clear that basic graded ideals correspond exactly to the functions \(f\in \mathscr{F}_{E,R}\) which take values in \(\{0,R\}\). The assignments
\[f \longmapsto \max\left\{(H,S) \in \mathscr{T}_E^* \mid f(H,S)=R \right\}\]
(this maximum exists by the saturation property of \(f\)) and
\[I \longmapsto \left( f: (H,S) \longmapsto \begin{cases}
R & \text{if } v\in I \text{ for } v\in H \text{ and } w^H\in I \text{ for } w\in S,\\
0 & \text{else}
\end{cases}\right)\]
define mutually inverse lattice isomorphisms between the lattice of functions in \(\mathscr{F}_{E,R}\) which take values in \(\{0,R\}\), and \(\mathscr{T}_E\). So we get back the known classification of basic graded ideals. In contrast with $\mathcal{L}_{\rm gr}(L_R(E))$ and $\mathcal{L}(L_R(E))$, which are influenced by the ideal structures of $R$ and $R[x,x^{-1}]$ respectively, the lattice of basic graded ideals of $L_R(E)$ does not depend on $R$.
\end{remark}

\subsection{Prime ideals} Finally, let us try to use our description of the ideal lattice and multiplication of ideals of \(L_R(E)\) to determine which ideals are prime. In the case where \(R\) is a field, a necessary and sufficient condition for an ideal to be prime was already given in \cite[Theorem 3.12]{rangaswamyprime}. While we have not yet found such a necessary and sufficient condition for Leavitt path algebras over arbitrary commutative rings, we will give some necessary conditions, and a counterexample to show these are not sufficient. Hopefully this will illustrate which kind of behaviour can occur in the general case. 

For simplicity, we will assume that \(E\) is a row-finite graph, and that \(E\) satisfies condition (K), i.e.\! \(C_u(E)=\varnothing\). In particular, we do not have to worry about breaking vertices or non-graded ideals. With these restrictions and in the case where \(R=K\) is a field, the condition previously mentioned is as follows: an ideal \(I\trianglelefteq L_K(E)\) is uniquely determined by the hereditary saturated subset \(H:= I\cap E^0\), and then \(I\) is prime if and only if \(H\neq E^0\) and \(E^0\setminus H\) is \emph{downward directed}, in the sense of the following definition:

\begin{definition}
A subset \(S\subseteq E^0\) is said to be \emph{downward directed} if for each \(v,w\in S\), there exists some vertex \(u\in S\), such that there exist paths in \(E\) from \(v\) to \(u\) and from \(w\) to \(u\). (The term downward directed appears in \cite{rangaswamy}; in other places it has been called the \emph{MT-3 condition}.)
\end{definition}

If \(R\) is an arbitrary commutative ring, we have the following necessary conditions for primeness. Let \(I\trianglelefteq L_R(E)\) be an ideal corresponding to the saturated function \(f\in \mathscr{F}_{E,R}\). (As \(C_u(E)=\varnothing\), all ideals are graded, so we can use Corollary \ref{classification of graded ideals}.)

\begin{lemma}\label{prime condition 1}
If \(I\) is prime, then \(f\) takes values in \(\operatorname{Spec} (R) \cup \{R\}\).
\end{lemma}
\begin{proof}
Suppose for contradiction that we have a hereditary saturated subset \(H\subseteq E^0\), such that \(f(H)=J\), and that we have two ideals \(A,B\trianglelefteq R\) with \(AB\subseteq J \trianglelefteq R\), but \(A,B\nsubseteq J\).
Consider the two ideals \(I_A\) and \(I_B\) of \(L_R(E)\) corresponding to
\[f_A: H' \mapsto \begin{cases}
A & \text{ if } H'\subseteq H\\
0 & \text{ else }
\end{cases}
\qquad \text{ and } \qquad
f_B: H' \mapsto \begin{cases}
B & \text{ if } H'\subseteq H\\
0 & \text{ else }
\end{cases}\]
respectively. Then we have \(I_A\cdot I_B\subseteq I\), but \(I_A,I_B\nsubseteq I\), so \(I\) is not prime.
\end{proof}

\begin{lemma}\label{prime condition 2}
If \(I\) is prime, then for any ideal \(J\trianglelefteq R\), the set \(H:=\max\left\{H' \in \mathscr{T}_E \mid J \subseteq f(H')\right\}\) (which exists as \(f\) is saturated) has the property that \(E^0\setminus H\) is downward directed.
\end{lemma}
\begin{proof}
By contradiction, assume we have such an ideal \(J\trianglelefteq R\), and vertices \(v,w\in E^0\setminus H\) witnessing the fact that \(E^0\setminus H\) is not downward directed. Consider the two ideals \(I_1, I_2\) of \(L_R(E)\) corresponding to 
\[f_1: H' \mapsto \begin{cases}
J & \text{ if } H'\subseteq \overline{\{v\}}\\
0 & \text{ else }
\end{cases}
\qquad \text{ and } \qquad
f_2: H' \mapsto \begin{cases}
J & \text{ if } H'\subseteq \overline{\{w\}}\\
0 & \text{ else }
\end{cases}\]
respectively. Then \(I_1,I_2\nsubseteq I\). However, the vertices of any hereditary saturated subset \(H'\) for which \(f_1(H') \cdot f_2(H')\neq 0\) will lie in \(H\), by the way we chose \(v\) and \(w\). For those subsets, we have \(f_1(H')\cdot f_2(H') = J^2 \subseteq f(H')\), so \(f_1\cdot f_2\leq f\), and hence \(\overline{f_1\cdot f_2} \leq f\). This shows that \(I_1\cdot I_2\subseteq I\), so \(I\) is not prime.
\end{proof}

Now consider the  graph $E$ with $E^0 = \{v,w \}$, $E^1 = \{e_1, e_2, f\}$, $r(e_i) = s(e_i) = s(f) = v$, and $r(f) = w$. Let us work over the ring \(\ZZ\), and consider the ideal $I$ of \(L_R(E)\) corresponding to
\[\begin{tikzcd}
(0) \arrow[d] \arrow[loop right] \arrow[loop left] \\
(2)
\end{tikzcd}\]
using Notation \ref{drawing functions}. This ideal clearly satisfies the conclusions of Lemmas \ref{prime condition 1} and \ref{prime condition 2}. Let $I_1$ and $I_2$ be the ideals of \(L_R(E)\) corresponding to
\begin{align*}\begin{tikzcd}
(0) \arrow[d] \arrow[loop right] \arrow[loop left] \\
\ZZ
\end{tikzcd}
&& \text{and} &&
\begin{tikzcd}
(4) \arrow[d] \arrow[loop right] \arrow[loop left] \\
(2)
\end{tikzcd}\end{align*}
respectively. Then $I_1 I_2 = I$ but neither $I_1$ nor $I_2$ are contained in $I$, so $I$ is not prime. This shows that these two conditions are not sufficient for an ideal to be prime. Note that for this counterexample to work, we needed to find a prime ideal contained in a different prime ideal. So we could have done this only with rings of Krull dimension at least 1. This suggests that some sufficient condition for primeness must also take into account the structure of the prime ideals of~\(R\).

\section*{Acknowledgements}

We would like to thank Gene Abrams and Kulumani Rangaswamy for their careful reading of the paper, encouragement, and many helpful suggestions.

\end{document}